\documentclass[11pt]{amsart} 
\usepackage[lmargin=1in,rmargin=1in,tmargin=1in,bmargin=1in]{geometry}
\usepackage[ps,all,arc,rotate]{xy}
\usepackage{graphicx, float, epstopdf}
\usepackage{color}
\usepackage[unicode,bookmarks=false]{hyperref}
\usepackage{centernot}
\usepackage{fancyhdr}
\usepackage{multirow}
\usepackage[utf8]{inputenc}
\usepackage{amsfonts,amssymb,amsmath,amsthm,mathrsfs}
\usepackage{graphics, setspace}
\usepackage{braket}
\usepackage{mathtools}
\usepackage{tikz}  
\usepackage{upgreek}
\usepackage{xcolor}
\usepackage{array,esint}

\numberwithin{equation}{section}
\numberwithin{figure}{section}
\allowdisplaybreaks[4]         
 
\newtheorem{theorem}{Theorem}[section]
\newtheorem{lemma}[theorem]{Lemma}
\newtheorem{proposition}[theorem]{Proposition}

\newtheorem{corollary}[theorem]{Corollary}
\theoremstyle{definition}
\newtheorem{definition}[theorem]{Definition}
\newtheorem{remark}[theorem]{Remark}
\usepackage{booktabs}
\usepackage{array}
\usepackage{vcell}
\newcommand{\Z}{\mathbb{Z}}

\newcommand{\R}{\mathbb{R}}
\newcommand{\C}{\mathbb{C}}
\newcommand{\N}{\mathbb{N}}

\newcommand{\ID}{\mathbf{1}}

\newcommand{\e}{\operatorname{e}}

\newcommand{\real}{\operatorname{Re}}

\newcommand{\quadand}{\quad \textnormal{and} \quad}

\newcommand{\norm}[1]{\left\lVert#1\right\rVert}

\newcommand{\g}{\mathfrak{g}}
\newcommand{\rad}{\operatorname{rad}}
\newcommand{\A}{\mathcal{A}}
\newcommand{\fM}{\mathfrak{M}}
\newcommand{\fm}{\mathfrak{m}}
\newcommand{\da}{d_{a/b}}
\newcommand{\dma}{d_{-a/b}}

\begin{document}

\title[Heath-Brown identities for fractional powers of $\zeta$]{Heath-Brown identities for fractional powers of $\zeta$}
\author[N.~Robles]{Nicolas Robles}
\address{RAND Corporation, Engineering and Applied Sciences,
Arlington, VA}
\email{robles.nicolas.m@gmail.com}

\begin{abstract}
We construct finite Heath-Brown-type identities for the fractional
powers $\zeta(s)^{\pm a/b}$ of the Riemann zeta-function, for every reduced fraction $a/b$ with
$0 < a/b < 1$, from Newton's binomial series in the algebra of
arithmetic functions, and we use them to prove the
Vinogradov-quality bound
$\sum_{n \le x} d_{\pm a/b}(n)\e(n\alpha) \ll_{a,b}
( x q^{-1/2} + x^{4/5} + x^{1/2} q^{1/2} ) (\log 2x)^{C}$, for some constant $C=C(a,b)>0$, whenever
$|\alpha - r/q| \le 1/q^2$ with $(r, q) = 1$. The bound carries no
$x^{\varepsilon}$ loss, and the same machinery gives the endpoint
case of the M\"obius function with an absolute constant. As
applications we determine the major-arc expansion of
$S_z(x, \alpha) = \sum_{n \le x} d_z(n)\e(n\alpha)$ to arbitrary
logarithmic precision for real $0<|z|<1$. For rational
$z \in (-1,1)$, we determine the order of magnitude of
$\sup_{\alpha} |S_z(x, \alpha)|$ and prove an asymptotic formula for the moments $\int_0^1 |S_z(x, \alpha)|^{s}\, d\alpha$ for every fixed real $s > 2$. The minor-arc analysis avoids the theory of $L$-functions entirely, and all constants are effective except those inherited from the Siegel-Walfisz theorem.
\end{abstract}

\maketitle

\section{Introduction}\label{sec:intro}

\subsection{Motivation and previous results}\label{subsec:motivation}

For $z \in \C$ let $d_z$ denote the multiplicative function defined on
prime powers by
\begin{equation}\label{eq:dz-def}
d_z(p^k) = \binom{-z}{k}(-1)^k
= \frac{z(z+1)\cdots(z+k-1)}{k!} \qquad (k \ge 0).
\end{equation}
Here and throughout, $p$ denotes a prime and $k$ a nonnegative
integer. Since $\sum_{k \ge 0} d_z(p^k) t^k = (1-t)^{-z}$ for $|t| < 1$, the
function $d_z$ is the sequence of Dirichlet coefficients of the
$z$-th power of the Riemann zeta-function:
\[
\sum_{n = 1}^{\infty} \frac{d_z(n)}{n^{s}} = \zeta(s)^{z}
\qquad (\real s > 1).
\]
This is the notation of \cite[\S 17.2.1, Exercise~9]{MV}. The family
$d_z$ interpolates the classical multiplicative functions: $d_1 = \ID$
is the constant function $1$, $d_k = \tau_k$ is the $k$-fold divisor
function for every integer $k \ge 1$, and $d_{-1} = \mu$ is the
M\"obius function. The subject of this paper is the exponential sum
\[
S_z(x, \alpha) := \sum_{n \le x} d_z(n)\e(n\alpha),
\qquad \e(\theta) := e^{2\pi i \theta},
\]
for \emph{fractional} values of $z$. Throughout this introduction,
$r/q$ denotes a reduced fraction with $|\alpha - r/q| \le 1/q^2$.

Let us first recall the standard estimates available at the integer points of the family. For the von Mangoldt function,
Vinogradov's estimate (\cite[Theorem~17.1]{MV},
\cite[Theorem~23.8]{K}) states that
\begin{equation}\label{eq:vinogradov}
\sum_{n \le x} \Lambda(n)\e(n\alpha)
\ll \Big( \frac{x}{\sqrt q} + x^{4/5} + \sqrt{xq} \Big)
(\log x)^{5/2}.
\end{equation}
For the M\"obius function, which is the point $z = -1$ of our family,
one has (\cite[\S 17.2.1, Exercise~1, eq.~(17.33)]{MV}; see also
\cite[Theorem~1.4]{BRZ} and
\cite[Exercise~23.4(b)]{K})
\begin{equation}\label{eq:mobius-eps}
\sum_{n \le x} \mu(n)\e(n\alpha)
\ll_{\varepsilon} \Big( \frac{x}{\sqrt q} + \sqrt{xq} \Big)
(\log x)^{3} \;+\; x^{4/5 + \varepsilon}.
\end{equation}
Here $\varepsilon$ denotes an arbitrarily small positive number, and
the implied constant depends on it. In \cite{BRZ} this estimate
drives an application to partitions.
In both \eqref{eq:vinogradov} and \eqref{eq:mobius-eps}
the exponent $4/5$ has the same source. A combinatorial identity --
Vaughan's identity \cite{Va} (see \cite[eq.~(17.5)]{MV},
\cite[Lemma~23.1]{K}), or Heath-Brown's identity \cite{HB} --
decomposes the weight into
finitely many convolutions in which every factor is either
\emph{short}, meaning supported on a small interval $[1, v]$, or \emph{long and smooth}: in general a fixed smooth weight such as
$\ID$ or a power of $\log$ (\cite[Theorem~23.5]{K} is stated for
$f * \log^{v}$); in the identities of this paper, every long factor
is a convolution power of $\ID$ itself. The type~I/type~II method of \cite[Chapter~17]{MV},
\cite[Chapter~23]{K} then converts this structure into cancellation.

For the higher divisor functions $\tau_k = d_k$ with $k \ge 2$, the
same shape holds, and in fact no identity is required to reach it.
The factorization $\tau_k = \ID^{*k}$ already exhibits the weight as
a convolution in which every factor is smooth, and a dyadic
decomposition of the factors feeds the type~I/type~II machinery
directly. Carrying this out, Pandey \cite[Proposition~4]{Pa} proved that for every fixed integer $k \ge 2$,
\begin{equation}\label{eq:pandey}
\sum_{n \le x} \tau_k(n)\e(n\alpha)
\ll_k \Big( \frac{x}{\sqrt q} + x^{4/5} + \sqrt{qx} \Big)
(\log x)^{O(1)},
\end{equation}
which is precisely the shape \eqref{eq:vinogradov}, uniformly across
the integer column of the family. This bound is what powers the minor
arcs in his asymptotic evaluation of the moments
$\int_0^1 |\sum_{n \le x} \tau_k(n) \e(n\alpha)|^s\, d\alpha$ for
real $s > 2$.

The fractional functions $d_z$, $z \notin \Z$, sit in a genuine gap
between these situations. There is no factorization of $d_z$ into
smooth pieces, so the elementary route to \eqref{eq:pandey} is
closed; and none of the classical identities applies as stated. To
our knowledge, the one fractional case treated in the two source chapters just cited is the
object of \cite[\S 17.2.1, Exercise~9]{MV}, which is devoted to
$d_{1/2}$. Part~(c) of that exercise proposes a decomposition of
$d_{1/2}$ modelled on Vaughan's identity, and part~(d) derives from
it the bound
\begin{equation}\label{eq:six-sevenths}
\sum_{n \le x} d_{1/2}(n)\e(n\alpha)
\ll \Big( \frac{x}{\sqrt q} + x^{6/7} + \sqrt{xq} \Big)
(\log x)^{3}.
\end{equation}

Comparing \eqref{eq:six-sevenths} with \eqref{eq:vinogradov},
\eqref{eq:mobius-eps} and \eqref{eq:pandey}, two questions
arise. First, does an estimate of this shape hold for every
rational exponent $z = \pm a/b$ in the \emph{closed} interval
$[-1, 1]$, and not only for $z = 1/2$? Second, can the middle
exponent $6/7$ be pushed down to the classical Type I/Type II benchmark $4/5$?

Three pieces of evidence suggested that the answer to both questions
should be yes. The first came from looking again at the exercise
itself. The decomposition of \cite[\S 17.2.1, Exercise~9(c)]{MV} is
not exploited in full in part~(d): the triple-convolution piece --
the analogue of the term $S_2$ in the standard notation for
Vaughan's identity -- is estimated there only as a type~I sum. A
routine adaptation, treating that piece bilinearly with the two
cutoffs re-balanced at $u = v = x^{1/3}$, lowers the middle exponent
from $6/7$ to $5/6$, still with purely logarithmic losses (the
coefficients involved are divisor-bounded, so their mean squares
cost only powers of $\log x$). The final proof in this
paper supersedes that intermediate estimate, but it was the first
indication that the identity of the exercise had not been pushed to
its limit. The second piece of evidence is the behaviour at the
neighbouring exponents: the integer column sits at $4/5$ by \eqref{eq:pandey},
and the endpoint $\mu = d_{-1}$ sits at $4/5$ up to $x^\varepsilon$
by \eqref{eq:mobius-eps}. It would be strange, indeed unsatisfactory, for the fractional
interior of the family to behave worse than both boundaries.

The third piece of evidence, of a different kind, came from a
companion program on the additive side. In \cite{GR}, Gafni and the
author studied $S_f(\alpha; X) = \sum_{n \le X} f(n) \e(n\alpha)$
for the class $\mathcal F_0$ of \emph{additive} functions $f$ with
$f(p) = 1$ at every prime -- the prototypes being $\omega$, the
number of distinct prime factors, and $\Omega$, the number of prime
factors with multiplicity -- and proved, for every
$\Delta \in (0, \tfrac12)$,
\begin{equation}\label{eq:GR}
S_{\omega}(\alpha; X)
\ll \Big( \frac{X}{q^{\Delta}} + X^{5/6}
+ X^{1-\Delta} q^{\Delta} \Big) (\log X)^{4},
\end{equation}
with an analogous bound across the class; the estimate was then
converted, through the circle method, into the asymptotic solution of
the ternary Goldbach-Vinogradov problem for $\Omega$. The
concluding section of \cite{GR} raises the problem of bounding the
twisted sums $\sum_{n \le X} z^{f(n)} \e(n\alpha)$ and identifies
the generalized divisor function $d_z$ as the closely related central
object for which a bound would be `very interesting' to derive. The
present paper takes up precisely that problem, for the function $d_z$
itself; it is a pleasant coincidence that the additive class of
\cite{GR} rests, for now, at the same middle exponent $5/6$ as the
intermediate multiplicative estimate described above.

The device that realizes this hope is an identity of Heath-Brown
type for $\zeta^{\pm a/b}$. Heath-Brown's identity for $\mu$ and
$\Lambda$ (\cite{HB}; see \cite[\S 17.1.1, Exercise~6]{MV}, \cite[Exercise~23.5]{K}, \cite[Proposition~13.3]{IK}) may be viewed as the depth-$K$ member of a
family whose depth-$2$ member is Vaughan's identity \cite{Va}.
It is instructive to compare \cite[Theorem~13.9]{IK}, which is
likewise free of the $x^{\varepsilon}$ loss but carries the middle
term $x^{9/10}$ in place of $x^{4/5}$. The middle exponent records
the lower edge of the range in which the bilinear (type~II)
estimate is applied. The bilinear estimate saves the square root of
the shorter side of the factorization, so the worst case sits at the
lower edge of the admissible window: an edge $x^{\vartheta}$
contributes
$x \cdot x^{-\vartheta/2}$, so an edge $x^{1/5}$ yields $x^{9/10}$,
while the full window $x^{2/5} < M \le x^{3/5}$, when it is available,
yields $x^{4/5}$; when it is not available, the same combinatorial lemma
forces a long smooth variable and the type~I alternative applies.
The exponent $4/5$ also governs the exponential sum over the
prime-detecting function in the partition setting \cite{Va2}. A recent announcement of Maynard, Pandey and Radziwi{\l}{\l},
reported in Radziwi{\l}{\l}'s CANT 2026 talk \cite{MPR}, improves this
exponent to $19/24$ in the corresponding estimate for exponential sums
over primes, using large-value estimates. A
complementary direction is taken in \cite{DRZZ, DRZZ2}, where exponential
sums twisted by a broad class of arithmetic functions --
multiplicative, additive, or neither -- are bounded within a
common framework; the generality there is wider than ours, and
correspondingly the estimates for any single weight are not of
Vinogradov quality, which is the trade-off our restriction to the
family $d_{\pm a/b}$ is designed to avoid. Increasing the
depth produces more convolution pieces, but it shortens the short
factors to length $x^{1/K}$; and at depth $K = 5$ every dyadic
configuration of variables falls either to the bilinear (type~II)
estimate or to the linear (type~I) estimate. This is the numerology
behind the exponent $4/5$ in \eqref{eq:vinogradov}, and the feature
of the identity that makes it work is the one emphasized in
\cite[p.~240]{K}: all of the long functions appearing in it are
smooth.

For an integer power of $\zeta$ the identity terminates for
algebraic reasons, the relevant binomial expansion being finite. For
a fractional power $\zeta^{-a/b}$ the natural generating identity is
Newton's binomial series
\[
(1 - t)^{-1/b} = \sum_{k \ge 0} \binom{-1/b}{k} (-t)^{k},
\]
which does \emph{not} terminate. The truncation at depth $K$ must
instead be earned by a support argument. If $F$ denotes the partial
sum of $\zeta^{-a/b}$ up to $V$, then the remainder
\[
H = 1 - \zeta^{a} F^{b}
\]
has Dirichlet coefficients supported on the integers exceeding $V$.
Consequently $H^{K}$ is supported beyond $V^{K}$, and it contributes
nothing below $x$ as soon as $V^{K} \ge x$. Carrying this out, first
for $d_{\pm 1/2}$ and then for general $d_{\pm a/b}$, occupies
Section~\ref{sec:identity}. The resulting identities have explicit
rational coefficients -- for instance $\tfrac{315}{128}$,
$-\tfrac{105}{32}$, $\tfrac{189}{64}$, $-\tfrac{45}{32}$,
$\tfrac{35}{128}$ in the case $z = -1/2$, $K = 5$ -- and every long
factor in them is a positive power of $\zeta$, hence smooth.

\subsection{Our new results}\label{subsec:results}

For the minor-arc results -- Theorem~\ref{thm:main} and
Corollary~\ref{cor:minor}, and the identities of
Section~\ref{sec:identity} behind them -- we fix integers $a, b$
with
\begin{equation}\label{eq:ab-cond}
1 \le a < b, \qquad (a, b) = 1,
\end{equation}
so that the exponents $z = \pm a/b$ range over the nonzero rationals
of the open interval $(-1, 1)$; the endpoints $z = \pm 1$ are
discussed after the statements. The major-arc statements below
(Theorems~\ref{thm:local} and~\ref{thm:major}) hold for every real
$z$ with $0 < |z| < 1$, while Theorems~\ref{thm:sup}
and~\ref{thm:scalar} return to rational $z$. Our main result is the following.

\begin{theorem}\label{thm:main}
Fix $a, b$ satisfying \eqref{eq:ab-cond}. There is a constant $C = C(a,b) > 0$ such that
for every $x \ge 3$, every $\alpha \in \R$, and every reduced
fraction $r/q$ with $|\alpha - r/q| \le 1/q^2$,
\begin{equation}\label{eq:main}
\sum_{n \le x} d_{\pm a/b}(n)\e(n\alpha)
\ll_{a,b}
\Big( \frac{x}{\sqrt{q}} + x^{4/5} + \sqrt{xq} \Big)
(\log 2x)^{C}.
\end{equation}
\end{theorem}

Estimate \eqref{eq:main} is \eqref{eq:vinogradov} with $\Lambda$
replaced by $d_{\pm a/b}$: all three terms carry only a power of
$\log x$, and in particular there is no $x^{\varepsilon}$ anywhere in
the statement. A word on where the losses live inside the proof may
be useful. A pointwise divisor bound of size $x^{o(1)}$ is used at
exactly one place, in the type~I analysis of the middle range
$x^{3/10} < q < x^{7/10}$; there the type~I support is at most
$x^{7/10}$ up to bounded factors, and the gap between the exponents
$7/10$ and $4/5$ absorbs the loss into the middle term of
\eqref{eq:main} before it can reach the statement. Outside that
range, the type~I sums are handled by a mean-square estimate that
loses only logarithms. This two-range organization is not a
refinement but a necessity: the tempting alternative -- a pointwise,
divisor-weighted type~I estimate with logarithmic losses, uniform in
$q$ -- is false in general, and the proof is arranged so that no
statement of that kind is ever used.

The statement actually consumed by the circle method is the
following corollary, in which each of the three terms of the minor-arc estimate is
$O_{a,b,A}\big( x (\log x)^{-A} \big)$ throughout the stated range of
denominators.

\begin{corollary}\label{cor:minor}
Fix $a, b$ satisfying \eqref{eq:ab-cond} and let $C = C(a,b)$ be as
in Theorem~\ref{thm:main}. Let $A > 0$. If $x \ge x_0(a,b,A)$, if
$\alpha \in \R$ admits a reduced approximation $r/q$ with
$|\alpha - r/q| \le 1/q^2$, and if
\[
(\log x)^{2A + 2C} \le q \le \frac{x}{(\log x)^{2A + 2C}},
\]
then
\[
\sum_{n \le x} d_{\pm a/b}(n)\e(n\alpha)
\ll_{a,b,A} \frac{x}{(\log x)^{A}}.
\]
\end{corollary}

The hypothesis $a < b$ leaves out the endpoints $z = \pm 1$ of the
closed interval. At $z = 1$ the sum $S_1(x, \alpha)$ is a geometric
series and satisfies the elementary bound
$\min(x, \norm{\alpha}^{-1})$, far stronger than \eqref{eq:main}; and
$z = 0$ gives $d_0 = \delta$, for which the sum is trivial. The
endpoint $z = -1$ is the M\"obius function, and it satisfies the
same estimate.

\begin{theorem}\label{thm:mu}
There is an absolute constant $C_0 > 0$ such that for all $x \ge 3$, all $\alpha \in \R$, and every reduced
fraction $r/q$ with $|\alpha - r/q| \le 1/q^2$,
\[
\sum_{n \le x} \mu(n)\e(n\alpha)
\ll
\Big( \frac{x}{\sqrt q} + x^{4/5} + \sqrt{xq} \Big) (\log x)^{C_0} .
\]
\end{theorem}

The constant $C_0$ is effective and could be extracted from the
proof; we have
made no attempt to optimize it. For $q$ in the analogous range, with $C_0$ in
place of $C$, the argument of Corollary~\ref{cor:minor} applied to
Theorem~\ref{thm:mu} gives
$\sum_{n \le x} \mu(n)\e(n\alpha) \ll_A x (\log x)^{-A}$. This
is the minor-arc half of Davenport's classical theorem, in which
the same bound holds for \emph{all} $\alpha$ once the major arcs
are supplied with the Siegel-Walfisz theorem for $\mu$; a
textbook treatment is \cite[Exercise~23.4(c)]{K}.

We emphasize that Theorem~\ref{thm:mu} is not a corollary of
Theorem~\ref{thm:main}: the hypothesis $a < b$ excludes the endpoint.
What happens at $z = -1$ is that the identity of
Section~\ref{sec:identity} degenerates into Heath-Brown's classical
identity for $\mu$ \cite{HB}, and the estimation machinery of
Section~\ref{sec:main-proof} applies to it verbatim, the short
coefficients being again bounded by $1$. We state the case separately
because it refines two known estimates: it removes the
$\varepsilon$ from \eqref{eq:mobius-eps} and from
\cite[Theorem~1.4]{BRZ}. Taken together with the elementary endpoint $z = 1$, where
$d_1 = \ID$ and the geometric series gives the same shape,
Theorems~\ref{thm:main} and~\ref{thm:mu} answer the two questions of
Section~\ref{subsec:motivation}: the Vinogradov shape, with middle
exponent $4/5$ and purely logarithmic losses, holds at every nonzero
rational $z$ of the closed interval $[-1, 1]$.

Two remarks on the scope of the statements. First, the rationality
of the exponent is essential to our method: the identity of
Section~\ref{sec:identity} realizes $\zeta^{-a/b}$ as a $b$-th root
of an integer power of $\zeta$, and no analogous algebraic handle is available within this construction for irrational $z$. Second, the restriction $|z| \le 1$ is
what keeps the short coefficients bounded by $1$; for exponents
outside $[-1, 1]$ they need not be, and although the arguments could
be enlarged to allow short coefficients bounded by a fixed divisor
function, at the cost of changing only logarithmic powers, we do not
pursue that generalization here.

The proof of Theorem~\ref{thm:main} has two independent layers. The
first is the identity itself (Section~\ref{sec:identity}): an exact
combinatorial decomposition of $d_{\pm a/b}$, valid on the integers
up to $V^5$, into five convolutions in which every factor is either
supported on $[1, V]$ and bounded by $1$, or a positive convolution
power of the constant function $\ID$. The second is the estimation
layer (Section~\ref{sec:main-proof}), which applies the identity
with $V = x^{1/5}$: a dyadic decomposition; a bilinear (type~II) estimate
whenever some subproduct of the variables can be grouped into the
window $[x^{2/5}, x^{3/5}]$; and, in the remaining configurations,
where a long smooth variable is forced to exist, the two-range
type~I treatment described above. Theorem~\ref{thm:mu} is obtained
by feeding the classical identity for $\mu$ into the same machinery.

\subsection{Applications}\label{subsec:applications}

With Corollary~\ref{cor:minor} in hand, the natural next step is the
one taken for $\Lambda$ in \cite[Chapter~17]{MV} and
\cite[Chapter~23]{K}: a matching major-arc analysis, and then
consequences via the circle method. We give three applications,
proved in Section~\ref{sec:applications}.

On the major arcs the arithmetic is carried by a family of local
factors, one for each modulus $q$, which play for $d_z$ the role that
$\mu(q)/\varphi(q)$ plays for the primes (visible in
\cite[eq.~(23.19)]{K}). Since these factors appear in the main terms
of all three applications, we introduce them first. Recall
Ramanujan's sum
\begin{equation}\label{eq:ramanujan-def}
c_q(m) = \sum_{\substack{1 \le a \le q \\ (a, q) = 1}}
\e\Big( \frac{am}{q} \Big),
\end{equation}
and write $\rad(a)$ for the radical of $a$, the product of the
distinct primes dividing $a$.

\begin{theorem}
\label{thm:local}
Fix a real number $z$ with $0 < |z| < 1$. For $q \ge 1$ set
\begin{equation}\label{eq:g-def}
\g_z(q) := \Big( \prod_{p \mid q} \Big( 1 - \frac1p \Big)^{z} \Big)
\cdot \frac{1}{\varphi(q)} \sum_{\substack{a \ge 1 \\ \rad(a) \mid q}}
\frac{d_z(a)\, c_q(a)}{a},
\end{equation}
where the series converges absolutely. Then $\g_z$ is multiplicative,
$\g_z(1) = 1$,
\begin{equation}\label{eq:gamma-p}
\g_z(p) = 1 - \Big( 1 - \frac1p \Big)^{z - 1},
\end{equation}
and $\g_z(q) \ll_{z, \varepsilon} q^{-1 + \varepsilon}$ for every
$\varepsilon > 0$.
\end{theorem}

The definition \eqref{eq:g-def} is the expression that the
computation on the arc produces; the content of the theorem is that
this expression is multiplicative and small, and takes the displayed
elementary value at the primes. The first application is a major-arc
expansion in which $\g_z$ carries, at leading order, the entire
dependence on the modulus.

\begin{theorem}
\label{thm:major}
Fix a real number $z$ with $0 < |z| < 1$, a real number $B > 0$, and
an integer $J \ge 1$. There exist complex numbers
$\lambda_j(z, q)$, $0 \le j < J$, depending only on $z, j, q$, with
\begin{equation}\label{eq:lambda-props}
\lambda_0(z, q) = \frac{\g_z(q)}{\Gamma(z)}
\qquad \text{and} \qquad
\lambda_j(z, q) \ll_{z, J, \varepsilon} q^{-1+\varepsilon}
\quad (0 \le j < J,\ \varepsilon > 0),
\end{equation}
where $\g_z$ is the function of Theorem~\ref{thm:local}, such that
the following holds. For all $x \ge 3$, every reduced
fraction $r/q$ with $q \le (\log x)^{B}$, and every real $\beta$ with
$|\beta| \le (\log x)^{B}/x$,
\begin{equation}\label{eq:major}
\sum_{n \le x} d_z(n) \e\Big( n \Big( \frac{r}{q} + \beta \Big) \Big)
=
\sum_{j=0}^{J-1} \lambda_j(z, q)
\int_{2}^{x} (\log t)^{z - 1 - j}\e(\beta t)\, dt
\;+\;
O_{z, B, J}\Big( \frac{x}{(\log x)^{J + 1 - z - 2B}} \Big).
\end{equation}
The implied constant is ineffective.
\end{theorem}

The coefficients $\lambda_j(z, q)$ are explicit Selberg-Delange
data: they are assembled at \eqref{eq:lambda-def} from the Taylor
coefficients at $s = 1$ of the associated local generating function,
normalized by reciprocal Gamma factors as in \cite[Ch.~II.5]{Te};
see Propositions~\ref{prop:LSD} and~\ref{prop:beta0}. In particular
\eqref{eq:major} holds for every fixed $J$ -- this is the arbitrary
logarithmic precision of the abstract.

Given any $A > 0$, the choice $J = \lceil A + 1 + 2B \rceil$ makes
the error term in \eqref{eq:major} at most $O(x (\log x)^{-A})$. Note
that the main term of \eqref{eq:major} does not depend on $r$; this
is the analogue of the fact, visible in \cite[eq.~(23.19)]{K}, that
the main term of $\sum_{p \le x} \e(rp/q)$ is
$\mu(q) \operatorname{li}(x)/\varphi(q)$, independently of $r$. The ineffectivity enters through Siegel's theorem, used in the proof
of Proposition~\ref{prop:SW} in Appendix~\ref{app:SW} -- the same
source of ineffectivity as in the classical Siegel-Walfisz theorem
--  and is inherited by every statement resting on
Theorem~\ref{thm:major}.

The second application combines the minor arcs and the major arcs
into uniform statements over all $\alpha \in \R$, with no Diophantine
hypothesis: by Dirichlet's approximation theorem every $\alpha$ is
covered by one of the two regimes.

\begin{sloppypar}
For completely multiplicative $f$ with $|f| \le 1$, Montgomery
and Vaughan \cite{MV77} obtained the sharpest general bounds for
such sums; de la
Bret\`eche and Granville \cite{BG} showed that
$\sum_{n \le x} f(n)\e(n\alpha)$ can be large only when $f$ is
pretentious (see also the minor-arc counterpart \cite{GL}).
Theorem~\ref{thm:sup} complements this structural principle with
the exact order of the supremum for the multiplicative family
$d_z$.
\end{sloppypar}

\begin{theorem}
\label{thm:sup}
Fix $a, b$ satisfying \eqref{eq:ab-cond}.
\begin{enumerate}
\item[(i)] For every $x \ge 1$,
\[
\sup_{\alpha \in \R}\,
\Big| \sum_{n \le x} d_{a/b}(n)\e(n\alpha) \Big|
= \sum_{n \le x} d_{a/b}(n),
\]
the supremum being attained at $\alpha = 0$; consequently, by the
Selberg-Delange asymptotic for the partial sums
(Proposition~\ref{prop:LSD} at $q = 1$), for $x \ge 3$,
\[
\sup_{\alpha \in \R}\,
\Big| \sum_{n \le x} d_{a/b}(n)\e(n\alpha) \Big|
= \frac{x\, (\log x)^{a/b - 1}}{\Gamma(a/b)}
\Big( 1 + O_{a,b}\Big( \frac{1}{\log x} \Big) \Big),
\]
with an effective implied constant. In particular the supremum for
$d_{1/2}$ is asymptotic to $x / \sqrt{\pi \log x}$.
\item[(ii)] For all $x \ge 3$,
\[
\sup_{\alpha \in \R}\,
\Big| \sum_{n \le x} d_{-a/b}(n)\e(n\alpha) \Big|
\asymp_{a, b} \frac{x}{(\log x)^{1 + a/b}}
\]
(for $z = -\tfrac12$ the order is $x (\log x)^{-3/2}$); the
implied constant in the upper bound is ineffective, while that in
the lower bound is effective.
\end{enumerate}
\end{theorem}

The contrast between the two parts is worth noting: for the positive
exponents the supremum is attained at the trivial frequency and the
extremal size is that of the mean value, while for the negative
exponents the sign changes of $d_{-a/b}$ force cancellation at
\emph{every} frequency, and the supremum drops below the $\ell^1$
norm by a factor of a full power of $\log x$.

The third application is the fractional analogue of Pandey's moment
theorem \cite{Pa} for $\tau_k$. For the remainder of this subsection and in
Section~\ref{subsec:moments-proof} the length of the sum is
denoted $X$, and we abbreviate $S_z(\alpha) := S_z(X, \alpha)$.

\begin{theorem}
\label{thm:scalar}
Fix a rational $z$ with $0 < |z| < 1$ and a real $s > 2$. Then, as
$X \to \infty$,
\[
\int_0^1 | S_z(\alpha) |^{s}\, d\alpha
=
\frac{A_s}{|\Gamma(z)|^{s}}
\Big( \sum_{q = 1}^{\infty} \varphi(q)\, |\g_z(q)|^{s} \Big)
X^{s-1} (\log X)^{s(z-1)}
\Big( 1 + O_{s, z}\Big( \frac{1}{\log X} \Big) \Big),
\]
where
\[
A_s = \frac{2}{\pi} \int_0^{\infty}
\frac{|\sin t|^{s}}{t^{s}}\, dt,
\]
the series converges absolutely, and the leading constant is
positive. The implied constant is ineffective.
\end{theorem}

The restriction $s > 2$ is a genuine boundary of the method. At
$s = 2$, Parseval's identity gives
$\int_0^1 |S_z(x, \alpha)|^2\, d\alpha = \sum_{n \le x} d_z(n)^2
\sim c(z)\, x (\log x)^{z^2 - 1}$ with $c(z) > 0$, by the
Selberg-Delange method; the extrapolation of the theorem to $s = 2$
would predict $x (\log x)^{2(z-1)}$, and since
$z^2 - 1 = 2(z-1) + (z-1)^2$, the true second moment is larger by
the positive power $(z-1)^2$ of $\log x$. The major arcs cease to
dominate at $s = 2$, and correspondingly the singular series
$\mathfrak G_s(z)$ diverges as $s \downarrow 2$. The range
$0 < s < 2$ appears to require different methods; see
Section~\ref{sec:conclusion}.

Theorem~\ref{thm:scalar} is proved in
Section~\ref{sec:applications}. The
comparison with \cite{Pa} is instructive. On the minor arcs the two
situations are now on the same footing: Pandey's input is
\eqref{eq:pandey} and ours is Corollary~\ref{cor:minor}, both of
Vinogradov shape. The difference is on the major arcs: for $\tau_k$
these rest on Voronoi summation and are valid for $q$ up to a power
of $X$, which yields the power saving $\delta_{s,k}$ of \cite{Pa}
(a saving that degrades as $k$ grows through the Voronoi error
term); our major arcs, Theorem~\ref{thm:major}, hold in the
restricted range $q \le (\log X)^B$, which is why our error term
saves powers of $\log X$ rather than powers of $X$ -- but they hold
uniformly across the fractional family, where no comparable Voronoi-type formula is available to us.

The plan of the paper is as follows. Section~\ref{sec:tools} is a toolbox: all of its lemmas are proved in full, with two
exceptions quoted from \cite{K} with precise pointers, the bilinear
type~II estimate (Lemma~\ref{lem:typeII}) and Dirichlet's
approximation theorem (Lemma~\ref{lem:dirichlet}). The applications
in Section~\ref{sec:applications} additionally use two analytic
inputs: the Selberg-Delange method, quoted there with a
precise citation as Proposition~\ref{prop:LSD}, and a Siegel-Walfisz
theorem for $d_z \chi$, stated as Proposition~\ref{prop:SW} and
proved in Appendix~\ref{app:SW}. Section~\ref{sec:identity} constructs the Heath-Brown
identity for $\zeta^{\pm a/b}$: first the cases $z = \pm 1/2$ in explicit form, then the general
case, including the endpoint $\mu = d_{-1}$ and the
positive-exponent identity. Section~\ref{sec:main-proof} proves
Theorems~\ref{thm:main} and~\ref{thm:mu} and
Corollary~\ref{cor:minor}. Section~\ref{sec:applications} proves the
three applications. Section~\ref{sec:conclusion} closes with remarks.

\section{Preliminary lemmas}\label{sec:tools}

This section collects the elementary tools used later, in the order
in which they are first needed.
Lemmas~\ref{lem:tau-basic}--\ref{lem:minsq} serve the minor-arc
argument: they contain everything used in the proofs of
Theorems~\ref{thm:main} and~\ref{thm:mu} in Sections
\ref{sec:identity}--\ref{sec:main-proof}, with one exception among them, the bilinear type~II estimate
(Lemma~\ref{lem:typeII}), quoted from \cite{K} with a precise pointer
in place of the proof.
Lemmas~\ref{lem:ramanujan}--\ref{lem:singular-series} are additional
tools for the applications of Section~\ref{sec:applications}; where
a statement uses the major-arc notation introduced there, this is
said explicitly. Several of the lemmas are standard, or are close to
exercises in \cite{MV}, \cite{K}, \cite{IK}, \cite{Te}; proofs are
nevertheless included whenever our statements differ from the
textbook forms or the uniformity of the constants matters, in keeping with the self-contained standard of the paper. The two
lemmas quoted without proof, each with a precise pointer, are
Lemma~\ref{lem:typeII} (\cite[Theorem~23.6]{K}) and
Lemma~\ref{lem:dirichlet} (\cite[Lemma~23.4]{K}).

\subsection{Notation and conventions}\label{subsec:notation}

Throughout, $\log$ denotes the natural logarithm, and $x \ge 3$, so
that $\log\log x > 0$. The letter $p$ always denotes a prime.
Constants implied by $\ll$ and $O(\cdot)$ are absolute unless
dependence is indicated by a subscript; the integers $a, b$ of
Theorem~\ref{thm:main} are regarded as fixed, and from
Section~\ref{sec:identity} onwards all implied constants may depend
on $a$ and $b$ without further mention.

We write $\norm{\theta}$ for the distance from $\theta$ to the
nearest integer, as in \cite[p.~242]{K} and \cite[\S 17.2]{MV};
throughout, an expression $\min(A, 1/\norm{\theta})$ is to be read
as $A$ when $\norm{\theta} = 0$. The
Dirichlet convolution of two arithmetic functions $f, g$ is
$(f * g)(n) = \sum_{de = n} f(d)\, g(e)$, and $f^{*j}$ denotes the
$j$-fold convolution $f * \cdots * f$, with $f^{*0} := \delta$, where
$\delta(n) = \ID_{n = 1}$ is the identity for convolution. We write
$\ID$ for the constant function $\ID(n) = 1$. Following
\cite[eq.~(23.6)]{K}, for a parameter $V \ge 1$ we write
\[
f_{\le V}(n) := \ID_{n \le V} \cdot f(n).
\]
For an integer $m \ge 1$, $\tau_m = \ID^{*m}$ denotes the $m$-fold
divisor function, so that
\[
\tau_m(n) = \#\{ (n_1, \dots, n_m) \in \N^m : n_1 \cdots n_m = n \},
\]
$\tau_1 = \ID$ and $\tau_2 = \tau$. For an arithmetic function $f$ we
write $\|f\|_2^2 = \sum_{n \ge 1} |f(n)|^2$ for the $\ell^2$-norm.
Finally, $\A$ denotes the set of all arithmetic functions
$f : \N \to \C$, regarded as a commutative ring under pointwise
addition and Dirichlet convolution, with multiplicative identity
$\delta$.

\subsection{Divisor function estimates}\label{subsec:tau}

\begin{lemma}\label{lem:tau-basic}
For every $m \ge 1$, the function $\tau_m$ is multiplicative,
\begin{equation}\label{eq:tau-pp}
\tau_m(p^k) = \binom{k + m - 1}{m - 1},
\end{equation}
and
\begin{equation}\label{eq:tau-power}
\tau_m(n) \le \tau(n)^{m-1} \qquad (n \ge 1,\ m \ge 2).
\end{equation}
\end{lemma}

\begin{proof}
Multiplicativity: $\tau_m = \ID^{*m}$ is a convolution of
multiplicative functions, hence multiplicative. For
\eqref{eq:tau-pp}, $\tau_m(p^k)$ counts the tuples
$(k_1, \dots, k_m)$ of non-negative integers with
$k_1 + \cdots + k_m = k$, and the number of these is
$\binom{k+m-1}{m-1}$ (the classical ``stars and bars'' count: such
tuples are in bijection with the arrangements of $k$ unlabelled
stars and $m-1$ bars in a row, an arrangement being determined by
the choice of the $m-1$ bar positions among $k + m - 1$ slots).

For \eqref{eq:tau-power}, by multiplicativity it suffices to prove
$\tau_m(p^k) \le \tau(p^k)^{m-1} = (k+1)^{m-1}$, and we induct on
$m$. The case $m = 2$ is an identity. For $m \ge 3$, since
$\tau_m = \ID * \tau_{m-1}$,
\[
\tau_m(p^k) = \sum_{i = 0}^{k} \tau_{m-1}(p^i)
\le (k+1)\, \tau_{m-1}(p^k)
\le (k+1) \cdot (k+1)^{m-2},
\]
where the first inequality uses that $\tau_{m-1}(p^i)$ is
non-decreasing in $i$ (clear from \eqref{eq:tau-pp}), and the second
is the induction hypothesis.
\end{proof}

\begin{lemma}\label{lem:submult}
For every $m \ge 1$ and all integers $u, v \ge 1$,
\[
\tau_m(uv) \le \tau_m(u)\, \tau_m(v).
\]
\end{lemma}

\begin{proof}
We exhibit an injection from the set of $m$-tuples
$(n_1, \dots, n_m)$ with $n_1 \cdots n_m = uv$ into the set of pairs
of $m$-tuples $\big( (u_1, \dots, u_m), (v_1, \dots, v_m) \big)$
with $u_1 \cdots u_m = u$ and $v_1 \cdots v_m = v$.

Given $(n_1, \dots, n_m)$ with $\prod_i n_i = uv$, define
successively
\[
u_1 = \gcd(n_1, u), \qquad
u_i = \gcd\Big( n_i,\ \frac{u}{u_1 \cdots u_{i-1}} \Big)
\quad (2 \le i \le m),
\]
and set $v_i = n_i / u_i$. Each $u_i$ divides the previous remainder
$u/(u_1 \cdots u_{i-1})$, so $u_1 \cdots u_m \mid u$. We claim
$u_1 \cdots u_m = u$. It suffices to check this at each prime $p$:
write $\beta = v_p(u)$ and $\alpha_i = v_p(n_i)$, where $v_p$
denotes the $p$-adic valuation; then
$\sum_i \alpha_i = v_p(uv) \ge \beta$, and by construction
$v_p(u_i) = \min\big( \alpha_i,\ \beta - v_p(u_1 \cdots u_{i-1})
\big)$. If we had $v_p(u_1 \cdots u_m) < \beta$, then at every step
the minimum was attained by the first argument, i.e.\
$v_p(u_i) = \alpha_i$ for all $i$, whence
$v_p(u_1 \cdots u_m) = \sum_i \alpha_i \ge \beta$, a contradiction.
Hence $\prod_i u_i = u$, and consequently
$\prod_i v_i = \prod_i n_i / \prod_i u_i = v$; moreover each
$v_i = n_i/u_i$ is a positive integer since $u_i \mid n_i$.

The map $(n_1, \dots, n_m) \mapsto \big( (u_i)_i, (v_i)_i \big)$ is
injective because $n_i = u_i v_i$ is recovered from the image. This
proves the claimed inequality.
\end{proof}

\begin{lemma}\label{lem:taujk}
For all integers $j, k \ge 1$ and all $n \ge 1$,
\[
\tau_j(n)\, \tau_k(n) \le \tau_{jk}(n).
\]
\end{lemma}

\begin{proof}
Since $\tau_{jk} = (\tau_j)^{*k}$ (both sides equal $\ID^{*jk}$),
\[
\tau_{jk}(n)
= \sum_{m_1 \cdots m_k = n} \tau_j(m_1) \cdots \tau_j(m_k).
\]
By Lemma~\ref{lem:submult}, applied repeatedly,
$\tau_j(m_1) \cdots \tau_j(m_k) \ge \tau_j(m_1 \cdots m_k)
= \tau_j(n)$ for every tuple in the sum. Hence
\[
\tau_{jk}(n) \ge \sum_{m_1 \cdots m_k = n} \tau_j(n)
= \tau_j(n)\, \tau_k(n),
\]
since the number of tuples $(m_1,\dots,m_k)$ with
$m_1 \cdots m_k = n$ is $\tau_k(n)$.
\end{proof}

\begin{lemma}\label{lem:tau-max}
There is an absolute constant $C_1 > 0$ such that
\[
\tau(n) \le \exp\Big( C_1 \frac{\log n}{\log\log n} \Big)
\qquad (n \ge 3).
\]
Consequently, by \eqref{eq:tau-power} for $m \ge 2$, and trivially
for $m = 1$, since $\tau_1 \equiv 1$,
\[
\tau_m(n) \le \exp\Big( (m-1) C_1 \frac{\log n}{\log\log n} \Big)
\qquad (n \ge 3,\ m \ge 1).
\]
\end{lemma}

\begin{proof}
It suffices to prove the bound for $n \ge n_0$, where $n_0$ is an
absolute constant, since for $3 \le n < n_0$ the bound holds after
enlarging $C_1$ (the left side is bounded and the exponent on the
right is bounded away from $0$ on $[3, n_0]$).

Let $n = \prod_p p^{k_p}$, write $L = \log n$, and set
$\delta = 1/\log L$; assume $n$ is large enough that
$0 < \delta < 1/2$. Then
\[
\frac{\tau(n)}{n^{\delta}}
= \prod_{p \mid n} \frac{k_p + 1}{p^{k_p \delta}} .
\]
We estimate the factors in two ranges.

If $p \ge 2^{1/\delta}$, then
$p^{k_p \delta} \ge 2^{k_p} \ge k_p + 1$, so the factor is at most
$1$.

If $p < 2^{1/\delta}$, we use the crude bound
$k_p + 1 \le 1 + \log_2 n \le 3L$ (valid for $n \ge 2$, since
$2^{k_p} \le p^{k_p} \le n$), so each such factor is at most $3L$;
and the number of such primes is at most
$2^{1/\delta} = 2^{\log L} = L^{\log 2}$.

Combining,
\[
\tau(n) \le n^{\delta}\, (3L)^{L^{\log 2}}
= \exp\Big( \frac{L}{\log L} + L^{\log 2} \log(3L) \Big).
\]
Since $\log 2 < 0.7$, we have
$L^{\log 2}\log(3L) = o(L/\log L)$ as $L \to \infty$; in particular
$L^{\log 2}\log(3L) \le L/\log L$ for all $n \ge n_0$ with $n_0$
absolute. This gives the bound with $C_1 = 2$ for $n \ge n_0$, and
the general case follows as explained above.
\end{proof}

\begin{lemma}\label{lem:tau-mean}
For every fixed $m \ge 1$ and all $Y \ge 1$,
\[
\sum_{n \le Y} \tau_m(n) \ll_m Y (\log 2Y)^{m-1}.
\]
\end{lemma}

\begin{proof}
Induction on $m$. For $m = 1$ the left side is
$\lfloor Y \rfloor \le Y$. For $m \ge 2$, since
$\tau_m = \ID * \tau_{m-1}$,
\[
\sum_{n \le Y} \tau_m(n)
= \sum_{d \le Y} \sum_{e \le Y/d} \tau_{m-1}(e)
\ll_m \sum_{d \le Y} \frac{Y}{d}
\Big( \log \frac{2Y}{d}\Big)^{m-2}
\le Y (\log 2Y)^{m-2} \sum_{d \le Y} \frac{1}{d},
\]
by the induction hypothesis. Finally, by comparison with the
integral $\int_1^Y dt/t$ we have
$\sum_{d \le Y} 1/d \le 1 + \log Y$, and
$1 + \log Y \le 2 \log 2Y$ for $Y \ge 1$, since
$2 \log 2Y - (1 + \log Y) = \log Y + 2 \log 2 - 1
\ge 2\log 2 - 1 > 0$. This closes the induction.
\end{proof}

\begin{lemma}\label{lem:tau-meansq}
For every fixed $m \ge 1$ and all $Y \ge 1$,
\[
\sum_{n \le Y} \tau_m(n)^2 \ll_m Y (\log 2Y)^{m^2 - 1}.
\]
\end{lemma}

\begin{proof}
By Lemma~\ref{lem:taujk} with $j = k = m$ we have
$\tau_m(n)^2 \le \tau_{m^2}(n)$ pointwise, and then
Lemma~\ref{lem:tau-mean} with $m^2$ in place of $m$ gives the
result.
\end{proof}

\subsection{The functions $d_z$ and Newton's binomial series}\label{subsec:dz}

For $\beta \in \R$ and $k \ge 0$ define
\begin{equation}\label{eq:gk-def}
g_k(\beta) := \frac{\beta(\beta+1)\cdots(\beta + k - 1)}{k!},
\qquad g_0(\beta) := 1,
\end{equation}
so that $g_k(\beta) = \binom{-\beta}{k}(-1)^k$ and, by
\eqref{eq:dz-def}, $d_z(p^k) = g_k(z)$.

\begin{lemma}\label{lem:binom}
For $\beta \in \R$ let
$F_\beta(t) := \exp( -\beta \operatorname{Log}(1-t))$ for
$|t| < 1$, where $\operatorname{Log}$ is the principal logarithm.
Then:
\begin{enumerate}
\item[(i)] $F_\beta$ is holomorphic on the open unit disc and
\[
F_\beta(t) = \sum_{k \ge 0} g_k(\beta)\, t^k \qquad (|t| < 1),
\]
the series converging absolutely there.
\item[(ii)] $F_\beta F_\gamma = F_{\beta + \gamma}$ for all
$\beta, \gamma \in \R$, and $(1-t) F_1(t) = 1$.
\item[(iii)] If $0 < \beta \le 1$ then $0 < g_k(\beta) \le 1$ for all
$k \ge 0$, and $|g_k(-\beta)| \le \beta / k \le 1$ for all $k \ge 1$.
\item[(iv)] If $f, g$ are holomorphic on the unit disc with Taylor
expansions $f = \sum_k c_k(f) t^k$, $g = \sum_k c_k(g) t^k$ there,
then $c_k(fg) = \sum_{i=0}^{k} c_i(f)\, c_{k-i}(g)$ for every
$k \ge 0$. In other words, the map sending a function holomorphic on
the unit disc to its Taylor series at $0$ is a ring homomorphism into
the ring $\C[[t]]$ of formal power series.
\end{enumerate}
\end{lemma}

\begin{proof}
(i) For $|t| < 1$ we have $\real(1 - t) = 1 - \real t > 0$, so $1 - t$ lies
in the right half-plane, where $\operatorname{Log}$ is holomorphic; hence
$F_\beta$ is holomorphic on the unit disc. Since
$\exp(\operatorname{Log} z) = z$, we have
$\exp(-\operatorname{Log}(1-t)) = (1-t)^{-1}$, and since
$\exp(u)\exp(u') = \exp(u + u')$ for complex $u, u'$,
\[
F_\beta'(t)
= F_\beta(t) \cdot (-\beta) \cdot \frac{d}{dt} \operatorname{Log}(1-t)
= \frac{\beta\, F_\beta(t)}{1 - t}
= \beta\, F_{\beta + 1}(t).
\]
By induction on $k$,
$F_\beta^{(k)}(t) = \beta(\beta+1)\cdots(\beta+k-1)\, F_{\beta+k}(t)$,
and since $F_{\beta+k}(0) = \exp(0) = 1$ we get
$F_\beta^{(k)}(0)/k! = g_k(\beta)$. The expansion and its absolute
convergence on the disc of holomorphy now follow from the Taylor
expansion of holomorphic functions \cite[Theorem~10.16]{Ru}.

(ii) $F_\beta F_\gamma = \exp(-\beta L)\exp(-\gamma L)
= \exp(-(\beta+\gamma)L) = F_{\beta+\gamma}$, where
$L = \operatorname{Log}(1-t)$; and
$(1-t)F_1(t) = (1-t)\exp(-\operatorname{Log}(1-t)) = (1-t)(1-t)^{-1} = 1$.

(iii) For $0 < \beta \le 1$ and $k \ge 1$,
\[
g_k(\beta) = \prod_{j=1}^{k} \frac{\beta + j - 1}{j}
= \prod_{j=1}^{k} \frac{j - (1 - \beta)}{j},
\]
and each factor lies in $(0, 1]$ because $0 \le 1 - \beta < 1$. For the
second claim,
\[
|g_k(-\beta)|
= \frac{\big| (-\beta)(1-\beta)(2-\beta)\cdots(k-1-\beta) \big|}{k!}
= \frac{\beta \prod_{j=1}^{k-1}(j - \beta)}{k!}
\le \frac{\beta\, (k-1)!}{k!} = \frac{\beta}{k},
\]
since $0 < j - \beta \le j$ for $1 \le j \le k-1$.

(iv) By the Leibniz rule (induction from the product rule),
$(fg)^{(k)}(0) = \sum_{i} \binom{k}{i} f^{(i)}(0) g^{(k-i)}(0)$; divide
by $k!$.
\end{proof}

\begin{lemma}\label{lem:dz}
Let $\beta, \gamma \in \R$.
\begin{enumerate}
\item[(a)] $d_\beta * d_\gamma = d_{\beta + \gamma}$. In particular,
$d_0 = \delta$, $d_1 = \ID$, $d_m = \ID^{*m} = \tau_m$ for every
integer $m \ge 1$, $\big( d_{\pm a/b} \big)^{*b} = d_{\pm a}$, and
$d_a * d_{-a} = \delta$.
\item[(b)] For every real $z$ with $|z| \le 1$ and all $n \ge 1$,
$|d_z(n)| \le 1$. In particular $0 < \da(n) \le 1$ for all $n$, and
$|\dma(p^k)| \le (a/b)/k$ for $k \ge 1$.
\end{enumerate}
\end{lemma}

\begin{proof}
(a) Both sides are multiplicative -- each $d_z$ is multiplicative
by \eqref{eq:dz-def}, and a Dirichlet convolution of multiplicative
functions is multiplicative -- so it suffices to compare
values at $n = p^k$. The divisors of $p^k$ are $p^i$, $0 \le i \le k$,
so, using $d_z(p^i) = g_i(z)$ and Lemma~\ref{lem:binom}(iv) and then
(ii),
\[
(d_\beta * d_\gamma)(p^k)
= \sum_{i=0}^{k} g_i(\beta)\, g_{k-i}(\gamma)
= c_k( F_\beta F_\gamma )
= c_k( F_{\beta + \gamma} )
= g_k(\beta + \gamma)
= d_{\beta+\gamma}(p^k).
\]
For the particular cases: $g_k(0) = 0$ for $k \ge 1$ and $g_0(0) = 1$,
so $d_0 = \delta$; $g_k(1) = k!/k! = 1$, so $d_1 = \ID$; iterating the
convolution identity gives $d_m = d_1^{*m} = \ID^{*m} = \tau_m$ and
$(d_{\pm a/b})^{*b} = d_{\pm a}$; and
$d_a * d_{-a} = d_0 = \delta$.

(b) For $z = 0$ the claim is trivial, since $d_0 = \delta$. For
$0 < z \le 1$, Lemma~\ref{lem:binom}(iii) with $\beta = z$ gives
$0 < g_k(z) \le 1$ for all $k \ge 0$, and multiplicativity yields
$0 < d_z(n) \le 1$. For $-1 \le z < 0$, apply
Lemma~\ref{lem:binom}(iii) with $\beta = -z \in (0, 1]$: then
$|d_z(p^k)| = |g_k(-\beta)| \le 1$ for all $k \ge 1$, and
multiplicativity yields $|d_z(n)| \le 1$. In particular
$0 < \da(n) \le 1$, and $|\dma(p^k)| = |g_k(-a/b)| \le (a/b)/k$ for
$k \ge 1$, by the same lemma.
\end{proof}

\subsection{The algebra of arithmetic functions}\label{subsec:algebra}

\begin{lemma}\label{lem:domain}
$\A$ is a commutative ring with multiplicative identity $\delta$, and
it is an integral domain. Consequently, if $f, g, g' \in \A$ satisfy
$f * g = f * g'$ and $f \ne 0$, then $g = g'$.
\end{lemma}

\begin{proof}
\begin{sloppypar}
Commutativity and associativity follow from the symmetric
expressions $(f*g)(n) = \sum_{de = n} f(d)g(e)$ and
$(f{*}g{*}h)(n) = \sum_{de\ell = n} f(d)g(e)h(\ell)$, and $\delta$ is clearly
an identity; distributivity over pointwise addition is immediate.
\end{sloppypar}

For the domain property, suppose $f, g \ne 0$ and let
$m(f) = \min\{ n : f(n) \ne 0 \}$ and likewise $m(g)$. If
$de = m(f) m(g)$ with $f(d) g(e) \ne 0$, then $d \ge m(f)$ and
$e \ge m(g)$, which forces $d = m(f)$ and $e = m(g)$. Hence
$(f * g)\big( m(f) m(g) \big) = f(m(f))\, g(m(g)) \ne 0$, so
$f * g \ne 0$.

The cancellation statement follows by applying the domain property
to $f * (g - g') = 0$.
\end{proof}

\begin{lemma}
\label{lem:root}
Let $b \ge 1$ and let $A, C \in \A$ satisfy $A^{*b} = C^{*b}$ and
$A(1) = C(1) = 1$. Then $A = C$.
\end{lemma}

\begin{proof}
Let $\omega = \e(1/b)$, a primitive $b$-th root of unity. The
polynomial $X^b - 1 \in \C[X]$ is monic of degree $b$ and vanishes
at the $b$ distinct points $\omega^j$, $0 \le j \le b-1$, so
$X^b - 1 = \prod_{j=0}^{b-1} (X - \omega^j)$. Expanding the
product, this is an equality of coefficients: writing
$\prod_{j=0}^{b-1}(X - \omega^j Y) = \sum_{i=0}^{b} s_i X^{b-i}
Y^{i}$ with $s_i \in \C$ (the signed elementary symmetric functions
of the $\omega^j$), the one-variable identity says $s_0 = 1$,
$s_b = -1$ and $s_i = 0$ for $0 < i < b$. Hence
\begin{equation}\label{eq:XbYb}
X^b - Y^b = \prod_{j=0}^{b-1} (X - \omega^j Y)
\end{equation}
as polynomials in $\C[X, Y]$. Since $\A$ is a commutative
$\C$-algebra, we may substitute $X = A$, $Y = C$ (both sides of
\eqref{eq:XbYb} expand into the same $\C$-linear combination of the
elements $A^{*i} * C^{*(b-i)}$), obtaining
\[
0 = A^{*b} - C^{*b}
= (A - C) * \prod_{j=1}^{b-1} {}^{*}\, (A - \omega^j C),
\]
where the product is a convolution product. For $1 \le j \le b-1$
we have $(A - \omega^j C)(1) = 1 - \omega^j \ne 0$, so each of
these factors is a nonzero element of $\A$. Since $\A$ is an
integral domain (Lemma~\ref{lem:domain}), a finite convolution
product vanishes only if one of its factors vanishes (induction on
the number of factors). The factors with $j \ge 1$ do not vanish,
so $A - C = 0$.
\end{proof}

\subsection{Exponential sum estimates}\label{subsec:expsums}

\begin{lemma}[Geometric series bound]\label{lem:geom}
Let $\beta \in \R$ and let $I$ be an interval whose intersection with $\Z$ consists of exactly $N \ge 1$ integers. Then
\[
\Big| \sum_{n \in I \cap \Z} \e(n\beta) \Big|
\le \min\Big( N,\ \frac{1}{2\norm{\beta}} \Big),
\]
with the convention that the second entry of the minimum is
$+\infty$ when $\norm{\beta} = 0$.
\end{lemma}

\begin{proof}
The bound by $N$ is the triangle inequality, and covers the case
$\norm{\beta} = 0$. If $\beta \notin \Z$, write the integers of $I$
as $n_0, n_0 + 1, \dots, n_0 + N - 1$; summing the geometric
series,
\[
\Big| \sum_{j = 0}^{N-1} \e\big( (n_0 + j)\beta \big) \Big|
= \Big| \frac{1 - \e(N\beta)}{1 - \e(\beta)} \Big|
\le \frac{2}{|1 - \e(\beta)|}
= \frac{1}{|\sin \pi \beta|}.
\]
Finally $|\sin \pi \beta| = \sin(\pi \norm{\beta}) \ge 2
\norm{\beta}$, since $\sin(\pi t) \ge 2t$ for $0 \le t \le 1/2$
(the function $\sin(\pi t)$ is concave on $[0, 1/2]$ and equals the
linear function $2t$ at the endpoints $t = 0$ and $t = 1/2$). This
is the computation of \cite[eq.~(23.23)]{K}, stated there for the
interval $[1, x]$; the proof for a general interval is identical,
as above.
\end{proof}

\begin{lemma}[{Type I estimate; \cite[eq.~(17.32)]{MV}}]\label{lem:typeI}
Let $\alpha \in \R$ and let $r/q$ be a reduced fraction with
$|\alpha - r/q| \le 1/q^2$. Then for all $N \ge 2$ and $T \ge 1$,
\[
\sum_{0 < t \le T} \min\Big( \frac{N}{t},\ \frac{1}{\norm{t\alpha}} \Big)
\ll \Big( \frac{N}{q} + T + q \Big) \log 2Tq .
\]
\end{lemma}

\begin{proof}
The statement is \cite[eq.~(17.32)]{MV} (see also
\cite[eqs.~(23.29)--(23.30), (23.36)]{K} and
\cite[Chapter~13]{IK} for equivalent forms); we
include the classical block-spacing proof, which is short. If
$q = 1$, each term is at most $N/t$, and
$\sum_{t \le T} N/t \le N(1 + \log T) \ll (N/q)\log(2Tq)$. Assume
$q \ge 2$, write $\beta = \alpha - r/q$, and split $[1, T]$ into
blocks of $\lceil q/2 \rceil$ consecutive integers; there are at
most $2T/q + 1$ blocks.

Fix a block $B$. For distinct $t, t' \in B$ we have
$0 < |t - t'| \le \lceil q/2 \rceil - 1 < q$, so
$q \nmid (t - t') r$ and
\[
\norm{(t - t')\alpha}
\ge \norm{(t - t') r/q} - |t - t'|\, |\beta|
\ge \frac1q - \frac{q/2}{q^2} = \frac1{2q}.
\]
Thus the points $t\alpha$, $t \in B$, are pairwise
$(2q)^{-1}$-spaced modulo $1$. Order the points $t\alpha$, $t \in B$, by distance to the nearest
integer. For $j \ge 2$, the points lying within distance strictly
less than $(j-1)/(4q)$ of $\Z$ occupy an open arc of length
$(j-1)/(2q)$ modulo $1$; since $j$ points that are pairwise
$(2q)^{-1}$-spaced span at least $(j-1)/(2q)$, that arc contains at
most $j - 1$ of them. The $j$-th closest point therefore has
distance at least $(j-1)/(4q)$. All points of $B$ except the
closest one therefore contribute at most
\[
\sum_{2 \le j \le \lceil q/2 \rceil} \frac{4q}{j-1}
\le 4q(1 + \log q) \le 8q\log(2q)
\]
to the left side of the claimed inequality, upon bounding each term
by $\norm{t\alpha}^{-1}$.

It remains to bound the contribution of the closest point $t_B^*$
of each block. If $B$ is the first block, then
$1 \le t_B^* \le \lceil q/2 \rceil \le q - 1$, so $q \nmid t_B^*$,
whence $q \nmid t_B^* r$ and
\[
\norm{t_B^*\alpha} \ge \frac1q - \lceil q/2 \rceil\, |\beta|
\ge \frac1q - \frac{q+1}{2q^2}
= \frac{q-1}{2q^2} \ge \frac1{4q}
\qquad (q \ge 2),
\]
so this point contributes at most $4q$. If $B$ is the $k$-th block
with $k \ge 2$, then
$t_B^* > (k-1)\lceil q/2 \rceil \ge (k-1)q/2$, and we use the other
branch of the minimum: the contribution is at most
$N/t_B^* < 2N/((k-1)q)$, and summing over $2 \le k \le 2T/q + 1$
gives $\ll (N/q)\log(2T)$. Adding the three contributions over all
blocks,
\[
\Big( \frac{2T}{q} + 1 \Big) \big( 8q\log(2q) + 4q \big)
+ \frac{N}{q} \cdot O\big( \log(2T) \big)
\ll \Big( \frac{N}{q} + T + q \Big) \log(2Tq),
\]
which is the claimed bound.
\end{proof}

\begin{lemma}[{Type II estimate; \cite[Theorem~23.6]{K}}]\label{lem:typeII}
Let $f, g : \N \to \C$ be arithmetic functions with
$\operatorname{supp}(f) \subseteq [1, y]$ and
$\operatorname{supp}(g) \subseteq [1, z]$. Let $\alpha \in \R$ and
let $r/q$ be a reduced fraction with $|\alpha - r/q| \le 1/q^2$. Then
for all $x \ge 1$,
\[
\sum_{n \le x} (f * g)(n)\e(\alpha n)
\ll
\Big( q + y + z + \frac{yz}{q} \Big)^{1/2} \sqrt{\log 2q}\;
\|f\|_2\, \|g\|_2 .
\]
\end{lemma}

\begin{proof}[Source of proof]
This is \cite[Theorem~23.6]{K} (see also
\cite[Lemma~13.8]{IK}), proved in full there: after an
application of the Cauchy-Schwarz inequality that removes the factor
$g$, the resulting diagonal-plus-spacing sum is bounded by
\cite[eq.~(23.36)]{K}, which is proved there uniformly in the shift
parameter.
\end{proof}

\begin{lemma}\label{lem:minsq}
Let $\alpha$, $r/q$ be as in Lemma~\ref{lem:typeI} and let
$\kappa \ge 1$ be fixed. Then for all $x \ge 2$ and $T \ge 1$,
\[
\sum_{T < k \le 2^{\kappa} T}
\min\Big( \frac{x}{k},\ \frac{1}{\norm{k\alpha}} \Big)^2
\ll_{\kappa} \frac{x}{T} \Big( \frac{x}{q} + T + q \Big) \log 2Tq .
\]
\end{lemma}

\begin{proof}
Each term of the sum satisfies
$\min(x/k, \norm{k\alpha}^{-1}) \le x/k \le x/T$ on the range
$T < k \le 2^{\kappa} T$. Therefore
\[
\min\Big( \frac{x}{k},\ \frac{1}{\norm{k\alpha}} \Big)^2
\le
\frac{x}{T}
\min\Big( \frac{x}{k},\ \frac{1}{\norm{k\alpha}} \Big).
\]
Now apply Lemma~\ref{lem:typeI} with $N = x$ and $T$ replaced by
$2^{\kappa} T$, discarding the nonnegative terms with $k \le T$,
and note that
$\log(2 \cdot 2^{\kappa} T q) \ll_{\kappa} \log 2Tq$.
\end{proof}

\subsection{Arithmetic lemmas for the applications}\label{subsec:app-arith}

The remaining tools are used only in Section~\ref{sec:applications}.
We begin with the arithmetic ones. Recall Ramanujan's sum $c_q(m)$
and the radical $\rad(a)$ from Section~\ref{subsec:applications}.

\begin{lemma}
\label{lem:ramanujan}
For all $q \ge 1$ and $m \in \Z$:
\begin{enumerate}
\item[(i)] $c_q(m) = \sum_{d \mid (q, m)} d\, \mu(q/d)$
(Kluyver's formula); in particular $c_q(m)$ depends only on $q$ and $(q, m)$, and
$c_q(1) = \mu(q)$.
\item[(ii)] For a prime power $q = p^{k}$ and $p^{j} \,\|\, (q, m)$
(so $0 \le j \le k$),
\[
c_{p^k}(m) =
\begin{cases}
\varphi(p^{k}) & \text{if } j = k, \\
-p^{k-1} & \text{if } j = k - 1, \\
0 & \text{if } j \le k - 2.
\end{cases}
\]
\item[(iii)] $|c_q(m)| \le (q, m)\, \tau\big( (q, m) \big)$.
\end{enumerate}
\end{lemma}

\begin{proof}
(i) Detecting the condition $(c, q) = 1$ by M\"obius inversion,
\[
c_q(m) = \sum_{c = 1}^{q} \e\Big( \frac{cm}{q} \Big)
\sum_{d \mid (c, q)} \mu(d)
= \sum_{d \mid q} \mu(d) \sum_{\substack{1 \le c \le q \\ d \mid c}}
\e\Big( \frac{cm}{q} \Big)
= \sum_{d \mid q} \mu(d) \sum_{e = 1}^{q/d}
\e\Big( \frac{em}{q/d} \Big),
\]
writing $c = de$. The inner sum is a full sum of the additive
character $e \mapsto \e(em/(q/d))$ over $\Z/(q/d)\Z$; it equals
$q/d$ if $(q/d) \mid m$ and $0$ otherwise. Substituting $d' = q/d$,
\[
c_q(m) = \sum_{\substack{d' \mid q \\ d' \mid m}} \mu\Big(
\frac{q}{d'} \Big)\, d'
= \sum_{d' \mid (q, m)} d'\, \mu\Big( \frac{q}{d'} \Big),
\]
which is (i); the evaluation $c_q(1) = \mu(q)$ is the case
$(q, m) = 1$.

(ii) By (i) with $(q, m) = p^{j}$,
$c_{p^k}(m) = \sum_{i = 0}^{j} p^{i} \mu(p^{k - i})$, and
$\mu(p^{k-i}) = 0$ unless $i \ge k - 1$. If $j \le k - 2$ no term
survives. If $j = k - 1$ only $i = k - 1$ survives, giving
$-p^{k-1}$. If $j = k$ the terms $i = k - 1, k$ survive, giving
$p^{k} - p^{k-1} = \varphi(p^{k})$.

(iii) From (i), $|c_q(m)| \le \sum_{d \mid (q,m)} d \le
(q, m)\, \tau\big( (q,m) \big)$.
\end{proof}

\begin{lemma}\label{lem:cq-mult}
For fixed $m$, the function $q \mapsto c_q(m)$ is multiplicative: if
$(q_1, q_2) = 1$ then $c_{q_1 q_2}(m) = c_{q_1}(m)\, c_{q_2}(m)$.
\end{lemma}

\begin{proof}
By Lemma~\ref{lem:ramanujan}(i),
$c_q(m) = \sum_{d \mid (q, m)} d\, \mu(q/d)$. If $(q_1, q_2) = 1$,
then $(q_1 q_2, m) = (q_1, m)(q_2, m)$ with the two factors coprime
(compare valuations at each prime $p$: at most one of $v_p(q_1)$,
$v_p(q_2)$ is nonzero, so
$\min\big( v_p(q_1) + v_p(q_2), v_p(m) \big)
= \min\big( v_p(q_1), v_p(m) \big)
+ \min\big( v_p(q_2), v_p(m) \big)$), so every divisor $d$ of
$(q_1 q_2, m)$ factors uniquely as $d = d_1 d_2$ with
$d_i \mid (q_i, m)$, and
$\mu(q_1 q_2/d) = \mu(q_1/d_1)\, \mu(q_2/d_2)$ since
$(q_1/d_1, q_2/d_2) = 1$. The double sum then factors.
\end{proof}

\begin{lemma}
\label{lem:fourier}
Let $q \ge 1$, let $m \in \Z$, and let $b$ be an integer with
$(b, q) = 1$. Then
\[
\e\Big( \frac{mb}{q} \Big)
= \frac{1}{\varphi(q)} \sum_{\chi \bmod q}
\overline{\chi}(b)\, \tau(\chi, m),
\qquad \text{where} \qquad
\tau(\chi, m) := \sum_{\substack{1 \le c \le q \\ (c, q) = 1}}
\chi(c) \e\Big( \frac{mc}{q} \Big) .
\]
Moreover $\tau(\chi_0, m) = c_q(m)$ and $|\tau(\chi, m)| \le
\varphi(q)$ for every $\chi$.
\end{lemma}

\begin{proof}
Expanding the right-hand side and interchanging the finite sums,
\[
\frac{1}{\varphi(q)} \sum_{\chi} \overline{\chi}(b) \tau(\chi, m)
= \sum_{\substack{c \bmod q \\ (c,q) = 1}} \e\Big( \frac{mc}{q} \Big)
\cdot \frac{1}{\varphi(q)} \sum_{\chi} \overline{\chi}(b) \chi(c) .
\]
By the orthogonality of Dirichlet characters, the inner average is
$1$ if $c \equiv b \pmod q$ and $0$ otherwise (both $b$ and $c$ being
coprime to $q$); only $c \equiv b$ survives, leaving $\e(mb/q)$. The
identification $\tau(\chi_0, m) = c_q(m)$ is the definition
\eqref{eq:ramanujan-def}, and the bound $|\tau(\chi, m)| \le
\varphi(q)$ is the triangle inequality.
\end{proof}

\begin{lemma}\label{lem:rad-sums}
Let $q \ge 1$.
\begin{enumerate}
\item[(i)] For $\sigma > 0$,
$\displaystyle \sum_{\rad(a) \mid q} a^{-\sigma}
= \prod_{p \mid q} \big( 1 - p^{-\sigma} \big)^{-1}$,
the series converging absolutely.
\item[(ii)] For $T \ge 1$,
$\displaystyle \sum_{\substack{a > T \\ \rad(a) \mid q}} \frac1a
\le \frac{4^{\omega(q)}}{\sqrt T}$,
where $\omega(q)$ is the number of distinct prime factors of $q$.
\item[(iii)] For every constant $C \ge 1$ and every
$\varepsilon > 0$, $C^{\omega(q)} \ll_{C, \varepsilon}
q^{\varepsilon}$.
\end{enumerate}
\end{lemma}

\begin{proof}
(i) Every $a$ with $\rad(a) \mid q$ factors uniquely as
$a = \prod_{p \mid q} p^{j_p}$ with $j_p \ge 0$; expanding the product
of the geometric series $\sum_{j \ge 0} p^{-j\sigma} =
(1 - p^{-\sigma})^{-1}$ and rearranging (all terms positive) gives the
identity.

(ii) For $a > T$ we have $a^{-1} \le T^{-1/2} a^{-1/2}$, so by (i)
with $\sigma = 1/2$ the sum is at most
$T^{-1/2} \prod_{p \mid q} (1 - p^{-1/2})^{-1} \le T^{-1/2}
\big( (1 - 2^{-1/2})^{-1} \big)^{\omega(q)} \le T^{-1/2}
4^{\omega(q)}$, since $(1 - 2^{-1/2})^{-1} < 4$.

(iii) Split the primes dividing $q$ at $C^{1/\varepsilon}$:
\[
C^{\omega(q)} = \prod_{\substack{p \mid q \\ p \le C^{1/\varepsilon}}} C
\cdot \prod_{\substack{p \mid q \\ p > C^{1/\varepsilon}}} C
\le C^{\pi( C^{1/\varepsilon} )}
\prod_{\substack{p \mid q \\ p > C^{1/\varepsilon}}} p^{\varepsilon}
\le C^{ C^{1/\varepsilon} }\, q^{\varepsilon},
\]
where in the middle step we used $C \le p^{\varepsilon}$ for
$p > C^{1/\varepsilon}$, and $\pi(y) \le y$.
\end{proof}

\begin{lemma}[$\g_z$ at prime powers]\label{lem:gz-primepower}
Let $0 < |z| < 1$ and let $\g_z$ be as in \eqref{eq:g-def}. For
every prime $p$ and every $k \ge 1$,
\begin{equation}\label{eq:gz-primepower}
\g_z(p^{k})
=
\Big( 1 - \frac1p \Big)^{z}
\bigg[
\sum_{j \ge k} \frac{d_z(p^{j})}{p^{j}}
\;-\;
\frac{d_z(p^{k-1})}{p^{k-1} (p - 1)}
\bigg] .
\end{equation}
In particular, the same computation applies verbatim at
$z = \pm 1$, where all the series involved still converge
absolutely: at $z = 1$ one has $\g_1(p^{k}) = 0$ for
all $k \ge 1$; and at $z = -1$ the local series
terminates, with $\g_{-1}(p) = -(2p-1)/(p-1)^2$,
$\g_{-1}(p^2) = (p-1)^{-2}$ and $\g_{-1}(p^{k}) = 0$ for $k \ge 3$.
\end{lemma}

\begin{proof}
Specializing \eqref{eq:g-def} to $q = p^{k}$, the summation
variable is $a = p^{j}$, $j \ge 0$, and by
Lemma~\ref{lem:ramanujan}(ii),
\[
\g_z(p^{k})
= \Big( 1 - \frac1p \Big)^{z} \frac{1}{p^{k-1}(p-1)}
\bigg[
- p^{k-1}\, \frac{d_z(p^{k-1})}{p^{k-1}}
+ p^{k-1} (p - 1) \sum_{j \ge k} \frac{d_z(p^{j})}{p^{j}}
\bigg],
\]
which is \eqref{eq:gz-primepower}. At $k = 1$ the tail is
$\sum_{j \ge 1} d_z(p^{j}) p^{-j} = (1 - 1/p)^{-z} - 1$, and
\eqref{eq:gz-primepower} becomes
$(1 - 1/p)^{z} [ (1-1/p)^{-z} - 1 - \frac{1}{p-1} ]
= 1 - (1 - 1/p)^{z-1}$, matching Theorem~\ref{thm:local}. At
$z = 1$, where $d_1 \equiv 1$, the tail equals
$p^{-k} \cdot p/(p-1) = 1/(p^{k-1}(p-1))$ and the bracket vanishes
for every $k \ge 1$. At $z = -1$, where $d_{-1} = \mu$, only
$j \in \{0, 1\}$ contribute: for $k = 1$ the bracket is
$-1/p - 1/(p-1) = -(2p-1)/(p(p-1))$ and
$\g_{-1}(p) = (1-1/p)^{-1} \cdot ( -(2p-1)/(p(p-1)) )
= -(2p-1)/(p-1)^2$; for $k = 2$ the tail is empty and the bracket is
$-d_{-1}(p)/(p(p-1)) = 1/(p(p-1))$, giving
$\g_{-1}(p^2) = (p-1)^{-2}$; for $k \ge 3$ both the tail and
$d_{-1}(p^{k-1})$ vanish.
\end{proof}

\begin{lemma}
\label{lem:dirichlet}
Let $\alpha \in \R$ and $Q \ge 1$. There is a reduced fraction $r/q$
with
\[
1 \le q \le Q \quadand \Big| \alpha - \frac rq \Big| \le \frac{1}{qQ}
\le \frac{1}{q^{2}} .
\]
\end{lemma}

\begin{proof}[Source of proof]
This is \cite[Lemma~23.4]{K}. If the fraction produced there is not
reduced, lowering it to $r/q$ preserves both inequalities, since the
denominator only decreases.
\end{proof}

\subsection{Analytic lemmas for the applications}\label{subsec:app-analytic}

We now record the analytic tools. The first two are elementary
estimates for binomial expansions and logarithmic integrals.

\begin{lemma}
\label{lem:taylor}
Let $w \in \R$. Then
$\big| \binom{w}{i} \big| \le \big( e(i+1) \big)^{|w|}$ for all
$i \ge 0$. Moreover, for every integer $I \ge 1$ there is a constant
$C_2(w, I) > 0$, depending only on $w$ and $I$, such that for all
$0 \le u \le 1/2$,
\[
\bigg| (1 - u)^{w} - \sum_{i = 0}^{I - 1} \binom{w}{i} (-u)^{i} \bigg|
\;\le\; C_2(w, I)\, u^{I} .
\]
\end{lemma}

\begin{proof}
For $i \ge 1$,
\[
\Big| \binom{w}{i} \Big|
= \prod_{l = 0}^{i-1} \frac{|w - l|}{l + 1}
\le \prod_{l = 0}^{i-1} \frac{|w| + l}{l + 1}
= \prod_{l = 0}^{i-1} \Big( 1 + \frac{|w| - 1}{l + 1} \Big)
\le \exp\Big( |w| \sum_{l = 0}^{i-1} \frac{1}{l+1} \Big)
\le \big( e (i + 1) \big)^{|w|},
\]
using $1 + t \le e^{t}$, $\sum_{l < i} (l+1)^{-1} \le 1 + \log i$
and $\exp( |w| (1 + \log i) ) \le ( e (i+1) )^{|w|}$;
for $i = 0$ the bound is trivial. By Lemma~\ref{lem:binom}(i) (applied
with $\beta = -w$, and $t = u$), the expansion
$(1 - u)^{w} = \sum_{i \ge 0} \binom{w}{i} (-u)^{i}$ holds with
absolute convergence for $|u| < 1$. Hence the left-hand side of the
claim equals $| \sum_{i \ge I} \binom{w}{i} (-u)^{i} |$, which for
$0 \le u \le 1/2$ is at most
\[
\sum_{i \ge I} \big( e(i+1) \big)^{|w|} u^{i}
= u^{I} \sum_{m \ge 0} \big( e(I + m + 1) \big)^{|w|} u^{m}
\le u^{I} \sum_{m \ge 0} \big( e(I + m + 1) \big)^{|w|} 2^{-m}
=: C_2(w, I)\, u^{I},
\]
the last series converging since its terms are polynomial in $m$
times $2^{-m}$.
\end{proof}

\begin{lemma}\label{lem:log-integral}
Let $w \le 0$ be real. For all $x \ge 9$,
\[
\int_{2}^{x} (\log t)^{w}\, dt
\le 2^{|w|}\, x (\log x)^{w} + (\log 2)^{w} \sqrt x .
\]
\end{lemma}

\begin{proof}
Since $w \le 0$, the integrand is decreasing in $t$. On
$[2, \sqrt x]$ bound it by its value at $t = 2$, giving
$(\log 2)^{w} \sqrt x$; on $[\sqrt x, x]$ bound it by its value at
$t = \sqrt x$, namely
$(\tfrac12 \log x)^{w} = 2^{|w|} (\log x)^{w}$, giving
$2^{|w|} x (\log x)^{w}$.
\end{proof}

The remaining lemmas concern the kernels produced by the major-arc
expansion of Theorem~\ref{thm:major}. In them, $X \ge 3$ is large,
$L := \log X$, and $w$ denotes a fixed real number with $w < 0$; in
the application, $w = z - 1 - j$ for the $j$-th term of
\eqref{eq:major}. For $u \in \R$ set
\begin{equation}\label{eq:Phi-def}
\Phi_w(u)
:=
\int_{2/X}^{1} \Big( 1 + \frac{\log v}{\log X} \Big)^{w}
\e(uv)\, dv,
\end{equation}
which is the natural rescaling of the kernel of
Theorem~\ref{thm:major}: substituting $t = vX$ gives
\[
\int_{2}^{X} (\log t)^{w}\e(\beta t)\, dt
= X (\log X)^{w}\, \Phi_w(\beta X).
\]
The model to which $\Phi_w$ will be compared is the kernel
\begin{equation}\label{eq:K0}
K_0(u) := \int_0^1 \e(uv)\, dv = \frac{\e(u) - 1}{2\pi i u},
\qquad
|K_0(u)| = \Big| \frac{\sin \pi u}{\pi u} \Big|,
\end{equation}
so that $K_0$ is entire with $K_0(0) = 1$ and
$|K_0(u)| \le \min(1, 1/|u|)$ for all real $u$.

\begin{lemma}
\label{lem:Wj-bounds}
For each fixed real $w < 0$, all $X \ge 4$ and all real
$\beta \ne 0$,
\begin{equation}\label{eq:Wj-bounds}
\int_2^X (\log t)^{w}\e(\beta t)\, dt
\ll_w
\min\Big( X (\log X)^{w},\ \frac{1}{|\beta|} \Big).
\end{equation}
\end{lemma}

\begin{proof}
For the first bound, split the integral of the absolute value at
$\sqrt X$: on $[2, \sqrt X]$ the decreasing integrand is at most
$(\log 2)^{w} = O_w(1)$, giving $O_w(\sqrt X)$, while on
$[\sqrt X, X]$ it is at most $(\log \sqrt X)^{w} = 2^{-w} (\log X)^w$,
giving $O_w(X (\log X)^w)$; and $\sqrt X \ll X(\log X)^w$ for any
fixed $w$. For the second bound, integrate by parts:
\[
\int_2^X (\log t)^{w} \e(\beta t)\, dt
= \Big[ (\log t)^{w} \frac{\e(\beta t)}{2\pi i \beta} \Big]_2^X
- \frac{1}{2\pi i \beta} \int_2^X \frac{w (\log t)^{w-1}}{t}\,
\e(\beta t)\, dt,
\]
and the boundary terms are $\ll_w 1/|\beta|$, while
$\int_2^X |w| (\log t)^{w-1} t^{-1}\, dt
= (\log 2)^{w} - (\log X)^{w} \le (\log 2)^{w} = O_w(1)$
since $w < 0$.
\end{proof}

\begin{lemma}
\label{lem:kernel-approx}
Fix a real number $w < 0$ and a constant $A > 0$. Then there is
$X_0 = X_0(w, A)$ such that for all $X \ge X_0$,
\begin{equation}\label{eq:kernel-approx}
\big| \Phi_w(u) - K_0(u) \big|
\ll_{w, A}
\min\Big( \frac{1}{\log X},\ \frac{1}{|u| (\log X)^{1/2}} \Big)
\qquad
\big( 0 < |u| \le (\log X)^{A} \big),
\end{equation}
and the uniform bound
$|\Phi_w(u) - K_0(u)| \ll_w 1/\log X$ holds for all real $u$,
including $u = 0$.
\end{lemma}

\begin{remark}
The restriction on $|u|$ in \eqref{eq:kernel-approx} is the natural
one: the lower limit $2/X$ in
\eqref{eq:Phi-def} contributes to $\Phi_w - K_0$ the term
$-\int_0^{2/X} \e(uv)\, dv$, whose modulus at $u = X/4$ equals
$1/(\pi u)$, so the proof cannot yield the displayed decaying bound
uniformly in all real $u$.
\end{remark}

\begin{proof}
Write $g(v) = (1 + \frac{\log v}{\log X})^{w} - 1$ for
$v \in [2/X, 1]$, so that
\begin{equation}\label{eq:diff-split}
\Phi_w(u) - K_0(u)
=
\int_{2/X}^1 g(v)\e(uv)\, dv \;-\; \int_0^{2/X} \e(uv)\, dv,
\end{equation}
and the second integral is $\ll 1/X$, which is admissible for both
bounds in \eqref{eq:kernel-approx}, since
$1/X \ll_A 1/(|u| (\log X)^{1/2})$ in the range
$|u| \le (\log X)^{A}$. Note that on $[2/X, 1]$ the base
satisfies $1 + \frac{\log v}{\log X} \ge \frac{\log 2}{\log X}$, so
$|g(v)| \ll (\log X)^{|w|}$ pointwise, and that for
$v \ge X^{-1/2}$ the base lies in $[\tfrac12, 1]$, whence by the
mean value theorem applied to $h \mapsto (1+h)^{w}$ on
$[-\tfrac12, 0]$,
\begin{equation}\label{eq:g-mvt}
|g(v)| \ll_w \frac{|\log v|}{\log X}
\qquad (X^{-1/2} \le v \le 1).
\end{equation}

\emph{The uniform bound.} By \eqref{eq:g-mvt} and the pointwise
bound,
\[
\int_{2/X}^1 |g(v)|\, dv
\ll_w
\frac{1}{\log X} \int_0^1 |\log v|\, dv
\;+\; X^{-1/2} (\log X)^{|w|}
\ll_w \frac{1}{\log X},
\]
using $\int_0^1 |\log v|\, dv = 1$. This proves the first bound in
\eqref{eq:kernel-approx} and, since no restriction on $u$ was used,
the final uniform claim of the lemma as well.

\emph{The decaying bound.} Set $v_1 = \exp(-\sqrt{\log X})$, so that
$1 + \frac{\log v_1}{\log X} = 1 - (\log X)^{-1/2} \ge \tfrac12$ for
large $X$. The portion of the first integral in
\eqref{eq:diff-split} over $[2/X, v_1]$ is, by \eqref{eq:g-mvt} and
the pointwise bound,
\[
\ll_w \int_{X^{-1/2}}^{v_1} \frac{|\log v|}{\log X}\, dv
+ X^{-1/2} (\log X)^{|w|}
\ll_w \frac{v_1 \sqrt{\log X}}{\log X} + X^{-1/2} (\log X)^{|w|}
\ll \exp\big( -\tfrac12 \sqrt{\log X} \big),
\]
which is $\ll_A 1/(|u| (\log X)^{1/2})$ in the range
$|u| \le (\log X)^{A}$, hence admissible; it is only here that the
restriction on $|u|$ is used. On $[v_1, 1]$ we integrate by parts. Since
$g(1) = 0$,
\[
\int_{v_1}^1 g(v)\e(uv)\, dv
= - g(v_1) \frac{\e(u v_1)}{2\pi i u}
- \frac{1}{2\pi i u} \int_{v_1}^1 g'(v)\e(uv)\, dv.
\]
By \eqref{eq:g-mvt}, $|g(v_1)| \ll_w \sqrt{\log X}/\log X
= (\log X)^{-1/2}$. For the second term,
$g'(v) = \frac{w}{v \log X} (1 + \frac{\log v}{\log X})^{w-1}$, and
on $[v_1, 1]$ the base is at least $\tfrac12$, so
$|g'(v)| \ll_w \frac{1}{v \log X}$ and
\[
\int_{v_1}^1 |g'(v)|\, dv
\ll_w \frac{1}{\log X} \log \frac{1}{v_1}
= \frac{\sqrt{\log X}}{\log X}
= (\log X)^{-1/2}.
\]
Both contributions are $\ll_w |u|^{-1} (\log X)^{-1/2}$, proving the
second bound in \eqref{eq:kernel-approx}.
\end{proof}

\begin{lemma}
\label{lem:kernel-moments}
For every fixed $s > 1$,
\begin{equation}\label{eq:kernel-moments}
\int_{\R} |K_0(u)|^{s}\, du
=
\frac{2}{\pi} \int_0^{\infty} \frac{|\sin t|^{s}}{t^{s}}\, dt
= A_s,
\qquad
\int_{|u| > T} |K_0(u)|^{s}\, du \ll_s T^{1-s}
\quad (T \ge 1).
\end{equation}
\end{lemma}

\begin{proof}
By \eqref{eq:K0}, $\int_{\R} |K_0|^s\, du
= 2 \int_0^{\infty} |\sin \pi u|^{s} (\pi u)^{-s}\, du$, and the
substitution $t = \pi u$ gives the first identity exactly; the
integral converges at $0$ (integrand $\le 1$) and at $\infty$
(integrand $\le t^{-s}$, $s > 1$). The tail bound follows from
$|K_0(u)| \le 1/|u|$.
\end{proof}

\begin{lemma}
\label{lem:perturbation}
Let $s \ge 1$ and let $w_1, w_2 \in \mathbb{C}$. Then
\begin{equation}\label{eq:perturbation}
\big| |w_1|^{s} - |w_2|^{s} \big|
\le s \big( |w_1| + |w_2| \big)^{s-1} |w_1 - w_2| .
\end{equation}
If $s > 1$, then also
\[
\big| |w_1|^{s} - |w_2|^{s} \big|
\ll_s
\big( |w_2|^{s-1} + |w_1 - w_2|^{s-1} \big) |w_1 - w_2| .
\]
\end{lemma}

\begin{proof}
By the mean value theorem applied to $x \mapsto x^{s}$ on the
interval with endpoints $|w_1|, |w_2|$,
\[
\big| |w_1|^{s} - |w_2|^{s} \big|
\le
s \max\big( |w_1|, |w_2| \big)^{s-1} \big| |w_1| - |w_2| \big|
\le
s \big( |w_1| + |w_2| \big)^{s-1} |w_1 - w_2|,
\]
using the triangle inequality twice. The second form follows from
$|w_1| \le |w_2| + |w_1 - w_2|$ and
$(x + y)^{s-1} \ll_s x^{s-1} + y^{s-1}$ for $x, y \ge 0$.
\end{proof}

The final lemma invokes Theorem~\ref{thm:local}, which is proved in
Section~\ref{sec:local} independently of the present lemma; there is
no circularity. For the last lemma of the toolbox, let $z$ be a fixed rational with
$0 < |z| < 1$, let $s > 2$ be a fixed real number, let
$\lambda_0(z, q) = \g_z(q)/\Gamma(z)$ be as in
Theorem~\ref{thm:major}, and let $X \ge 3$, $L := \log X$, and
$P := L^{B}$ for a fixed $B \ge 2/(s-2)$.

\begin{lemma}\label{lem:singular-series}
Let $z$, $s$, $X$, $L$ and $P = L^{B}$, with $B \ge 2/(s-2)$ fixed,
be as above. The series
$\mathfrak G_{s}(z) := \sum_{q \ge 1} \varphi(q)\,
|\g_{z}(q)|^{s}$
converges absolutely for $s > 2$, with
$\mathfrak G_{s}(z) \ge 1$, and, for every fixed
$\varepsilon$ with $0 < \varepsilon < (s-2)/(2s)$,
\begin{equation}\label{eq:singular-series}
\sum_{q \le P} \varphi(q)\, |\lambda_0(z, q)|^{s}
=
\frac{\mathfrak G_{s}(z)}{|\Gamma(z)|^{s}} \;+\; O_{z, s, \varepsilon}\big( L^{-1} \big),
\qquad
\sum_{q \le P} \varphi(q)\, q^{(-1+\varepsilon)s} \ll_s 1 .
\end{equation}
\end{lemma}

\begin{proof}
By Theorem~\ref{thm:local},
$\g_z(q) \ll_{z, \varepsilon} q^{-1+\varepsilon}$, so
$\varphi(q)\, |\g_z(q)|^{s} \ll q^{1 + (-1+\varepsilon)s}$, and with
$\varepsilon < (s-2)/(2s)$ the exponent is at most
$-1 - (s-2)/2 < -1$: the series converges absolutely, and the tail
beyond $P$ is
$\ll P^{-(s-2)/2} = L^{-B(s-2)/2} \ll L^{-1}$, since
$B(s-2)/2 \ge 1$ by the hypothesis $B \ge 2/(s-2)$. The
term $q = 1$ equals $1$ (as $\g_z(1) = 1$), and all terms are
nonnegative, so $\mathfrak G_{s}(z) \ge 1$. The first claim of
\eqref{eq:singular-series} follows from
$\lambda_0(z, q) = \g_z(q)/\Gamma(z)$, and the second from the same
exponent computation.
\end{proof}

\section{A Heath-Brown identity for $\zeta^{\pm a/b}$}\label{sec:identity}

This section is entirely algebraic: no exponential sums appear in it. We
work in the ring $\A$ of Section~\ref{subsec:notation}, written in
Dirichlet-series notation: a series $\sum_n f(n) n^{-s}$ stands for
its coefficient function $f \in \A$, products of series correspond
to Dirichlet convolutions, and an identity between series is said to
hold \emph{for coefficients with $n \le x$} when the coefficient
functions of the two sides agree at every integer $n \le x$. The
threshold $V \ge 1$ below is a free parameter; the estimation of
Section~\ref{sec:main-proof} will take $V = x^{1/5}$. We warm up
with the exponents $\pm 1/2$, where every object can be displayed
explicitly, and then prove the general identity; the warm-up makes
no claims that are not proved, in general, in
Section~\ref{subsec:general}.

\subsection{Warm-up: the exponents $\pm 1/2$}\label{subsec:half}

Take $b = 2$ and let
\[
F(s) = \sum_{n \le V} d_{-1/2}(n)\, n^{-s},
\]
the initial segment of $\zeta^{-1/2}$. Newton's binomial series for the inverse square root is
\[
(1 - Z)^{-1/2}
= 1 + \frac12 Z + \frac38 Z^2 + \frac{5}{16} Z^3
+ \frac{35}{128} Z^4 + Z^5 R_2(Z),
\]
in which the coefficient of $Z^k$ is $g_k(1/2) = \binom{2k}{k}/4^k$
in the notation \eqref{eq:gk-def}, and the remainder $R_2$ is simply the entire infinite tail of the series,
shifted down by five degrees,
\[
R_2(Z) = \sum_{j \ge 0} g_{j+5}\big( \tfrac12 \big)\, Z^{j}
= \frac{63}{256} + \frac{231}{1024} Z + \frac{429}{2048} Z^2
+ \cdots .
\]
By Lemma~\ref{lem:binom}(iii) every $g_k(1/2)$ is strictly positive,
so the series genuinely never terminates: this is the contrast with
an integer exponent, where the analogous remainder vanishes
identically once the degree passes the exponent, and it is the
reason no classical identity applies as stated. Any finite identity
for $\zeta^{-1/2}$ must therefore remove the contribution of the tail $Z^5 R_2(Z)$ by some means other than algebra.

The object that accomplishes this is the remainder
\[
H = 1 - \zeta F^2 .
\]
It measures how far $F$ is from being an exact square root of
$\zeta^{-1}$: if $F$ were the full series $\zeta^{-1/2}$, then
$\zeta F^2$ would equal $1$ and $H$ would vanish identically.
Because $F$ is only the initial segment, $H$ vanishes only
initially -- but that is all we need. Indeed, fix $n \le V$. The
coefficient of $n^{-s}$ in $F^2$ is $\sum_{de = n} F(d) F(e)$, and
every divisor of $n$ is at most $V$, where the coefficients of $F$
and of $\zeta^{-1/2}$ agree; so this coefficient equals that of
$(\zeta^{-1/2})^2 = \zeta^{-1}$. Hence the coefficient of
$\zeta F^2$ at $n$ equals that of $\zeta \cdot \zeta^{-1} = 1$, and
$H$ has coefficient $0$ at every $n \le V$. Consequently $H^5$ has no coefficients at integers $n \le V^5$:
indeed, every nonzero coefficient of $H^5$ occurs at an integer
$n \ge (\lfloor V \rfloor + 1)^5 > V^5$.

Now substitute $Z = H$ into the \emph{full} Newton series. The
substitution is legitimate because $H$ has zero coefficient at
$n = 1$, and it produces an \emph{exact} identity in $\A$, with
infinitely many terms: writing $B = \sum_{j \ge 0} c_{2,j} H^j$
with $c_{2,j} = g_j(1/2)$, one has $(1 - H) B^2 = 1$, so that
$FB$ is a square root of $\zeta^{-1}$ with coefficient $1$ at
$n = 1$, and by the uniqueness of such roots $FB = \zeta^{-1/2}$
exactly, at every integer. The role of the support argument is then
to show that the infinite part of this identity is invisible in the
range we care about: under the substitution, the tail $Z^5 R_2(Z)$
becomes $H^5 R_2(H)$, and every one of its coefficients at
$n \le V^5$ vanishes, because every term carries the factor $H^5$.
The tail of the binomial series is thus not discarded as an
approximation; it is confined beyond $V^5$, where the identity
below makes no claim. (Both steps -- the legitimacy of the
substitution and the uniqueness of roots -- are proved in general
in Section~\ref{subsec:general}.) For completeness, we spell out the collection. Write
$P = \zeta F^{2}$, so that $H = 1 - P$, and expand the five retained
terms:
\[
1 + \tfrac12 H + \tfrac38 H^{2} + \tfrac{5}{16} H^{3}
+ \tfrac{35}{128} H^{4}
= 1 + \tfrac12 (1 - P) + \tfrac38 (1 - P)^{2}
+ \tfrac{5}{16} (1 - P)^{3} + \tfrac{35}{128} (1 - P)^{4} .
\]
Collecting, the coefficient of $P^{m}$ is
\[
P^{0}\colon\ 1 + \tfrac12 + \tfrac38 + \tfrac{5}{16}
+ \tfrac{35}{128} = \tfrac{315}{128},
\qquad
P^{1}\colon\ -\tfrac12 - \tfrac34 - \tfrac{15}{16} - \tfrac{35}{32}
= -\tfrac{105}{32},
\]
\[
P^{2}\colon\ \tfrac38 + \tfrac{15}{16} + \tfrac{105}{64}
= \tfrac{189}{64},
\qquad
P^{3}\colon\ -\tfrac{5}{16} - \tfrac{35}{32} = -\tfrac{45}{32},
\qquad
P^{4}\colon\ \tfrac{35}{128}.
\]
Multiplying by $F$ and using
$F P^{m} = F (\zeta F^{2})^{m} = \zeta^{m} F^{2m+1}$, the
coefficients rearrange into the rationals announced in the
introduction:
\begin{equation}\label{eq:d-minus-half-example}
\zeta(s)^{-1/2}
= \frac{315}{128} F
- \frac{105}{32} \zeta F^3
+ \frac{189}{64} \zeta^2 F^5
- \frac{45}{32} \zeta^3 F^7
+ \frac{35}{128} \zeta^4 F^9
\qquad (n \le V^5).
\end{equation}
Every term on the right is a convolution of copies of $d_{-1/2}$
restricted to $[1, V]$ -- coefficients bounded by $1$, by
Lemma~\ref{lem:dz}(b) -- with a nonnegative power of $\zeta$, which is smooth in the sense
of the introduction: a power $\zeta^{J}$ is the $J$-fold convolution
power of the constant function $\ID$, so its variables are
unrestricted and carry coefficient $1$.

The positive exponent is obtained without constructing a new
identity: since $\zeta^{1/2} = \zeta \cdot \zeta^{-1/2}$,
multiplying \eqref{eq:d-minus-half-example} by $\zeta$ gives
\begin{equation}\label{eq:d-plus-half-example}
\zeta(s)^{1/2}
= \frac{315}{128} \zeta F
- \frac{105}{32} \zeta^2 F^3
+ \frac{189}{64} \zeta^3 F^5
- \frac{45}{32} \zeta^4 F^7
+ \frac{35}{128} \zeta^5 F^9
\qquad (n \le V^5),
\end{equation}
the range of validity being preserved because multiplication by
$\zeta$ only combines coefficients at divisors (this step, too, is
proved in general below). Note the effect: every summand of
\eqref{eq:d-plus-half-example} now contains at least one full factor
$\zeta$, so each term has a long smooth part, while the short
factors are still weighted by $d_{-1/2}$ and bounded by $1$. This is
the reason the positive case is derived from the negative one rather
than from a direct expansion of $(1-Z)^{1/2}$.

\subsection{The general identity}\label{subsec:general}

Throughout this subsection, $1 \le h \le b$ and $(h, b) = 1$; the
endpoint $h = b$ forces $h = b = 1$ and is deliberately included,
since it produces the M\"obius case needed for
Theorem~\ref{thm:mu}. The standing exponents $\pm a/b$ of
Theorem~\ref{thm:main} correspond to $h = a < b$.

\subsubsection*{The binomial coefficients}
For an integer $b \ge 1$, define, in the notation \eqref{eq:gk-def},
\begin{equation}\label{eq:c-b-j-def}
c_{b,j} := g_j(1/b)
= \frac{(1/b)(1/b + 1) \cdots (1/b + j - 1)}{j!}
\qquad (j \ge 0),
\end{equation}
so that, by Lemma~\ref{lem:binom}(i) applied with $\beta = 1/b$,
\begin{equation}\label{eq:binomial-four-terms}
(1 - Z)^{-1/b}
= c_{b,0} + c_{b,1} Z + c_{b,2} Z^2 + c_{b,3} Z^3 + c_{b,4} Z^4
+ Z^5 R_b(Z)
\qquad (|Z| < 1),
\end{equation}
where $R_b(Z) = \sum_{j \ge 0} c_{b, j+5} Z^{j}$, with
\begin{align*}
c_{b,0} &= 1,
& c_{b,1} &= \frac1b,
& c_{b,2} &= \frac{b+1}{2b^2}, \\
c_{b,3} &= \frac{(b+1)(2b+1)}{6b^3},
& c_{b,4} &= \frac{(b+1)(2b+1)(3b+1)}{24b^4}.
\end{align*}
When powers of $1 - P$ are expanded, the collected coefficient of
$P^m$ is
\begin{equation}\label{eq:lambda-bm}
\lambda_{b,m}
:= (-1)^m \sum_{j = m}^{4} c_{b,j} \binom{j}{m}
\qquad (0 \le m \le 4),
\end{equation}
which is to say
\begin{equation}\label{eq:collected-polynomial}
\sum_{j=0}^{4} c_{b,j} (1 - P)^{j}
= \sum_{m=0}^{4} \lambda_{b,m} P^{m}.
\end{equation}
For $b = 2$ one recovers the values
$(\lambda_{2,0}, \dots, \lambda_{2,4})
= \big( \tfrac{315}{128}, -\tfrac{105}{32}, \tfrac{189}{64},
-\tfrac{45}{32}, \tfrac{35}{128} \big)$
of Section~\ref{subsec:half}; for $b = 3$,
\[
(\lambda_{3,0}, \dots, \lambda_{3,4})
= \Big( \frac{455}{243}, -\frac{455}{243}, \frac{130}{81},
-\frac{182}{243}, \frac{35}{243} \Big);
\]
and for $b = 1$, where all $c_{1,j} = 1$,
\[
(\lambda_{1,0}, \dots, \lambda_{1,4}) = (5, -10, 10, -5, 1),
\]
the binomial coefficients of Heath-Brown's classical identity.

\subsubsection*{The negative rational identity, including the endpoint}
Put
\begin{equation}\label{eq:theta-exponent}
\theta = \frac hb,
\qquad
A(s) = \zeta(s)^{-\theta},
\end{equation}
and define the short Dirichlet polynomial
\begin{equation}\label{eq:F-negative-def}
F(s) = \sum_{n \le V} d_{-h/b}(n)\, n^{-s} .
\end{equation}
The notation $n \le V$ means $n \le \lfloor V \rfloor$ when $V$ is
not an integer.

\begin{proposition}[Fifth-order identity for $\zeta^{-h/b}$]\label{prop:negative-identity}
Let $1 \le h \le b$ and $(h, b) = 1$. With $F$ as in
\eqref{eq:F-negative-def}, we have, for coefficients of $n^{-s}$
with $n \le V^5$,
\begin{equation}\label{eq:negative-identity}
\zeta(s)^{-h/b}
= \sum_{m=0}^{4} \lambda_{b,m}\, \zeta(s)^{hm} F(s)^{bm + 1} .
\end{equation}
\end{proposition}

\begin{proof}
Since $A^b \zeta^h = 1$, define
\begin{equation}\label{eq:H-negative-def}
H = 1 - \zeta^h F^b .
\end{equation}
Write $E = A - F$. The series $E$ is supported on integers $> V$.
Moreover
\[
H = \zeta^h A^b - \zeta^h F^b = \zeta^h (A^b - F^b)
= \zeta^h E \sum_{u=0}^{b-1} A^{b-1-u} F^{u} .
\]
Every factor other than $E$ is supported on the positive integers,
so $H$ is supported on integers $> V$, hence on integers
$n \ge \lfloor V \rfloor + 1$. Consequently every nonzero coefficient
of $H^5$ occurs at an integer
$n \ge (\lfloor V \rfloor + 1)^5 > V^5$.

Let
\[
C_b(t) = \sum_{j \ge 0} c_{b,j} t^{j} \in \C[[t]],
\qquad
B = C_b(H) = \sum_{j \ge 0} c_{b,j} H^{j} .
\]
Equality in $\C[[t]]$ is coefficientwise, and no convergence is
imposed. The substitution is legitimate in $\A$ because $H$ has zero
coefficient at $n = 1$: for a fixed integer $n$, the coefficient of
$n^{-s}$ in $H^j$ vanishes for all sufficiently large $j$ (indeed for $j > \log_2 n$, since $H$ has support in $[2, \infty)$
and hence $H^j$ in $[2^j, \infty)$), so the coefficient of $n^{-s}$ in $C_b(H)$
is a finite sum. Equivalently, the rule $t \mapsto H$ defines a ring
homomorphism from $\C[[t]]$ into $\A$: it respects products as well
as sums, since at each fixed $n$ all the rearrangements involved are
finite.

We claim that
\[
(1 - t)\, C_b(t)^{b} = 1 \qquad \text{in } \C[[t]].
\]
Indeed, by Lemma~\ref{lem:binom}(i), $C_b$ is the Taylor series at
$0$ of the function $F_{1/b}(t) = (1-t)^{-1/b}$ of that lemma; by
Lemma~\ref{lem:binom}(ii), $F_{1/b}^{\,b} = F_1$ and
$(1 - t) F_1(t) = 1$ on the unit disc; and by
Lemma~\ref{lem:binom}(iv) these relations between holomorphic
functions pass verbatim to their Taylor series in $\C[[t]]$.
Applying the homomorphism $t \mapsto H$ then gives
\[
(1 - H)\, B^{b} = 1 \qquad \text{in } \A .
\]

Put $G = FB$. Since $1 - H = \zeta^h F^b$, we have
\[
G^{b} \zeta^{h} = F^{b} B^{b} \zeta^{h} = (1 - H) B^{b} = 1
= A^{b} \zeta^{h} .
\]
The series $\zeta^h$ is nonzero, so cancellation in the integral
domain $\A$ (Lemma~\ref{lem:domain}) gives $G^b = A^b$. The
coefficients of $G$ and $A$ at $n = 1$ are both $1$, and therefore
Lemma~\ref{lem:root} gives $G = A$.

Since every $H^j$ with $j \ge 5$ has no coefficients at integers
$n \le V^5$, the coefficient of $n^{-s}$ for $n \le V^5$ in $B$
is obtained by truncating after $H^4$:
\[
B = \sum_{j=0}^{4} c_{b,j} H^{j} \qquad (n \le V^5).
\]
The same truncation remains valid after multiplication by $F$:
because $F$ is supported on the positive integers, $F H^j$ has no
coefficients at integers $n \le V^5$ for every $j \ge 5$. Thus
\[
A = FB = F \sum_{j=0}^{4} c_{b,j} H^{j} \qquad (n \le V^5).
\]
Substitute $H = 1 - P$, where $P = \zeta^h F^b$, and use
\eqref{eq:collected-polynomial}. Then
\[
A = F \sum_{m=0}^{4} \lambda_{b,m} P^{m}
= \sum_{m=0}^{4} \lambda_{b,m} \zeta^{hm} F^{bm+1}
\qquad (n \le V^5),
\]
which is \eqref{eq:negative-identity}.
\end{proof}

\begin{remark}\label{rem:depth}
The algebraic construction in Proposition~\ref{prop:negative-identity}
is not specific to the depth \(5\). More generally, if \(\kappa\ge1\)
and \(H=1-\zeta^h F^b\), then \(H^\kappa\) has no coefficients at
integers \(n\le V^\kappa\). Hence the same argument gives, for
coefficients of \(n^{-s}\) with \(n\le V^\kappa\),
\[
  \zeta(s)^{-h/b}
  =
  \sum_{m=0}^{\kappa-1}
  \lambda_{b,m}^{(\kappa)}\,
  \zeta(s)^{hm}F(s)^{bm+1},
\]
where
\[
  \lambda_{b,m}^{(\kappa)}
  =
  (-1)^m\sum_{j=m}^{\kappa-1} c_{b,j}\binom jm .
\]
In the present paper we take \(\kappa=5\), so that \(V=x^{1/5}\).
This choice aligns the identity with the Type I/Type II decomposition used in
Section~\ref{sec:main-proof}: the short variables have length at most
\(x^{1/5}\), and the window-or-long-variable lemma then gives either a Type II
grouping in \(x^{2/5}<M\le x^{3/5}\) or a valid Type I alternative. Larger
values of \(\kappa\) give valid identities, but they introduce more convolution
pieces and are not needed for the estimate proved here. We do not attempt to
optimize the depth.
\end{remark}

For example, when $h/b = 1/3$ and
$F(s) = \sum_{n \le V} d_{-1/3}(n) n^{-s}$,
Proposition~\ref{prop:negative-identity} gives
\begin{equation}\label{eq:d-minus-third-example}
\zeta(s)^{-1/3}
= \frac{455}{243} F
- \frac{455}{243} \zeta F^{4}
+ \frac{130}{81} \zeta^{2} F^{7}
- \frac{182}{243} \zeta^{3} F^{10}
+ \frac{35}{243} \zeta^{4} F^{13}
\qquad (n \le V^5).
\end{equation}

\subsubsection*{The endpoint identity for $\mu$}
Let
\begin{equation}\label{eq:F-mu-def}
F_\mu(s) = \sum_{n \le V} \mu(n)\, n^{-s} .
\end{equation}
The endpoint $d_{-1} = \mu$ is obtained by taking $h = b = 1$ in
Proposition~\ref{prop:negative-identity}. Since $c_{1,j} = 1$ and
$(\lambda_{1,0}, \dots, \lambda_{1,4}) = (5, -10, 10, -5, 1)$, we
obtain
\begin{equation}\label{eq:mu-identity}
\zeta(s)^{-1}
= 5 F_\mu - 10 \zeta F_\mu^2 + 10 \zeta^2 F_\mu^3
- 5 \zeta^3 F_\mu^4 + \zeta^4 F_\mu^5
\qquad (n \le V^5).
\end{equation}
Equivalently, this is the finite geometric expansion
$F_\mu (1 + H + H^2 + H^3 + H^4)$ with $H = 1 - \zeta F_\mu$, since
the terms involving $H^5$ have no coefficients at integers $n \le V^5$.
This is the fifth-order Heath-Brown identity for the M\"obius
function \cite{HB}; see also \cite[Chapter~13, Exercise~1]{IK}.

\subsubsection*{The positive rational identity}
For $1 \le h < b$, the positive exponent is handled indirectly:
\begin{equation}\label{eq:positive-via-negative}
\zeta(s)^{h/b} = \zeta(s)\, \zeta(s)^{-(b-h)/b} .
\end{equation}
Let
\begin{equation}\label{eq:G-positive-def}
G(s) = \sum_{n \le V} d_{-(b-h)/b}(n)\, n^{-s} .
\end{equation}
Applying Proposition~\ref{prop:negative-identity} to the exponent
$-(b-h)/b$ gives an identity for $\zeta^{-(b-h)/b}$ through
coefficients $n \le V^5$. Multiplying by $\zeta$ preserves equality
through the same range. Indeed, if two Dirichlet series $U$ and $W$
have equal coefficients at every integer $m \le V^5$, then the
coefficient of $n^{-s}$ in $\zeta U$ is $\sum_{d \mid n} U(n/d)$,
and for $n \le V^5$ every divisor quotient $n/d$ is also at most
$V^5$; hence the corresponding coefficient of $\zeta W$ is the
same. Therefore, for coefficients with $n \le V^5$,
\begin{equation}\label{eq:positive-identity}
\zeta(s)^{h/b}
= \sum_{m=0}^{4} \lambda_{b,m}\,
\zeta(s)^{1 + (b-h) m}\, G(s)^{bm + 1} .
\end{equation}
For instance, when $h/b = 1/3$ and
$G(s) = \sum_{n \le V} d_{-2/3}(n) n^{-s}$,
\begin{equation}\label{eq:d-plus-third-example}
\zeta(s)^{1/3}
= \frac{455}{243} \zeta G
- \frac{455}{243} \zeta^{3} G^{4}
+ \frac{130}{81} \zeta^{5} G^{7}
- \frac{182}{243} \zeta^{7} G^{10}
+ \frac{35}{243} \zeta^{9} G^{13}
\qquad (n \le V^5);
\end{equation}
when $h/b = 1/2$, it reduces to \eqref{eq:d-plus-half-example}.

\begin{remark}
\label{rem:positive-form}
There are other finite identities for $\zeta^{h/b}$, obtained by
expanding a positive fractional power directly. The form
\eqref{eq:positive-identity} is chosen for the analytic reduction:
each summand contains at least one unrestricted factor $\zeta$, so
at least one long variable has coefficient identically $1$. The
short variables are weighted by $d_{-(b-h)/b}$, and
Lemma~\ref{lem:dz}(b) gives the uniform bound
$|d_{-(b-h)/b}(n)| \le 1$. This is the same reason that, in
Section~\ref{subsec:half}, the identity for $d_{1/2}$ was obtained
by multiplying the $d_{-1/2}$ identity by $\zeta$ rather than by a
direct expansion of $(1 - Z)^{1/2}$.
\end{remark}

\section{Proof of Theorems \ref{thm:main} and \ref{thm:mu}}\label{sec:main-proof}

\subsection{Reduction to a common multilinear estimate}\label{subsec:reduction}

Throughout this section we take
\[
V = x^{1/5},
\]
so that the identities of Section~\ref{sec:identity} hold for all
coefficients $n \le x$, and we suppose that $\alpha \in \R$ and that
$r/q$ is a reduced fraction with
\begin{equation}\label{eq:alpha-rq}
(r, q) = 1,
\qquad
\Big| \alpha - \frac rq \Big| \le \frac{1}{q^2} .
\end{equation}

The identities of Section~\ref{sec:identity} reduce both theorems to
a single analytic statement, which we now formulate. The statement
is intentionally phrased in terms of variables, not in terms of a
new named class of functions; this keeps the proof close to the
type~I and type~II sums of the source chapters
\cite[Chapter~17]{MV}, \cite[Chapter~23]{K}.

\begin{proposition}
\label{prop:common-multilinear}
Fix integers $J, L \ge 0$, not both zero, and put $K = J + L$. Let
$x \ge 3$, $V = x^{1/5}$, and let $c(n)$ be a complex sequence
supported on $n \le V$ with $|c(n)| \le 1$. Define
\begin{equation}\label{eq:common-sum}
\mathcal T_{J,L}(x, \alpha)
=
\sum_{\substack{s_1, \dots, s_J \ge 1 \\ 1 \le r_1, \dots, r_L \le V \\
s_1 \cdots s_J r_1 \cdots r_L \le x}}
c(r_1) \cdots c(r_L)\,
\e\big( \alpha\, s_1 \cdots s_J r_1 \cdots r_L \big),
\end{equation}
where the product over the $s$-variables is absent if $J = 0$, and
the product over the $r$-variables is absent if $L = 0$. Suppose
\eqref{eq:alpha-rq} holds. Then
\begin{equation}\label{eq:common-sum-bound}
\mathcal T_{J,L}(x, \alpha)
\ll_{J,L}
\big( x q^{-1/2} + x^{4/5} + x^{1/2} q^{1/2} \big)
(\log 2x)^{C_{J,L}} ,
\end{equation}
where $C_{J,L} > 0$ depends only on $J$ and $L$.
\end{proposition}

The connection with the Heath-Brown-type identities of
Section~\ref{sec:identity} is direct, and worth spelling out. Fix the term $\zeta^{hm} F^{bm+1}$ of
\eqref{eq:negative-identity}, and write $J = hm$, $L = bm + 1$, and
$c(r) = d_{-h/b}(r)$ for $r \le V$, with $c(r) = 0$ otherwise; by
Lemma~\ref{lem:dz}(b), $|c(r)| \le 1$. The coefficient of $\zeta$
at every positive integer is $1$, and the coefficient of $F$ at $r$
is $c(r)$. Multiplying out the $J + L$ factors, the coefficient of
$n^{-s}$ in $\zeta^{J} F^{L}$ is therefore
\[
\sum_{s_1 \cdots s_J\, r_1 \cdots r_L \,=\, n}
c(r_1) \cdots c(r_L),
\]
the variables $s_i$ running over all positive integers and the
variables $r_i$ confined to $[1, V]$ by the support of $c$.
Multiplying by $\e(n\alpha)$ and summing over $n \le x$ turns this
into exactly $\mathcal T_{J, L}(x, \alpha)$. Since the identity
\eqref{eq:negative-identity} holds coefficientwise at every
$n \le V^5 = x$ -- this is where the choice $V = x^{1/5}$ enters
-- we conclude that
\[
\sum_{n \le x} d_{-h/b}(n)\e(n\alpha)
= \sum_{m=0}^{4} \lambda_{b,m}\,
\mathcal T_{hm,\, bm+1}(x, \alpha),
\]
where the $\lambda_{b,m}$ are the explicit rationals of
\eqref{eq:lambda-bm}, depending only on $b$. The negative case of
Theorem~\ref{thm:main} is thus the triangle inequality applied to
this display, once Proposition~\ref{prop:common-multilinear} is
proved; the positive case and Theorem~\ref{thm:mu} follow in the
same way from \eqref{eq:positive-identity} and
\eqref{eq:mu-identity}. The formal deductions are carried out in
Section~\ref{subsec:deduction}.

The intervening subsections prove the proposition:
Section~\ref{subsec:window} localizes the variables dyadically and
establishes the combinatorial dichotomy -- either some subproduct
of the variables can be grouped into the window
$[x^{2/5}, x^{3/5}]$, or a single smooth variable is long and the
complementary product is short; Section~\ref{subsec:estimates}
supplies the bilinear estimate for the first case and the two
type~I estimates (mean-square, and pointwise with an $x^{o(1)}$
loss) for the second; and Section~\ref{subsec:common-proof}
assembles them over the three ranges of $q$ cut at $x^{3/10}$ and
$x^{7/10}$.

\subsection{Dyadic decomposition and the window-or-long-variable lemma}\label{subsec:window}

We shall repeatedly split variables into dyadic intervals. The
following elementary observation records what is lost.

\begin{lemma}\label{lem:dyadic-loss}
Let $r$ be fixed. If $n_1 \cdots n_r \le x$ with $n_i \ge 1$, then
the possible dyadic choices
\[
N_i < n_i \le 2N_i,
\qquad N_i \in \{ \tfrac12, 1, 2, 4, \dots \}
\]
are $O_r\big( (\log 2x)^r \big)$. The value $N_i = \tfrac12$ is used
only for the singleton interval $(\tfrac12, 1]$, in which the
integer variable is forced to be $n_i = 1$. In all later size
arguments we delete such forced variables before applying
logarithms or Lemma~\ref{lem:window}; deleting them changes neither
the product $n_1 \cdots n_r$ nor the value of the exponential
phase. Thus, after this deletion, every remaining dyadic parameter
satisfies $N_i \ge 1$.

Consequently a bound for each dyadic piece may be multiplied by a
power of $\log 2x$ depending only on $r$.
\end{lemma}

\begin{proof}
Each $n_i$ lies in $[1, x]$, hence has $O(\log 2x)$ dyadic ranges,
including the possible singleton range $(\tfrac12, 1]$. Since $r$
is fixed, the number of choices is $O_r((\log 2x)^r)$. If
$N_i = \tfrac12$, then the only integer in $N_i < n_i \le 2N_i$ is
$n_i = 1$, so deleting that variable is exactly harmless.
\end{proof}

In Section~\ref{subsec:common-proof} this is applied to the
variables of \eqref{eq:common-sum}: each of the $K = J + L$
variables is localized to a dyadic range, at a total cost of
$O_K\big( (\log 2x)^{K} \big)$ pieces, and it then suffices to
bound each piece with the required uniformity.

The next lemma is the combinatorial heart of the case analysis in
Section~\ref{subsec:common-proof}: for every dyadic piece that is
not trivially small, either some subcollection of the variables can
be grouped into the type~II window $(x^{2/5}, x^{3/5}]$, or a
single variable is longer than $x^{3/10}$, and then the product of
all the other variables is small enough for the type~I estimates.

\begin{lemma}
\label{lem:window}
Let $x \ge 3$ and let $N_1, \dots, N_R$ be real numbers with
$N_i \ge 1$ for every $i$, and suppose that
\begin{equation}\label{eq:window-hyp}
x^{4/5} < N_1 \cdots N_R \le x .
\end{equation}
Then at least one of the following holds:
\begin{enumerate}
\item[(i)] there is a nonempty subset
$I \subseteq \{1, \dots, R\}$ with
\[
x^{2/5} < \prod_{i \in I} N_i \le x^{3/5};
\]
\item[(ii)] $\max_i N_i > x^{3/10}$, and then
$\prod_{i \ne i_0} N_i < x^{7/10}$, where $i_0$ is a maximizing
index.
\end{enumerate}
\end{lemma}

\begin{proof}
The hypothesis $N_i \ge 1$ ensures that the logarithms below are
nonnegative. Write $u_i := \log N_i / \log x$ and
$\sigma := \sum_i u_i$. Since $N_i \ge 1$, all $u_i$ are
nonnegative. Taking logarithms in \eqref{eq:window-hyp} and
dividing by $\log x$ gives $\tfrac45 < \sigma \le 1$. Call a subset
\emph{good} if the sum of its $u_i$ lies in
$(\tfrac25, \tfrac35]$; in the original scale, a subset $I$ is good
precisely when $\prod_{i \in I} N_i = x^{\sum_{i \in I} u_i}$ lies
in the window $(x^{2/5}, x^{3/5}]$, so alternative (i) is the
existence of a good subset. Suppose there is no good subset; we
claim that $u_{i_0} := \max_i u_i > \tfrac{3}{10}$. The claim gives
(ii): $\sum_{i \ne i_0} u_i = \sigma - u_{i_0} < 1 - \tfrac{3}{10}
= \tfrac{7}{10}$.

Assume, for contradiction, that $u_i \le \tfrac{3}{10}$ for every
$i$. Call an index \emph{small} if $u_i \le \tfrac15$ and
\emph{medium} if $u_i \in (\tfrac15, \tfrac{3}{10}]$.

First, there is at most one medium index: the $u$-sum of two medium
indices lies in $(\tfrac25, \tfrac35]$, which would be a good
subset.

Second, let $s$ be the $u$-sum of all small indices, and suppose
$s > \tfrac25$. List the small indices in any order and consider
the partial sums: they increase from $0$ past $\tfrac25$ in steps
of size at most $\tfrac15$, so the first partial sum exceeding
$\tfrac25$ is at most $\tfrac25 + \tfrac15 = \tfrac35$, and the
corresponding initial segment is a good subset. Hence
$s \le \tfrac25$.

Third, suppose a medium index exists, with value $u$, and suppose
$u + s > \tfrac25$. The partial sums
$u,\ u + u_{j_1},\ u + u_{j_1} + u_{j_2}, \dots$ over the small
indices start at $u \le \tfrac{3}{10} < \tfrac25$ and increase in
steps of size at most $\tfrac15$, so, exactly as before, the first
partial sum exceeding $\tfrac25$ yields a good subset. Hence
$u + s \le \tfrac25$.

Combining: $\sigma$ equals $s$ or $u + s$ according as no medium
index exists or one does, so $\sigma \le \tfrac25$, contradicting
$\sigma > \tfrac45$. This proves the claim, and the lemma.
\end{proof}

\subsection{Type I and type II estimates for grouped variables}\label{subsec:estimates}

In the type~I case of the dichotomy of Lemma~\ref{lem:window}, the
long smooth variable will be summed against a coefficient obtained
by grouping all the remaining variables; in the type~II case, the
variables will be grouped into two blocks around the window. This
subsection derives, from the raw estimates of
Section~\ref{subsec:expsums}, the three statements in the exact
forms consumed by Section~\ref{subsec:common-proof}: a mean-square
type~I estimate, a pointwise type~I estimate with an $x^{o(1)}$
loss, and a bilinear type~II estimate for divisor-bounded grouped
coefficients. None of them requires information about divisor
functions in arithmetic progressions. In particular, we do not use
a divisor-bounded type~I estimate with only logarithmic loss; such
a statement would be false in general, because the divisor weights
can concentrate on residue classes where $\norm{m\alpha}$ is
unusually small.

\begin{lemma}[Type I estimate, mean-square form]\label{lem:type-I-ms}
Fix $K \ge 1$. Let $x \ge 3$, $q \le x$, and
$\tfrac12 \le R \le x$. Let $A_m$ be complex numbers supported on
$m \le R$ with $|A_m| \le \tau_K(m)$, and let $I_m$ be any interval
of integers contained in $[1, x/m]$. Suppose \eqref{eq:alpha-rq}
holds. Then
\begin{equation}\label{eq:type-I-ms}
\sum_{m \le R} A_m \sum_{n \in I_m} \e(\alpha m n)
\ll_K
\Big( \frac{x}{\sqrt q} + \sqrt{xR} + \sqrt{xq} \Big)
(\log 2x)^{C_K},
\end{equation}
where $C_K > 0$ depends only on $K$.
\end{lemma}

\begin{proof}
For a fixed $m$, the inner sum runs over an interval containing at
most $x/m$ integers, so Lemma~\ref{lem:geom} (applied with
$\beta = m\alpha$, and with the convention stated there when
$\norm{m\alpha} = 0$) gives
\begin{equation}\label{eq:type-I-geometric-reduction}
\Big| \sum_{n \in I_m} \e(\alpha m n) \Big|
\le 2 \min\Big( \frac{x}{m},\ \frac{1}{\norm{m\alpha}} \Big).
\end{equation}
Split the range of $m$ into the dyadic pieces $M < m \le 2M$ with $M \in \{ \tfrac12, 1, 2, 4, \dots \}$, $M \le R$; there are
$O(\log 2R)$ of them. The piece $M = \tfrac12$, if present, consists
of the single term $m = 1$, whose contribution is at most
$2 \min(x, \norm{\alpha}^{-1})$ by
\eqref{eq:type-I-geometric-reduction}; for $q = 1$ this is at most
$2x = 2 x q^{-1/2}$, while for $q \ge 2$ the hypothesis
\eqref{eq:alpha-rq} gives
$\norm{\alpha} \ge 1/q - 1/q^{2} \ge 1/(2q)$, so the contribution is
at most $4q \le 4 \sqrt{xq}$, using $q \le x$. We may therefore
assume $M \ge 1$ in what follows, so that Lemma~\ref{lem:minsq}
applies on every piece. On each piece, the Cauchy-Schwarz inequality
gives
\[
\sum_{M < m \le 2M} |A_m| \min\Big( \frac{x}{m},
\frac{1}{\norm{m\alpha}} \Big)
\le
\Big( \sum_{M < m \le 2M} |A_m|^2 \Big)^{1/2}
\Big( \sum_{M < m \le 2M} \min\Big( \frac{x}{m},
\frac{1}{\norm{m\alpha}} \Big)^{2} \Big)^{1/2} .
\]
For the first factor, $|A_m| \le \tau_K(m)$ and
Lemma~\ref{lem:tau-meansq} give
\[
\sum_{M < m \le 2M} |A_m|^2
\le \sum_{m \le 2M} \tau_K(m)^2
\ll_K M (\log 4M)^{K^2 - 1}
\ll_K M (\log 2x)^{K^2},
\]
since $M \le x$. For the second factor, Lemma~\ref{lem:minsq}
(applied with $T = M$ and $\kappa = 1$) gives
\[
\sum_{M < m \le 2M} \min\Big( \frac{x}{m},
\frac{1}{\norm{m\alpha}} \Big)^{2}
\ll \frac{x}{M} \Big( \frac{x}{q} + M + q \Big) \log 2Mq
\ll \frac{x}{M} \Big( \frac{x}{q} + M + q \Big) \log 2x ,
\]
since $M \le x$ and $q \le x$ imply $\log 2Mq \ll \log 2x$.
Multiplying the two factors, the contribution of the piece is
\[
\ll_K \Big( x \Big( \frac{x}{q} + M + q \Big) \Big)^{1/2}
(\log 2x)^{(K^2 + 2)/2}
\le
\Big( \frac{x}{\sqrt q} + \sqrt{xM} + \sqrt{xq} \Big)
(\log 2x)^{(K^2 + 2)/2},
\]
using $\sqrt{u + v + w} \le \sqrt u + \sqrt v + \sqrt w$. Summing
over the $O(\log 2R)$ dyadic pieces and using $M \le R$ gives
\eqref{eq:type-I-ms} with $C_K = (K^2 + 4)/2$.
\end{proof}

\begin{lemma}[Type I estimate, pointwise $x^{o(1)}$ form]\label{lem:type-I-pw}
Fix $K \ge 1$. Under the hypotheses of
Lemma~\ref{lem:type-I-ms},
\begin{equation}\label{eq:type-I-pw}
\sum_{m \le R} A_m \sum_{n \in I_m} \e(\alpha m n)
\ll_K
\exp\Big( C_K' \frac{\log x}{\log\log x} \Big)
\Big( \frac{x}{q} + R + q \Big) \log 2x,
\end{equation}
where $C_K' > 0$ depends only on $K$.
\end{lemma}

\begin{proof}
If $R < 1$ the sum on the left of \eqref{eq:type-I-pw} is empty;
assume $R \ge 1$, so that Lemma~\ref{lem:typeI} applies with $T = R$
below. By Lemma~\ref{lem:tau-max}, $\max_{m \le x} \tau_K(m) \le K \exp\big( C_1 (K - 1) \log x /
\log\log x \big)$, the factor $K$ covering $m \le 2$. Hence, taking
absolute values inside and using
\eqref{eq:type-I-geometric-reduction},
\[
\Big| \sum_{m \le R} A_m \sum_{n \in I_m} \e(\alpha m n) \Big|
\le 2 \max_{m \le x} \tau_K(m)
\sum_{m \le R} \min\Big( \frac{x}{m},
\frac{1}{\norm{m\alpha}} \Big) .
\]
The remaining sum is exactly the object of Lemma~\ref{lem:typeI},
which, applied with $T = R$, bounds it by
$\ll (x/q + R + q) \log 2Rq \ll (x/q + R + q) \log 2x$, since
$\log 2Rq \le \log 2x^2 \le 2 \log 2x$. Combining the two bounds
gives \eqref{eq:type-I-pw}.
\end{proof}

\begin{remark}\label{rem:no-false-type-I}
Lemma~\ref{lem:type-I-pw} is deliberately not a divisor-weighted
version of Lemma~\ref{lem:typeI} with only a power of $\log x$
lost. Such a logarithmic-loss statement is false in general:
divisor weights can concentrate on integers with large common
factors with the denominator $q$, exactly where $\norm{m\alpha}$
can be anomalously small. The pointwise estimate above therefore
pays the honest maximal-order cost
\[
\exp\Big( C_K' \frac{\log x}{\log\log x} \Big) = x^{o(1)} .
\]
In Section~\ref{subsec:common-proof} this estimate is used only in
the range $x^{3/10} < q < x^{7/10}$, where the unweighted type~I
quantity is at most $x^{7/10}$ and there is a fixed exponent gap
before $x^{4/5}$. In the outer ranges of $q$ we avoid pointwise
divisor weights and use the mean-square estimate of
Lemma~\ref{lem:type-I-ms}.
\end{remark}

\begin{remark}\label{rem:type-I-use}
In Proposition~\ref{prop:common-multilinear} the number of
variables is bounded by the fixed integer $K$. The type~I case will
produce a grouped coefficient supported on
$R \le 2^{K} x^{7/10}$. In the outer ranges $q \le x^{3/10}$ and
$q \ge x^{7/10}$ we use the mean-square type~I estimate: its middle
term satisfies
\[
\sqrt{xR} \le 2^{K/2} x^{17/20},
\]
which is dominated by $x q^{-1/2}$ in the first outer range and by
$x^{1/2} q^{1/2}$ in the second outer range, after the constants
depending on $K$ are absorbed.

In the inner range $x^{3/10} < q < x^{7/10}$ we use the pointwise
form only in the deliberately weak sense explained in
Remark~\ref{rem:no-false-type-I}. Then $x/q$, $R$, and $q$ are all
$\ll_K x^{7/10}$, and the maximal order of $\tau_K$ costs only
$x^{o_K(1)}$. The fixed exponent gap $4/5 - 7/10 = 1/10$ absorbs
this loss. No logarithmic-loss divisor-weighted type~I estimate is
used anywhere.
\end{remark}

The bilinear input is Lemma~\ref{lem:typeII} of the toolbox; what
the case analysis needs is its consequence for divisor-bounded
grouped coefficients whose supports sit near the window.

\begin{corollary}[Type II estimate with divisor-bounded grouped coefficients]\label{cor:type-II-divisor}
Fix $K \ge 1$ and $D_0 \ge 1$. Suppose $|A_m| \le \tau_K(m)$ and
$|B_n| \le \tau_K(n)$, with supports $m \le M$, $n \le N$, and
suppose \eqref{eq:alpha-rq} holds. If $M N \le D_0 x$, then
\begin{equation}\label{eq:type-II-divisor-bound}
\sum_{mn \le x} A_m B_n\, \e(\alpha m n)
\ll_{K, D_0}
x^{1/2} \Big( q + M + N + \frac{x}{q} \Big)^{1/2}
(\log 2x)^{D_K} ,
\end{equation}
where $D_K > 0$ depends only on $K$. In particular, if
\begin{equation}\label{eq:type-II-range}
x^{2/5} \le M \le D_0 x^{3/5},
\qquad
N \le D_0 x^{3/5},
\end{equation}
then
\begin{equation}\label{eq:type-II-final-range}
\sum_{mn \le x} A_m B_n\, \e(\alpha m n)
\ll_{K, D_0}
\big( x q^{-1/2} + x^{4/5} + x^{1/2} q^{1/2} \big)
(\log 2x)^{D_K} .
\end{equation}
\end{corollary}
\begin{proof}
Regard $A$ and $B$ as arithmetic functions supported on $[1, M]$
and $[1, N]$ respectively, so that the sum in question is
$\sum_{n' \le x} (A * B)(n') \e(\alpha n')$. Since
$|A_m| \le \tau_K(m)$, Lemma~\ref{lem:tau-meansq} gives
\[
\|A\|_2^2 = \sum_{m \le M} |A_m|^2
\le \sum_{m \le M} \tau_K(m)^2
\ll_K M (\log 2M)^{K^2 - 1}.
\]
Since the sum in question is trivially zero if either support is
empty, we may assume $M, N \ge 1$; then $MN \le D_0 x$ gives
$M, N \le D_0 x$, so $(\log 2M)^{K^2-1}$ and $(\log 2N)^{K^2-1}$ are
$\ll_{K, D_0} (\log 2x)^{K^2 - 1}$. Hence
$\|A\|_2 \ll_{K, D_0} M^{1/2} (\log 2x)^{(K^2-1)/2}$, and likewise
$\|B\|_2 \ll_{K, D_0} N^{1/2} (\log 2x)^{(K^2-1)/2}$. Inserting these
bounds into Lemma~\ref{lem:typeII} (with $y = M$, $z = N$) gives
\[
\sum_{mn \le x} A_m B_n\, \e(\alpha m n)
\ll_{K, D_0} (MN)^{1/2} \Big( q + M + N + \frac{MN}{q} \Big)^{1/2}
(\log 2q)^{1/2} (\log 2x)^{K^2 - 1} .
\]
Since $MN \le D_0 x$, the factor $(MN)^{1/2}$ is
$\ll_{D_0} x^{1/2}$ and the term $MN/q$ inside the square root is
$\ll_{D_0} x/q$. If $q > x$, then by the identity
$\tau_K * \tau_K = \ID^{*2K} = \tau_{2K}$ and
Lemma~\ref{lem:tau-mean} the sum is at most
$\sum_{n' \le x} (\tau_K * \tau_K)(n')
= \sum_{n' \le x} \tau_{2K}(n')
\ll_K x (\log 2x)^{2K - 1}$, while
$x^{1/2} q^{1/2} \ge x$, so \eqref{eq:type-II-divisor-bound} holds
trivially. We may therefore assume $q \le x$, so that
$\log 2q \le \log 2x$. This proves \eqref{eq:type-II-divisor-bound},
after enlarging $D_K$ if necessary.
If \eqref{eq:type-II-range} holds, then taking square roots
termwise,
\[
x^{1/2} M^{1/2} \ll_{D_0} x^{1/2} x^{3/10} = x^{4/5},
\qquad
x^{1/2} N^{1/2} \ll_{D_0} x^{4/5},
\]
while
$x^{1/2} (x/q)^{1/2} = x q^{-1/2}$ and
$x^{1/2} q^{1/2}$ is already of the required shape. This proves
\eqref{eq:type-II-final-range}.
\end{proof}

\subsection{Proof of the common multilinear estimate}\label{subsec:common-proof}

We now prove Proposition~\ref{prop:common-multilinear}. The proof is
written in full because this is where the false pointwise type~I
estimate of Remark~\ref{rem:no-false-type-I} must be avoided.

\begin{proof}[Proof of Proposition~\ref{prop:common-multilinear}]
First suppose $q > x$. Since $|c(r)| \le 1$, the absolute value of
$\mathcal T_{J,L}(x, \alpha)$ is at most the number of tuples
$(s_1, \dots, s_J, r_1, \dots, r_L)$ with product at most $x$, with
some restrictions on the $r$-variables; ignoring those restrictions
only increases the count, and the number of tuples with product
equal to $n$ is $\tau_K(n)$. Hence
\[
|\mathcal T_{J,L}(x, \alpha)| \le \sum_{n \le x} \tau_K(n) .
\]
By Cauchy's inequality and Lemma~\ref{lem:tau-meansq},
\[
\sum_{n \le x} \tau_K(n)
\le x^{1/2} \Big( \sum_{n \le x} \tau_K(n)^2 \Big)^{1/2}
\ll_K x (\log 2x)^{C_K} .
\]
Since $q > x$, we have $x \le x^{1/2} q^{1/2}$, and
\eqref{eq:common-sum-bound} follows. Henceforth assume
\begin{equation}\label{eq:q-le-x-common}
q \le x .
\end{equation}

We split every variable into dyadic intervals
\[
N < n \le 2N, \qquad N \in \{ \tfrac12, 1, 2, 4, \dots \} .
\]
By Lemma~\ref{lem:dyadic-loss}, there are $O_K((\log 2x)^K)$
choices. If a variable lies in $(\tfrac12, 1]$, it is forced to
equal $1$, and we delete it before applying any logarithmic size
argument. After deletion and relabelling, a fixed dyadic piece has
the form
\begin{equation}\label{eq:common-dyadic-piece}
\mathcal S =
\sum_{\substack{n_i \in I_i \ (1 \le i \le R) \\
n_1 \cdots n_R \le x}}
c_1(n_1) \cdots c_R(n_R)\e(\alpha\, n_1 \cdots n_R),
\end{equation}
where $0 \le R \le K$, each $I_i$ is a dyadic interval
$N_i < n \le 2N_i$ with $N_i \ge 1$, each $|c_i(n)| \le 1$, and
every remaining variable coming from one of the $r$-variables has
$N_i \le V = x^{1/5}$.

Before estimating $\mathcal S$, let us record precisely how these
pieces reassemble into the object of the proposition. The sum
$\mathcal T_{J,L}(x, \alpha)$ is the sum, over all
$O_K((\log 2x)^K)$ admissible choices of the dyadic parameters, of
the corresponding restricted sums. In a given piece, the
coefficient attached to a variable is $c_i = \ID$, identically
$1$, if the variable is one of the $s$-variables, and $c_i = c$ if
it is one of the $r$-variables; a piece in which an $r$-variable
has $N_i \ge V$ vanishes identically, because $c$ is supported on
$[1, V]$, and this is why every surviving $r$-parameter satisfies
$N_i \le V$. A deleted variable -- one forced to equal $1$ --
contributes to its piece the constant factor $c_i(1)$, of modulus
at most $1$, which we pull out of the sum. Consequently
\[
|\mathcal T_{J,L}(x, \alpha)|
\ll_K (\log 2x)^{K} \max_{\text{pieces}} |\mathcal S| ,
\]
so that a bound for $|\mathcal S|$, uniform over the pieces, of
the shape $(x q^{-1/2} + x^{4/5} + x^{1/2} q^{1/2}) (\log 2x)^{C}$
yields \eqref{eq:common-sum-bound} with $C_{J,L} = C + K$; this
reassembly is the final line of the proof, and the body of the
proof is devoted to the uniform bound for a single piece. Note
that everything the argument uses about the coefficients $c_i$ is
that $|c_i| \le 1$ and that the $r$-parameters sit below $V$; in
particular, $R$ may be smaller than $K$ after the deletions, and a
piece does not remember which of the identities of
Section~\ref{sec:identity} it came from. This is the sense in
which the estimate is common to all of them.

If $R = 0$, then $\mathcal S = O(1)$, so there is nothing to prove.
Assume $R \ge 1$ and set
\[
P = N_1 \cdots N_R .
\]
The number of tuples in the dyadic piece is $O_K(P)$: each $n_i$
runs over at most $N_i + 1 \le 2N_i$ integers, and
$\prod_i 2N_i = 2^R P$ with $R \le K$. If $P \le x^{4/5}$, the
trivial estimate -- each tuple contributing at most $1$ in
absolute value -- gives $\mathcal S \ll_K x^{4/5}$, which is
acceptable. We therefore assume
\begin{equation}\label{eq:P-large-common}
P > x^{4/5} .
\end{equation}
If the dyadic piece is empty, it contributes $0$. Otherwise some
tuple in the piece satisfies $n_1 \cdots n_R \le x$, and since
$n_i > N_i$, we get $P < x$. Thus
\begin{equation}\label{eq:P-window-common}
x^{4/5} < P \le x,
\end{equation}
and Lemma~\ref{lem:window} applies.

\medskip
\noindent\emph{The type II case.}
Suppose there is a nonempty subset $I \subseteq \{1, \dots, R\}$
such that
\begin{equation}\label{eq:common-window}
x^{2/5} < \prod_{i \in I} N_i \le x^{3/5} .
\end{equation}
Group the variables with indices in $I$ into one variable $m$, and
all remaining variables into one variable $n$. Then
\[
\mathcal S = \sum_{mn \le x} A_m B_n\e(\alpha m n),
\]
where $A_m = \sum \prod_{i \in I} c_i(n_i)$, the sum running over
tuples $(n_i)_{i \in I}$ with $n_i \in I_i$ and
$\prod_{i \in I} n_i = m$, and similarly for $B_n$. Since each
$|c_i| \le 1$, the coefficient $A_m$ is bounded by the number of
such tuples, which is at most $\tau_{|I|}(m) \le \tau_K(m)$ --
the last inequality because $\tau_j(m) \le \tau_K(m)$ for
$j \le K$, as one sees by taking the last $K - j$ factors equal to
$1$ in the definition of $\tau_K$. Hence
\begin{equation}\label{eq:common-coeff-divisor}
|A_m| \le \tau_K(m), \qquad |B_n| \le \tau_K(n) .
\end{equation}
Moreover $A_m$ is supported on
\[
m \le M := \prod_{i \in I} 2N_i \le 2^K x^{3/5},
\]
and $B_n$ is supported on
\[
n \le N := \prod_{i \notin I} 2N_i
\le 2^K \frac{P}{\prod_{i \in I} N_i}
< 2^K \frac{x}{x^{2/5}} = 2^K x^{3/5},
\]
using $P \le x$ and the lower bound in \eqref{eq:common-window}.
Note also that the complementary index set is nonempty and carries
genuine mass: by \eqref{eq:P-large-common} and the upper bound in
\eqref{eq:common-window},
$\prod_{i \notin I} N_i = P / \prod_{i \in I} N_i > x^{4/5} /
x^{3/5} = x^{1/5}$, so neither side of the bilinear form is
degenerate.
Also $MN \le 2^K P \le 2^K x$. Corollary~\ref{cor:type-II-divisor},
with $D_0 = 2^K$, gives
\begin{equation}\label{eq:common-type-II-bound}
\mathcal S
\ll_K
\big( x q^{-1/2} + x^{4/5} + x^{1/2} q^{1/2} \big)
(\log 2x)^{C_K} .
\end{equation}

\medskip
\noindent\emph{The type I case.}
Suppose no subset satisfies \eqref{eq:common-window}. By
Lemma~\ref{lem:window}, a longest variable, say $n_{i_0}$,
satisfies
\begin{equation}\label{eq:common-long}
N_{i_0} > x^{3/10},
\qquad
\prod_{i \ne i_0} N_i < x^{7/10} .
\end{equation}
A variable coming from an $r$-factor has dyadic parameter at most
$V = x^{1/5}$, and $x^{3/10} > x^{1/5}$ for $x > 1$, so $n_{i_0}$
cannot come from an $r$-factor. Thus $n_{i_0}$ comes from a
$\zeta$-factor, and its coefficient is identically $1$. In
particular, if $J = 0$, this type~I alternative cannot occur, and
the window alternative must have occurred instead.

Group all other variables into $m$, and rename the long variable
$\ell$. Then the dyadic piece becomes
\begin{equation}\label{eq:common-type-I-piece}
\mathcal S = \sum_m A_m \sum_{\ell \in J_m} \e(\alpha m \ell),
\end{equation}
where $J_m = (N_{i_0}, 2N_{i_0}] \cap [1, x/m]$ is an interval of
integers contained in $[1, x/m]$. The coefficient $A_m$ is formed
from at most $K - 1$ variables with coefficients bounded by $1$,
so, exactly as in the type~II case,
\begin{equation}\label{eq:common-type-I-coeff}
|A_m| \le \tau_K(m),
\end{equation}
and, by \eqref{eq:common-long}, it is supported on
\begin{equation}\label{eq:R0-common}
m \le R_0 := 2^K x^{7/10} .
\end{equation}
For $x$ below a constant depending only on $K$, the desired
estimate is absorbed into the implied constant. We may therefore
assume $R_0 \le x$.

If $q \le x^{3/10}$ or $q \ge x^{7/10}$,
Lemma~\ref{lem:type-I-ms} gives
\[
\mathcal S
\ll_K
\Big( \frac{x}{\sqrt q} + \sqrt{x R_0} + \sqrt{xq} \Big)
(\log 2x)^{C_K} .
\]
Here $\sqrt{x R_0} \le 2^{K/2} x^{17/20}$, and the two outer ranges
absorb this middle term: if $q \le x^{3/10}$, then
$x q^{-1/2} \ge x^{1 - 3/20} = x^{17/20}$; if $q \ge x^{7/10}$,
then $x^{1/2} q^{1/2} \ge x^{1/2 + 7/20} = x^{17/20}$. Hence the
desired bound follows in the two outer ranges of $q$.

It remains to handle
\begin{equation}\label{eq:middle-q-common}
x^{3/10} < q < x^{7/10} .
\end{equation}
This is the only point in the paper where the pointwise estimate
Lemma~\ref{lem:type-I-pw} is used, and we emphasize that its loss
is the $x^{o(1)}$ maximal-order cost, not a logarithmic
divisor-weight loss. By Lemma~\ref{lem:type-I-pw},
\[
\mathcal S
\ll_K
\exp\Big( C_K' \frac{\log x}{\log\log x} \Big)
\Big( \frac{x}{q} + R_0 + q \Big) \log 2x .
\]
In the range \eqref{eq:middle-q-common}, the three quantities
$x/q$, $R_0$, and $q$ are all $\ll_K x^{7/10}$: the first because
$q > x^{3/10}$, the second by \eqref{eq:R0-common}, the third
because $q < x^{7/10}$. Hence, absorbing the factor $\log 2x$
into the exponential by enlarging the constant to $C_K''$,
\[
\mathcal S
\ll_K x^{7/10} \exp\Big( C_K'' \frac{\log x}{\log\log x} \Big) .
\]
For $x$ sufficiently large in terms of $K$, the exponential factor
is at most $x^{1/10}$, so that
$\mathcal S \ll_K x^{7/10} \cdot x^{1/10} = x^{4/5}$; the remaining
bounded range of $x$ is again absorbed into the implied constant.
This is the required bound in the middle range.

Combining the trivial case, the type~II case, and the type~I case
proves the required bound for each dyadic piece. Multiplying by the
$O_K((\log 2x)^K)$ dyadic choices proves
\eqref{eq:common-sum-bound}.
\end{proof}

\subsection{Deduction of Theorems \ref{thm:main} and \ref{thm:mu}, and of Corollary \ref{cor:minor}}\label{subsec:deduction}

\begin{proof}[Proof of Theorem \ref{thm:main}]
Consider first the negative exponent $z = -a/b$. Taking $h = a$ in
Proposition~\ref{prop:negative-identity} and carrying out the
computation of Section~\ref{subsec:reduction}, we have
\[
\sum_{n \le x} d_{-a/b}(n)\e(n\alpha)
= \sum_{m=0}^{4} \lambda_{b,m}\,
\mathcal T_{am,\, bm+1}(x, \alpha),
\]
where the short coefficient is $c(n) = d_{-a/b}(n)$ for
$n \le V$, so that $|c(n)| \le 1$ by Lemma~\ref{lem:dz}(b). For
each $0 \le m \le 4$, Proposition~\ref{prop:common-multilinear},
applied with $J = am$ and $L = bm + 1$ -- integers bounded in
terms of $a$ and $b$ only -- gives
\[
\mathcal T_{am,\, bm+1}(x, \alpha)
\ll_{a,b}
\big( x q^{-1/2} + x^{4/5} + x^{1/2} q^{1/2} \big)
(\log 2x)^{C_{am,\, bm+1}} .
\]
\begin{sloppypar}
There are five terms, and the coefficients $\lambda_{b,m}$ depend
only on $b$; taking
$C(a,b) = \max_{0 \le m \le 4} C_{am, bm+1}$ proves the negative case of \eqref{eq:main}; bounds from
Section~\ref{sec:tools} carry $\log 2x$, and $\log 2x \asymp \log x$
for $x \ge 3$, so the two normalizations are interchangeable after
enlarging $C(a,b)$.
\end{sloppypar}

For the positive exponent $z = a/b$, use the identity
\eqref{eq:positive-identity} with $h = a$: for coefficients with
$n \le x$,
\[
\zeta^{a/b}
= \sum_{m=0}^{4} \lambda_{b,m}\,
\zeta^{1 + (b-a)m}\, G^{bm+1},
\qquad
G(s) = \sum_{n \le V} d_{-(b-a)/b}(n)\, n^{-s} .
\]
Taking coefficients, multiplying by $\e(n\alpha)$, and summing over
$n \le x$, the $m$-th term is exactly
$\lambda_{b,m}\, \mathcal T_{J, L}(x, \alpha)$ with
\[
J = 1 + (b - a) m,
\qquad
L = bm + 1,
\qquad
c(n) = d_{-(b-a)/b}(n) \quad (n \le V) .
\]
Since $0 < (b-a)/b < 1$, Lemma~\ref{lem:dz}(b) again gives
$|c(n)| \le 1$; note that here $J \ge 1$ for every $m$, including
$m = 0$, reflecting the extra factor $\zeta$ in
\eqref{eq:positive-via-negative}. The integers $J, L$ are bounded
in terms of $a, b$, so every term is
$\ll_{a,b} ( x q^{-1/2} + x^{4/5} + x^{1/2} q^{1/2} )
(\log 2x)^{C}$ by Proposition~\ref{prop:common-multilinear}, and
the positive case of \eqref{eq:main} follows as before, after
enlarging $C(a,b)$ if necessary.
\end{proof}

\begin{proof}[Proof of Theorem \ref{thm:mu}]
Use the finite identity \eqref{eq:mu-identity}. Its five pieces are
\[
F_\mu,
\qquad
\zeta F_\mu^2,
\qquad
\zeta^2 F_\mu^3,
\qquad
\zeta^3 F_\mu^4,
\qquad
\zeta^4 F_\mu^5 .
\]
For the term $\zeta^m F_\mu^{m+1}$,
Proposition~\ref{prop:common-multilinear} applies with
\[
J = m,
\qquad
L = m + 1,
\qquad
c(n) = \mu(n) \quad (n \le V) .
\]
The bound $|\mu(n)| \le 1$ is immediate. Since $0 \le m \le 4$, the
number of variables is at most $9$, and the implied constants are
absolute. Summing the five estimates with the coefficients
$(5, -10, 10, -5, 1)$ of \eqref{eq:mu-identity} proves the required
estimate with $\log 2x$ in place of $\log x$, and with
$C_0 = \max_{0 \le m \le 4} C_{m,\, m+1}$ an absolute constant.
Since $\log 2x \asymp \log x$ for $x \ge 3$, the stated form of
Theorem~\ref{thm:mu} follows after enlarging $C_0$ if necessary.
\end{proof}

\begin{proof}[Proof of Corollary \ref{cor:minor}]
Let $C = C(a,b)$ be as in Theorem~\ref{thm:main}, let $A > 0$, and
let $r/q$ be a reduced approximation to $\alpha$ as in the
hypothesis, with $(\log x)^{2A + 2C} \le q \le x / (\log x)^{2A + 2C}$. Since
$x \ge 3$, we have $\log 2x \le 2 \log x$, so
$(\log 2x)^{C} \ll_C (\log x)^{C}$. We bound the three terms of
\eqref{eq:main} in turn. First,
\[
\frac{x}{\sqrt q}\, (\log x)^{C}
\le \frac{x (\log x)^{C}}{(\log x)^{A + C}}
= \frac{x}{(\log x)^{A}} ,
\]
by the lower bound on $q$. Next,
\[
\sqrt{xq}\, (\log x)^{C}
\le x^{1/2} \cdot \frac{x^{1/2}}{(\log x)^{A + C}} \cdot
(\log x)^{C}
= \frac{x}{(\log x)^{A}} ,
\]
by the upper bound on $q$. Finally,
$x^{4/5} (\log x)^{C} \le x / (\log x)^{A}$ as soon as
$x^{1/5} \ge (\log x)^{A + C}$, which holds for
$x \ge x_0(a, b, A)$. Combining the three bounds with
\eqref{eq:main} gives
$\sum_{n \le x} d_{\pm a/b}(n) \e(n\alpha)
\ll_{a,b,A} x (\log x)^{-A}$, as claimed.
\end{proof}

\section{Proofs of the applications}\label{sec:applications}

\subsection{The Selberg-Delange input}\label{subsec:LSD}

Following \cite[Ch.~II.5]{Te}, we say that a Dirichlet series
$F(s) = \sum_{n \ge 1} a_n n^{-s}$, convergent for $\real s > 1$, has
\emph{property $\mathcal P(w; c_0, \delta, M)$} (where $w \in \C$ and
$c_0, \delta, M > 0$ with $\delta \le 1$) if the function
\[
G(s; w) := F(s)\, \zeta(s)^{-w}
\]
continues analytically to the region
$\sigma \ge 1 - c_0/(1 + \log^{+} |t|)$ (where $s = \sigma + it$ and
$\log^+ u = \max(0, \log u)$), and satisfies there the bound
$|G(s; w)| \le M (1 + |t|)^{1 - \delta}$.

\begin{proposition}[Selberg-Delange; {\cite[Ch.~II.5, Theorem~5.2]{Te}}]
\label{prop:LSD}
Let $w \in \C$ with $|w| \le 1$, and suppose
$F(s) = \sum_n a_n n^{-s}$, with $|a_n| \le 1$ for all $n$, has
property $\mathcal P(w; c_0, \delta, M)$. Then there are complex
numbers $\mu_j = \mu_j(F, w)$, $j \ge 0$ -- independent of the
truncation depth $J$ below -- with
\[
\mu_j = \frac{1}{\Gamma(w - j)} \cdot \kappa_j(F, w),
\qquad
\mu_0 = \frac{G(1; w)}{\Gamma(w)},
\qquad
|\kappa_j(F, w)| \ll_{c_0, \delta, j} M,
\]
such that, for every integer $J \ge 1$ and all $x \ge 3$,
\[
\sum_{n \le x} a_n
= x \sum_{j=0}^{J-1} \mu_j\, (\log x)^{w - 1 - j}
\;+\; O_{c_0, \delta, J}\big( M\, x\, (\log x)^{\real w - 1 - J} \big).
\]
Here $\kappa_j(F,w)$ is $1/j!$ times the $j$-th derivative at $s = 1$
of $G(s; w) \big( (s-1)\zeta(s) \big)^{w} / s$, a quantity determined
by finitely many derivatives of $G$ at $s = 1$.
\end{proposition}

\begin{proof}
This is the Selberg-Delange theorem in the form given in
\cite[Ch.~II.5, Theorem~5.2]{Te} (third edition; in earlier
editions it is Theorem~3 of Ch.~II.5), whose statement already
records the required uniformity in \(M\). In the bounded range
\(|w|\le 1\) used here, the constants depend only on \(c_0\), \(\delta\),
and the order of expansion. See also
\cite[Appendix, Theorem~A.13]{CM} for a formulation adapted to
applications of this kind. The description of the coefficients is
\cite[Ch.~II.5, eqs.~(5.13)--(5.15)]{Te}.
With these conventions, the leading coefficient is
\(G(1; w)/\Gamma(w)\), and the lower-order coefficients are obtained
from the Taylor expansion recorded in the statement. The bound
\(|\kappa_j(F,w)|\ll_{c_0,\delta,j}M\) follows directly from Cauchy's
estimates on a fixed disc about \(s=1\): in that disc the factors
\(((s-1)\zeta(s))^w\) and \(1/s\) have derivatives bounded in terms
of \(j\), while the hypothesis gives \(G(s;w)\ll M\).
In this paper the proposition is applied only in the bounded-coefficient
range. Namely, it is applied to the coprime series
\[
  F_q(s)=\sum_{(n,q)=1}d_z(n)n^{-s}
\]
with real \(0<|z|<1\), and to the series \(\zeta(s)^{\pm a/b}\) with
\(0<a/b<1\). By Lemma~\ref{lem:dz}(b), the coefficients \(d_w(n)\)
satisfy \(|d_w(n)|\le1\) for \(-1\le w\le1\). Indeed, for
\(0<w\le1\),
\[
  d_w(p^\nu)=\prod_{j=1}^{\nu}\frac{w+j-1}{j},
\]
so \(0\le d_w(p^\nu)\le1\), and the case \(-1\le w<0\) is handled
there in the same prime-power calculation. Thus the hypothesis
\(|a_n|\le1\) is satisfied at every point where the proposition is
used. No bounded-coefficient assertion is intended here for positive
exponents \(w>1\), for which already \(d_w(p)=w\).
In this bounded-coefficient case the majorant hypothesis of
\cite[Ch.~II.5]{Te} is satisfied trivially: with \(b_n=1\) the
majorant series is \(\zeta(s)\), which has property
\(\mathcal P(1;c_0,\delta,1)\), its associated \(G(s;1)\) being
identically \(1\). Hence \(F\) is of type
\(\mathcal T(w,1;c_0,\delta,\max(M,1))\) in the terminology used
there, and the theorem applies with \(A=1\) and \(N=J-1\), whose
error term \(O(MR_{J-1}(x))\) is
\(O_{c_0,\delta,J}(M(\log x)^{-J})\) relative to the factor
\(x(\log x)^{w-1}\).
\end{proof}

\subsubsection*{A Siegel-Walfisz bound for character twists of $d_z$}

The second analytic input concerns non-principal characters. Note
that, by multiplicativity, for any Dirichlet character $\chi$,
\begin{equation}\label{eq:twist-series}
\sum_{n \ge 1} \frac{d_z(n) \chi(n)}{n^{s}}
= \prod_p \Big( 1 - \frac{\chi(p)}{p^{s}} \Big)^{-z}
=: L(s, \chi)^{z}
\qquad (\real s > 1),
\end{equation}
the Euler product defining the branch.

\begin{proposition}[Siegel-Walfisz for $d_z \chi$]\label{prop:SW}
Fix $z$ real with $0 < |z| \le 1$ and $A, B > 0$. Let
$q \le (\log x)^{B}$ and let $\chi$ be a non-principal character mod
$q$; recall from \eqref{eq:twist-series} that the Dirichlet
coefficients of $L(s, \chi)^{z}$ are $d_z(n) \chi(n)$. Then
\[
\sum_{n \le x} d_z(n) \chi(n) \ll_{z, A, B} \frac{x}{(\log x)^{A}} .
\]
The implied constant is ineffective.
\end{proposition}

\begin{proof}
The proof, by the classical contour method, is given in
Appendix~\ref{app:SW}. A more general estimate, a Siegel-Walfisz
theorem for the coefficients of products
$\prod_\psi L(s, \psi)^{\alpha_\psi}$, follows from the work of
Singha Roy~\cite{SR}. In the imprimitive case, one first reduces to
the primitive character inducing \(\chi\) and treats the finitely many
Euler factors at primes dividing \(q\), as in Appendix~\ref{app:SW}.
Treatments of this kind go back to Scourfield~\cite{Sc1, Sc2}; the
underlying Selberg-Delange machinery is in \cite[Ch.~II.5]{Te} and
\cite[Ch.~13]{K}, with \cite{CM} providing a version with additional
uniformity.
\end{proof}

\subsection{The local factors: proof of Theorem~\ref{thm:local}}
\label{sec:local}

Throughout this section $z$ is a fixed real number with
$0 < |z| < 1$, $q \ge 1$ is an arbitrary integer, and $r$ is an
integer with $(r, q) = 1$; no upper bound on $q$ is assumed here. All
implied constants may depend on $z$.

\subsubsection*{The local Dirichlet series}

\begin{definition}\label{def:local-series}
For a prime power $p^k$ ($k \ge 1$) and $\real s > 0$, set
\[
\sigma_{p^k}(s) := \sum_{j \ge 0}
\frac{d_z(p^j)\, c_{p^k}(p^j)}{p^{js}},
\qquad
\Sigma_q(s) := \prod_{p^k \,\|\, q} \sigma_{p^k}(s)
\quad (\Sigma_1(s) := 1).
\]
\end{definition}

\begin{lemma}\label{lem:Sigma}
Let $q \ge 1$ and $\real s > 0$. Then:
\begin{enumerate}
\item[(i)] each $\sigma_{p^k}(s)$ converges absolutely and is
holomorphic in $\real s > 0$, and
\begin{equation}\label{eq:Sigma-identity}
\Sigma_q(s) = \sum_{\rad(a) \mid q} \frac{d_z(a)\, c_q(a)}{a^{s}},
\end{equation}
the series converging absolutely;
\item[(ii)] using Lemma~\ref{lem:ramanujan}(ii),
\begin{equation}\label{eq:sigma-eval}
\sigma_{p^k}(s) = \varphi(p^k) \sum_{j \ge k}
\frac{d_z(p^j)}{p^{js}} \;-\; p^{k-1}\,
\frac{d_z(p^{k-1})}{p^{(k-1)s}} ;
\end{equation}
\item[(iii)] for $0 < \rho \le 1/4$ and $|s - 1| \le \rho$,
\[
|\sigma_{p^k}(s)| \le 4\, p^{k\rho},
\qquad \text{hence} \qquad
|\Sigma_q(s)| \le 4^{\omega(q)}\, q^{\rho} ;
\]
\item[(iv)] for every integer $i \ge 0$ and $0 < \rho \le 1/4$,
\[
\big| \Sigma_q^{(i)}(1) \big|
\le i!\, \rho^{-i}\, 4^{\omega(q)}\, q^{\rho},
\qquad \text{and} \qquad
\Sigma_q^{(i)}(1) = (-1)^i \sum_{\rad(a) \mid q}
\frac{d_z(a)\, c_q(a)\, (\log a)^{i}}{a} .
\]
\end{enumerate}
\end{lemma}

\begin{proof}
(i) Absolute convergence: $|d_z(p^j)| \le 1$
(Lemma~\ref{lem:dz}(b)) and $|c_{p^k}(p^j)| \le p^{k}$ trivially
from \eqref{eq:ramanujan-def}, so the series is dominated by
$p^{k} \sum_j p^{-j \real s}$; holomorphy follows by local uniformity.
For \eqref{eq:Sigma-identity}: each $a$ with $\rad(a) \mid q$ factors
uniquely as $a = \prod_{p \mid q} p^{j_p}$; then
$d_z(a) = \prod_p d_z(p^{j_p})$ by multiplicativity, and
$c_q(a) = \prod_{p^k \| q} c_{p^k}(a) = \prod_{p^k \| q}
c_{p^k}(p^{j_p})$, using Lemma~\ref{lem:cq-mult} and then the fact
that $c_{p^k}(a)$ depends only on $(p^k, a) = (p^k, p^{j_p})$
(Lemma~\ref{lem:ramanujan}(i)). Expanding the product of the
absolutely convergent series $\sigma_{p^k}(s)$ and rearranging gives
\eqref{eq:Sigma-identity}.

(ii) Immediate from Lemma~\ref{lem:ramanujan}(ii): the terms with
$j \le k - 2$ vanish, the term $j = k - 1$ contributes
$-p^{k-1} d_z(p^{k-1}) p^{-(k-1)s}$, and the terms $j \ge k$ carry the
factor $\varphi(p^k)$.

(iii) Let $|s - 1| \le \rho \le 1/4$, so $\real s \ge 3/4$. By
\eqref{eq:sigma-eval}, Lemma~\ref{lem:dz}(b), and
$\varphi(p^k) \le p^{k}$,
\[
|\sigma_{p^k}(s)|
\le p^{k} \sum_{j \ge k} p^{-j(1 - \rho)} + p^{k - 1}\,
p^{-(k-1)(1 - \rho)}
= \frac{p^{k\rho}}{1 - p^{-(1-\rho)}} + p^{(k-1)\rho}
\le \Big( \frac{1}{1 - 2^{-3/4}} + 1 \Big) p^{k\rho}
\le 4\, p^{k\rho},
\]
since $1 - \rho \ge 3/4$ and $(1 - 2^{-3/4})^{-1} < 2.5$. Taking the
product over $p^k \,\|\, q$ and using $\prod p^{k\rho} = q^{\rho}$
gives the second bound.

(iv) The first claim is Cauchy's estimate for the $i$-th derivative
on the disc $|s - 1| \le \rho$, using (iii). The second claim follows
by differentiating \eqref{eq:Sigma-identity} term by term $i$ times at
$s = 1$; this is justified because the series
\eqref{eq:Sigma-identity} converges absolutely, uniformly on compact
subsets of $\real s > 1/2$ (dominate by
$q^2 \sum_{\rad(a) \mid q} a^{-1/2}$, finite by
Lemma~\ref{lem:rad-sums}(i), after noting $|c_q(a)| \le
(q,a)\tau((q,a)) \le q \tau(q) \le q^2$ from
Lemma~\ref{lem:ramanujan}(iii)).
\end{proof}

\subsubsection*{The reduction to character sums}

\begin{lemma}
\label{lem:reduction}
Let $x \ge 3$, $q \ge 1$, $(r, q) = 1$. Then
\begin{equation}\label{eq:reduction}
\sum_{n \le x} d_z(n) \e\Big( \frac{rn}{q} \Big)
=
\frac{1}{\varphi(q)} \sum_{\chi \bmod q}
\sum_{\substack{a \le x \\ \rad(a) \mid q}}
d_z(a)\, \tau(\chi, ra)\, S_{\overline{\chi}}\Big( \frac{x}{a} \Big),
\end{equation}
where $\tau(\chi, m)$ is as in Lemma~\ref{lem:fourier} and
\[
S_{\psi}(y) := \sum_{b \le y} d_z(b)\, \psi(b)
\qquad (\text{note } \psi(b) = 0 \text{ unless } (b, q) = 1).
\]
Moreover, $\tau(\chi_0, ra) = c_q(ra) = c_q(a)$.
\end{lemma}

\begin{proof}
Decompose each $n \le x$ uniquely as $n = ab$ with
$\rad(a) \mid q$ and $(b, q) = 1$ (the $q$-part decomposition of
Section~\ref{subsec:notation}); then $d_z(n) = d_z(a) d_z(b)$ since
$(a, b) = 1$. Hence
\[
\sum_{n \le x} d_z(n) \e\Big( \frac{rn}{q} \Big)
= \sum_{\substack{a \le x \\ \rad(a) \mid q}} d_z(a)
\sum_{\substack{b \le x/a \\ (b, q) = 1}} d_z(b)\,
\e\Big( \frac{(ra) b}{q} \Big).
\]
For each $b$ in the inner sum, $(b, q) = 1$, so
Lemma~\ref{lem:fourier} with $m = ra$ gives
\[
\e( rab/q ) = \frac{1}{\varphi(q)} \sum_{\chi} \overline{\chi}(b)\,
\tau(\chi, ra).
\]
Substituting and interchanging the finite sums
yields \eqref{eq:reduction}, the coprimality condition on $b$ being
absorbed into $\overline\chi(b)$. Finally,
$\tau(\chi_0, ra) = c_q(ra)$ by Lemma~\ref{lem:fourier}, and
$c_q(ra) = c_q(a)$ because $c_q(m)$ depends only on $(q, m)$
(Lemma~\ref{lem:ramanujan}(i)) and $(q, ra) = (q, a)$ since
$(r, q) = 1$.
\end{proof}

\begin{remark}\label{rem:r-independence}
The last assertion of Lemma~\ref{lem:reduction} is the source of the
$r$-independence of the main term in Theorem~\ref{thm:major}: the
numerator $r$ survives only inside the non-principal Gauss sums
$\tau(\chi, ra)$, which will be estimated trivially and absorbed into
the error term.
\end{remark}

\subsubsection*{Proof of Theorem~\ref{thm:local}}

By \eqref{eq:Sigma-identity} at $s = 1$, the series in
\eqref{eq:g-def} converges absolutely and
\begin{equation}\label{eq:g-Sigma}
\g_z(q)
= \Big( \prod_{p \mid q} \Big( 1 - \frac1p \Big)^{z} \Big)
\frac{\Sigma_q(1)}{\varphi(q)}
= \prod_{p^k \,\|\, q} \gamma_{p^k},
\qquad
\gamma_{p^k} :=
\Big( 1 - \frac1p \Big)^{z}\, \frac{\sigma_{p^k}(1)}{\varphi(p^k)},
\end{equation}
where the factorization uses Definition~\ref{def:local-series} and the
multiplicativity of $\varphi$ and of $q \mapsto \prod_{p \mid q}(1 -
1/p)^z$. This exhibits $\g_z$ as multiplicative, with $\g_z(1) = 1$.

Next we evaluate $\gamma_{p^k}$. By \eqref{eq:sigma-eval} at $s = 1$,
\begin{equation}\label{eq:gamma-pk}
\gamma_{p^k}
= \Big( 1 - \frac1p \Big)^{z}
\bigg[ \sum_{j \ge k} \frac{d_z(p^{j})}{p^{j}}
\;-\; \frac{d_z(p^{k-1})}{\varphi(p^k)} \bigg]
\qquad (k \ge 1).
\end{equation}
For $k = 1$ this gives, using
$\sum_{j \ge 0} d_z(p^j) p^{-j} = (1 - 1/p)^{-z}$ (the defining Euler
factor, an absolutely convergent binomial series by
Lemma~\ref{lem:binom}(i)) and $\varphi(p) = p - 1$,
\[
\gamma_p
= \Big( 1 - \frac1p \Big)^{z}
\bigg[ \Big( 1 - \frac1p \Big)^{-z} - 1 - \frac{1}{p - 1} \bigg]
= 1 - \Big( 1 - \frac1p \Big)^{z} \cdot \frac{p}{p - 1}
= 1 - \Big( 1 - \frac1p \Big)^{z - 1},
\]
which is \eqref{eq:gamma-p}.

Finally, the decay. For every $k \ge 1$, from \eqref{eq:gamma-pk},
$|d_z| \le 1$, and $|(1 - 1/p)^{z}| \le (1 - 1/p)^{-1} \le 2$ (valid
for $-1 < z < 1$),
\[
|\gamma_{p^k}|
\le 2 \bigg[ \sum_{j \ge k} p^{-j}
+ \frac{1}{p^{k-1}(p - 1)} \bigg]
\le 2 \bigg[ \frac{2}{p^{k}} + \frac{2}{p^{k}} \bigg]
= \frac{8}{p^{k}},
\]
using $\sum_{j \ge k} p^{-j} = p^{-k} (1 - 1/p)^{-1} \le 2 p^{-k}$
for the first term and $p^{k-1}(p - 1) \ge p^{k}/2$ for the second.
Hence, by \eqref{eq:g-Sigma} and Lemma~\ref{lem:rad-sums}(iii),
\[
|\g_z(q)| \le \prod_{p^k \,\|\, q} \frac{8}{p^{k}}
= \frac{8^{\omega(q)}}{q}
\ll_{\varepsilon} q^{-1 + \varepsilon},
\]
for every $\varepsilon > 0$. This completes the proof of
Theorem~\ref{thm:local}. \qed

\begin{remark}[Consistency checks]\label{rem:checks}
The following two observations are purely illustrative and are used
nowhere in the paper.

(i) \emph{The case $z = 1$.} Here $d_1 = \ID$ and
$\sum_{n \le x} \e(rn/q) = O(1)$ for $q > 1$, so every main-term
coefficient must vanish. Indeed \eqref{eq:gamma-pk} gives, for every
$k \ge 1$,
\[
\gamma_{p^k}\big|_{z=1}
= \Big( 1 - \frac1p \Big)
\bigg[ \sum_{j \ge k} p^{-j} - \frac{1}{p^{k-1}(p-1)} \bigg]
= \Big( 1 - \frac1p \Big)
\bigg[ \frac{p^{1-k}}{p - 1} - \frac{p^{1-k}}{p - 1} \bigg]
= 0,
\]
so $\g_1(q) = 0$ for all $q > 1$: the entire local structure
annihilates the main term, prime power by prime power.

(ii) \emph{The case $z = 2$.} Here $d_2 = \tau$ and
\eqref{eq:gamma-p} gives $\gamma_p|_{z=2} = 1 - (1 - 1/p) = 1/p$. For
$q = p$ prime, the predicted leading term of
$\sum_{n \le x} \tau(n) \e(rn/p)$ is therefore
$\g_2(p)\, x (\log x)^{2-1}/\Gamma(2) = x \log x / p$. This matches
the classical asymptotic for this sum, whose leading term
$(x/q) \log x$ goes back to Estermann and is visible already in the
hyperbola computation
$\sum_{n \le x} \tau(n)\e(rn/q) = \sum_{d \le x} \sum_{m \le x/d}
\e(rdm/q)$: the terms with $q \mid d$ contribute
$\sum_{d \le x,\, q \mid d} \lfloor x/d \rfloor \sim (x/q)\log x$.
The formula for the local Euler factors continues algebraically to
this integer value, although the estimates of this section were
proved only in the range $0 < |z| < 1$. (We do not claim, and do not
use, that the remaining terms are of
lower order; for $z = 2$ this is classical, and for the range of $z$
treated in this paper the full claim is exactly
Theorem~\ref{thm:major}.)
\end{remark}


\subsection{Proof of Theorem~\ref{thm:major}}
\label{sec:proof-major}

Throughout this section $z$ is fixed with $0 < |z| < 1$, and $B > 0$
and the integer $J \ge 1$ are fixed; implied constants may depend on
$z, B, J$ and, where indicated, on $\varepsilon$. We write
\[
T(x) = T(x; r, q) := \sum_{n \le x} d_z(n) \e\Big( \frac{rn}{q} \Big).
\]

\subsubsection*{A preparatory lemma}

\begin{lemma}
\label{lem:propertyP}
Let $q \ge 1$ and let
\[
F_q(s) := \sum_{\substack{b \ge 1 \\ (b, q) = 1}}
\frac{d_z(b)}{b^{s}}
= \zeta(s)^{z}\, G_q(s),
\qquad
G_q(s) := \prod_{p \mid q} \big( 1 - p^{-s} \big)^{z}
\quad (\real s > 1).
\]
Then $F_q$ has property $\mathcal P(z; \tfrac14, 1, M_q)$ with
$M_q := 3^{\omega(q)}$, and $G_q(1) = \prod_{p \mid q}(1 - 1/p)^{z}$.
\end{lemma}

\begin{proof}
The Euler product identity is immediate from multiplicativity:
$F_q(s) = \prod_{p \nmid q} (1 - p^{-s})^{-z}$, and multiplying and
dividing by the factors at $p \mid q$ gives
$F_q = \zeta^{z} G_q$, where $\zeta(s)^{z}$ is defined by its Euler
product for $\real s > 1$ and by
$\exp( z \log \zeta(s) )$ in the standard zero-free region
thereafter, as in \cite[Ch.~II.5]{Te}.

The region $\sigma \ge 1 - \tfrac14 / (1 + \log^{+}|t|)$ is contained
in the half-plane $\sigma \ge 3/4$. There, for every prime $p$,
$|p^{-s}| = p^{-\sigma} \le 2^{-3/4} < 1$, so
$\real( 1 - p^{-s} ) \ge 1 - 2^{-3/4} > 0$ and the principal power
$(1 - p^{-s})^{z} = \exp( z \operatorname{Log}( 1 - p^{-s} ) )$ is
holomorphic; hence so is the finite product $G_q$. For real $z$ and
$w$ in the right half-plane, $|w^{z}| = |w|^{z}$, and
$|1 - p^{-s}| \in [ 1 - 2^{-3/4},\ 1 + 2^{-3/4} ]$, whence
\[
\big| ( 1 - p^{-s} )^{z} \big|
\le \max\Big( \big( 1 + 2^{-3/4} \big)^{|z|},\
\big( 1 - 2^{-3/4} \big)^{-|z|} \Big)
\le \big( 1 - 2^{-3/4} \big)^{-1} < 3,
\]
for $|z| \le 1$. Thus $|G_q(s)| \le 3^{\omega(q)} = M_q$ on the
region, which is the required bound with $\delta = 1$.
\end{proof}

\subsubsection*{The expansion at $\beta = 0$}

\begin{proposition}\label{prop:beta0}
Fix $z$, $B > 0$ and an integer $J \ge 1$. Define, for $j \ge 0$ and
$q \ge 1$,
\begin{equation}\label{eq:theta-def}
\theta_j(z, q) := \frac{1}{\varphi(q)}
\sum_{\substack{j' + i = j \\ j', i \ge 0}}
\mu_{j'}(q)\, \binom{z - 1 - j'}{i}\, (-1)^{i}\, m_i(q),
\end{equation}
where $\mu_{j'}(q) := \kappa_{j'}(F_q, z)/\Gamma(z - j')$ with
$\kappa_{j'}$ as in Proposition~\ref{prop:LSD} applied to the series
$F_q$ of Lemma~\ref{lem:propertyP}, and
$m_i(q) := (-1)^{i} \Sigma_q^{(i)}(1)$ as in
Lemma~\ref{lem:Sigma}(iv). Then:
\begin{enumerate}
\item[(i)] $\theta_0(z, q) = \g_z(q)/\Gamma(z)$, and
$|\theta_j(z, q)| \ll_{z, j, \varepsilon} q^{-1+\varepsilon}$ for
every $\varepsilon > 0$;
\item[(ii)] for all $x \ge 3$ and every reduced fraction $r/q$ with
$q \le (\log x)^{B}$,
\[
T(x; r, q)
= x \sum_{j = 0}^{J - 1} \theta_j(z, q)\, (\log x)^{z - 1 - j}
+ O_{z, B, J}\big( x (\log x)^{z - 1 - J} \big).
\]
\end{enumerate}
\end{proposition}

\begin{proof}
\emph{Part (i).} By Proposition~\ref{prop:LSD},
$\mu_0(q) = G_q(1)/\Gamma(z)$, and by Lemma~\ref{lem:Sigma}(iv) at
$i = 0$, $m_0(q) = \Sigma_q(1)$. Hence the single term
$(j', i) = (0, 0)$ of \eqref{eq:theta-def} gives
\[
\theta_0(z, q)
= \frac{1}{\varphi(q)} \cdot \frac{G_q(1)}{\Gamma(z)} \cdot \Sigma_q(1)
= \frac{1}{\Gamma(z)} \Big( \prod_{p \mid q} \Big(1 - \frac1p\Big)^{z}
\Big) \frac{\Sigma_q(1)}{\varphi(q)}
= \frac{\g_z(q)}{\Gamma(z)},
\]
by \eqref{eq:g-Sigma}. For the size bound, fix $\varepsilon > 0$ and
apply, with $\rho := \varepsilon$ (we may assume
$\varepsilon \le 1/4$): $|\kappa_{j'}| \ll_{j'} M_q = 3^{\omega(q)}$
and $|1/\Gamma(z - j')| \ll_{j'} 1$ (the function $1/\Gamma$ is
entire), so $|\mu_{j'}(q)| \ll_{j'} 3^{\omega(q)}$;
$|\binom{z - 1 - j'}{i}| \ll_{j} 1$ by Lemma~\ref{lem:taylor}; and
$|m_i(q)| \le i!\, \varepsilon^{-i}\, 4^{\omega(q)} q^{\varepsilon}$
by Lemma~\ref{lem:Sigma}(iv). Finally
$1/\varphi(q) = (1/q) \prod_{p \mid q} (1 - 1/p)^{-1} \le
2^{\omega(q)}/q$. Multiplying, and using
Lemma~\ref{lem:rad-sums}(iii) to bound $24^{\omega(q)} \ll_\varepsilon
q^{\varepsilon}$, we obtain
$|\theta_j| \ll_{z, j, \varepsilon} q^{-1 + 3\varepsilon}$; renaming
$3\varepsilon$ as $\varepsilon$ gives (i).

\emph{Part (ii).} Set
\[
x_1 := \exp\big( \sqrt{\log x} \big),
\qquad
J_1 := J + \lceil 2B \rceil + 2,
\qquad
I := 2 J_1,
\qquad
A' := J + 2 + 2B .
\]
We may assume $x \ge x_0(z, B, J)$ for a suitably large threshold:
for $3 \le x < x_0$, the trivial bound $|T(x)| \le x$
(Lemma~\ref{lem:dz}(b)) and the lower bound
$(\log x)^{z - 1 - J} \ge (\log x_0)^{z - 1 - J} \gg_{z, B, J} 1$
make (ii) vacuous. In particular we assume $x_1 \le \sqrt x$ and
$\sqrt{\log x} \ge 4$.

By the reduction lemma (Lemma~\ref{lem:reduction}),
\begin{equation}\label{eq:T-split}
T(x) = \frac{1}{\varphi(q)}
\sum_{\substack{a \le x \\ \rad(a) \mid q}} d_z(a)\, c_q(a)\,
S_{\chi_0}\Big( \frac xa \Big)
\;+\;
\frac{1}{\varphi(q)} \sum_{\chi \ne \chi_0}
\sum_{\substack{a \le x \\ \rad(a) \mid q}} d_z(a)\,
\tau(\chi, ra)\, S_{\overline\chi}\Big( \frac xa \Big)
=: T_{0} + T_{1}.
\end{equation}
Throughout we use $|d_z(a)| \le 1$ (Lemma~\ref{lem:dz}(b)),
$|c_q(a)| \le q \tau(q) \le q^{2}$
(Lemma~\ref{lem:ramanujan}(iii)), $|\tau(\chi, ra)| \le \varphi(q)$
(Lemma~\ref{lem:fourier}), and the identities
$\sum_{\rad(a) \mid q} a^{-1} = q/\varphi(q)$ and
$\sum_{a > x_1, \rad(a) \mid q} a^{-1} \le 4^{\omega(q)}
x_1^{-1/2}$ (Lemma~\ref{lem:rad-sums}(i), (ii)). We also record that, by
Lemma~\ref{lem:rad-sums}(iii) (with $\varepsilon = 1$), for each
fixed $C > 1$ one has $C^{\omega(q)} \ll_{C} q$, so that any factor
of the form $q^{\kappa}\, C^{\omega(q)}$ with fixed $\kappa \ge 0$
and fixed $C > 1$ satisfies
\[
q^{\kappa}\, C^{\omega(q)} \ll_{C, \kappa} q^{\kappa + 1}
\le (\log x)^{(\kappa + 1) B} :
\]
a fixed power of $\log x$. Since
$x_1^{-1/2} = \exp( -\tfrac12 \sqrt{\log x} )$, and likewise
$x_1^{-1/4}$, decays faster than any fixed power of $\log x$, we
conclude: \emph{every error term below that carries the factor
$x_1^{-1/4}$ or $x_1^{-1/2}$, multiplied by at most a fixed power of
$q$ times a factor $C^{\omega(q)}$ with fixed $C$, is}
$O_{z, B, J}( x (\log x)^{z - 1 - J} )$; we call such terms
\emph{negligible} without further comment.

\emph{The tails $a > x_1$.} Bounding $|S_{\psi}(y)| \le y$ trivially,
the terms of \eqref{eq:T-split} with $a > x_1$ contribute at most
\[
\frac{q^{2} x}{\varphi(q)} \sum_{\substack{a > x_1 \\ \rad(a) \mid q}}
\frac1a
\;+\;
\frac{\varphi(q) \cdot \varphi(q)\, x}{\varphi(q)}
\sum_{\substack{a > x_1 \\ \rad(a) \mid q}} \frac1a
\le
\big( q^{2} + \varphi(q) \big)\, x\, 4^{\omega(q)}\, x_1^{-1/2},
\]
which is negligible. From now on all sums over $a$ are restricted to
$a \le x_1$, so that $y := x/a \ge x/x_1 \ge \sqrt x$ and
\begin{equation}\label{eq:logs-comparable}
\tfrac12 \log x \le \log y \le \log x,
\qquad
q \le (\log x)^{B} \le \big( 2 \log y \big)^{B} \le
(\log y)^{B + 1}
\end{equation}
(the last step for $x \ge x_0$).

\emph{The non-principal part $T_1$.} Fix $\chi \ne \chi_0$. The
Dirichlet series of $b \mapsto d_z(b) \overline\chi(b)$ is
$L(s, \overline\chi)^{z}$ by \eqref{eq:twist-series}, so
Proposition~\ref{prop:SW} (applied to $\overline\chi$, with modulus
parameter $B + 1$ by \eqref{eq:logs-comparable} and saving $A'$)
gives $S_{\overline\chi}(y) \ll y (\log y)^{-A'} \ll
(x/a)(\log x)^{-A'}$, using \eqref{eq:logs-comparable} again. Hence
\[
|T_1|
\le \frac{1}{\varphi(q)} \sum_{\chi \ne \chi_0} \varphi(q)
\sum_{\substack{a \le x_1 \\ \rad(a) \mid q}}
\frac{x}{a} (\log x)^{-A'} \cdot O(1)
\ll
\varphi(q) \cdot \frac{q}{\varphi(q)} \cdot x (\log x)^{-A'}
\le q\, x\, (\log x)^{-A'} .
\]
Since $q \le (\log x)^{B}$ and $A' = J + 2 + 2B \ge J + 1 - z + B$,
this is $O( x (\log x)^{z - 1 - J} )$.

\emph{The principal part $T_0$: applying the Selberg-Delange expansion.} By
Lemma~\ref{lem:propertyP} and Proposition~\ref{prop:LSD} (with
$w = z$, $J_1$ terms, and $M = M_q$), for every $y \ge 3$,
\[
S_{\chi_0}(y)
= y \sum_{j' = 0}^{J_1 - 1} \mu_{j'}(q)\, (\log y)^{z - 1 - j'}
+ O_{J_1}\big( M_q\, y\, (\log y)^{z - 1 - J_1} \big),
\]
with the $\mu_{j'}(q)$ of the proposition statement, satisfying
$|\mu_{j'}(q)| \ll_{j'} M_q$. Substituting into $T_0$ (with
$y = x/a$, $a \le x_1$): the Selberg-Delange error terms contribute, by
\eqref{eq:logs-comparable},
\[
\ll \frac{1}{\varphi(q)} \cdot q^{2} \cdot M_q \cdot x\,
2^{J_1 + 1 - z} (\log x)^{z - 1 - J_1}
\sum_{\rad(a) \mid q} \frac1a
\ll
\frac{q^{3} M_q}{\varphi(q)^{2}}\, x\, (\log x)^{z - 1 - J_1}
\ll x (\log x)^{z - 1 - J_1 + 2B},
\]
where the last step used
$q^{3} M_q / \varphi(q)^{2} = q\, (q/\varphi(q))^{2}\, 3^{\omega(q)}
\ll_{\varepsilon} q^{1 + \varepsilon} \le q^{2} \le (\log x)^{2B}$
for $q \le (\log x)^{B}$, by Lemma~\ref{lem:rad-sums}(iii) and
$q/\varphi(q) \le 2^{\omega(q)}$, from the Euler product, as in
Part~(i). This is
$O( x (\log x)^{z - 1 - J} )$ since $J_1 \ge J + 2B + 2$. There remains the main expression
\begin{equation}\label{eq:main-expr}
\frac{x}{\varphi(q)} \sum_{j' = 0}^{J_1 - 1} \mu_{j'}(q)
\sum_{\substack{a \le x_1 \\ \rad(a) \mid q}}
\frac{d_z(a)\, c_q(a)}{a} \Big( \log \frac xa \Big)^{z - 1 - j'} .
\end{equation}

\emph{Expanding the logarithms.} Write, for $a \le x_1$,
\[
\Big( \log \frac xa \Big)^{z - 1 - j'}
= (\log x)^{z - 1 - j'} ( 1 - u_a )^{z - 1 - j'},
\qquad
u_a := \frac{\log a}{\log x} \in \Big[ 0,\ \frac{1}{\sqrt{\log x}}
\Big] \subseteq [ 0, \tfrac12 ].
\]
By Lemma~\ref{lem:taylor} with $w = z - 1 - j'$ and $I = 2 J_1$
terms,
\[
( 1 - u_a )^{z - 1 - j'}
= \sum_{i = 0}^{I - 1} \binom{z - 1 - j'}{i} (-1)^{i}
\frac{(\log a)^{i}}{(\log x)^{i}}
+ O_{z, J_1}\big( u_a^{I} \big).
\]
The remainder terms contribute to \eqref{eq:main-expr} at most (using
$u_a^{I} = (\log a)^{I} (\log x)^{-I}$, and then bounding the
resulting log-moment by Cauchy's estimate exactly as in
Lemma~\ref{lem:Sigma}(iv), applied to the series
$\sum_{\rad(a) \mid q} a^{-s}$ of Lemma~\ref{lem:rad-sums}(i), whose
factors satisfy $|(1 - p^{-s})^{-1}| \le 4$ for $|s - 1| \le 1/4$):
\[
\ll_{J_1} \frac{x\, M_q}{\varphi(q)}\, (\log x)^{z - 1}\,
(\log x)^{-I} \cdot q^{2} \cdot I!\, 4^{I}\, 4^{\omega(q)}\,
q^{1/4}
\ll x\, (\log x)^{z - 1 - I + 2B},
\]
where the last step used
$q^{2} M_q\, 4^{\omega(q)}\, q^{1/4} / \varphi(q)
= q\, (q/\varphi(q))\, 12^{\omega(q)}\, q^{1/4}
\ll_{\varepsilon} q^{5/4 + \varepsilon} \le q^{2} \le
(\log x)^{2B}$ for $q \le (\log x)^{B}$, again by
Lemma~\ref{lem:rad-sums}(iii). This is
$O( x (\log x)^{z - 1 - J} )$ since $I = 2 J_1 \ge J + 2B + 1$.

\emph{Completing the logarithmic moments.} In the surviving terms appear the
truncated moments
\[
\sum_{\substack{a \le x_1 \\ \rad(a) \mid q}}
\frac{d_z(a)\, c_q(a)\, (\log a)^{i}}{a} .
\]
Completing them to the full moments $m_i(q) = (-1)^{i}
\Sigma_q^{(i)}(1)$ of Lemma~\ref{lem:Sigma}(iv) costs, per pair
$(j', i)$, at most
\[
\frac{x}{\varphi(q)}\, |\mu_{j'}(q)|\, (\log x)^{z - 1 - j' - i}
\Big| \binom{z-1-j'}{i} \Big|
\sum_{\substack{a > x_1 \\ \rad(a) \mid q}}
\frac{q^{2} (\log a)^{i}}{a}
\ll_{z, J_1}
x\, q^{2} M_q\, 4^{\omega(q)}\, x_1^{-1/4},
\]
using $(\log a)^{i} \ll_{i} a^{1/4}$ and
Lemma~\ref{lem:rad-sums}(i) with $\sigma = 1/2$; this is negligible.

\emph{Assembly.} After these steps, \eqref{eq:main-expr} equals
\[
x \sum_{j' = 0}^{J_1 - 1} \sum_{i = 0}^{I - 1}
\bigg[ \frac{\mu_{j'}(q)}{\varphi(q)} \binom{z - 1 - j'}{i}
(-1)^{i} m_i(q) \bigg] (\log x)^{z - 1 - (j' + i)}
+ O\big( x (\log x)^{z - 1 - J} \big).
\]
The pairs with $j' + i = j < J$ are in bijection with the terms of
\eqref{eq:theta-def} (note $j < J \le J_1$ forces $j' < J_1$ and
$i < J \le I$, so no term of \eqref{eq:theta-def} is missing), and
reproduce exactly $x \sum_{j < J} \theta_j(z, q) (\log x)^{z-1-j}$.
The pairs with $j' + i \ge J$ number $O_{z, B, J}(1)$, and each
contributes at most
\[
\bigg| \frac{\mu_{j'}(q)}{\varphi(q)} \binom{z-1-j'}{i} m_i(q) \bigg|
\, x\, (\log x)^{z - 1 - (j' + i)}
\ll_{z, B, J} q^{-1 + \varepsilon}\, x\, (\log x)^{z - 1 - J},
\]
by the same coefficient estimates as in Part (i); these are absorbed
into the error term. This proves (ii).
\end{proof}

\subsubsection*{Proof of Theorem~\ref{thm:major}}

Define, with the $\theta_j$ of Proposition~\ref{prop:beta0} (and
$\theta_{-1} := 0$),
\begin{equation}\label{eq:lambda-def}
\lambda_j(z, q) := \theta_j(z, q) + (z - j)\, \theta_{j-1}(z, q)
\qquad (0 \le j < J).
\end{equation}
Then $\lambda_0 = \theta_0 = \g_z(q)/\Gamma(z)$, and
$|\lambda_j| \ll_{z, J, \varepsilon} q^{-1+\varepsilon}$ by
Proposition~\ref{prop:beta0}(i); this is \eqref{eq:lambda-props}. We
may assume $x \ge x_0(z, B, J)$: for $3 \le x < x_0$ the left-hand
side of \eqref{eq:major} is at most $x$ in modulus, each integral is
at most $O_J(x)$ in modulus, the coefficients are $O(1)$, and the
quantity $x(\log x)^{-(J+1-z-2B)}$ is $\gg_{z,B,J,x_0} x$ on the
compact range $3 \le x \le x_0$; so \eqref{eq:major} holds there
after enlarging the implied constant.

Let $\alpha = r/q + \beta$ with $q \le (\log x)^{B}$ and
$|\beta| \le (\log x)^{B}/x$. First,
\[
\sum_{n \le x} d_z(n)\e(n\alpha)
= \sum_{\sqrt x < n \le x} d_z(n) \e\Big( \frac{rn}{q} \Big)
\e(\beta n) + O\big( \sqrt x \big),
\]
by $|d_z| \le 1$. The $O(\sqrt x)$ term is admissible, being
$O( x (\log x)^{\real z - 1 - J} )$ for every fixed $J$. By Riemann-Stieltjes partial summation on
$(\sqrt x, x]$, with $T(t) = T(t; r, q)$,
\begin{equation}\label{eq:parts}
\sum_{\sqrt x < n \le x} d_z(n) \e\Big( \frac{rn}{q} \Big) \e(\beta n)
= \Big[ \e(\beta t)\, T(t) \Big]_{t = \sqrt x}^{t = x}
- 2\pi i \beta \int_{\sqrt x}^{x} T(t)\e(\beta t)\, dt .
\end{equation}
For $t \in [\sqrt x, x]$ we have
$q \le (\log x)^{B} \le (\log t)^{B+1}$ as in
\eqref{eq:logs-comparable}, so Proposition~\ref{prop:beta0}, applied
with the parameter $B + 1$ in place of $B$ and with $J$ terms, gives
\[
T(t) = M(t) + E(t),
\qquad
M(t) := t \sum_{j = 0}^{J-1} \theta_j(z, q) (\log t)^{z - 1 - j},
\qquad
|E(t)| \ll t (\log t)^{z - 1 - J} .
\]
(The coefficients $\theta_j$ do not depend on the parameter $B$: they
are defined once and for all by \eqref{eq:theta-def}.)

\emph{The error part.} The contribution of $E$ to \eqref{eq:parts} is
\[
\ll |E(x)| + |E(\sqrt x)| + |\beta| \int_{\sqrt x}^{x} |E(t)|\, dt
\ll x (\log x)^{z - 1 - J}
+ |\beta| \cdot x \cdot x\, 2^{J + 1 - z} (\log x)^{z - 1 - J},
\]
where we used that $(\log t)^{z - 1 - J}$ is decreasing on
$[\sqrt x, x]$ and $(\log \sqrt x)^{z-1-J} = 2^{J + 1 - z}
(\log x)^{z-1-J}$. Since $|\beta| x \le (\log x)^{B}$, this is
$O( x (\log x)^{z - 1 - J + B} )$, which is within the error term of
\eqref{eq:major}.

\emph{The main part.} Since $M$ is continuously differentiable on
$[\sqrt x, x]$ (note $\sqrt x \ge 3$), integration by parts in the
reverse direction gives
\[
\Big[ \e(\beta t) M(t) \Big]_{\sqrt x}^{x}
- 2\pi i \beta \int_{\sqrt x}^{x} M(t) \e(\beta t)\, dt
= \int_{\sqrt x}^{x} M'(t)\e(\beta t)\, dt .
\]
Differentiating and regrouping by the exponent of $\log t$,
\[
M'(t)
= \sum_{j = 0}^{J-1} \theta_j \Big[ (\log t)^{z-1-j}
+ (z - 1 - j)(\log t)^{z - 2 - j} \Big]
\]
which, recalling \eqref{eq:lambda-def}, equals
\[
\sum_{j = 0}^{J-1} \lambda_j(z, q)\, (\log t)^{z - 1 - j}
+ (z - J)\, \theta_{J-1}\, (\log t)^{z - 1 - J},
\]
(the term $j$ of the
bracket contributes $\theta_j$ to $\lambda_j$ and
$(z - 1 - j)\theta_j$ to $\lambda_{j+1}$, since $z - 1 - j =
z - (j+1)$). The last term contributes at most
$|z - J|\, |\theta_{J-1}|\, x\, 2^{J+1-z} (\log x)^{z - 1 - J}
\ll x (\log x)^{z-1-J}$, admissible. Finally, we extend the
integrals to $[2, x]$: for each $j < J$,
\[
\bigg| \int_{2}^{\sqrt x} (\log t)^{z - 1 - j} \e(\beta t)\, dt
\bigg|
\le \int_{2}^{\sqrt x} (\log t)^{z - 1 - j}\, dt
\le (\log 2)^{z - 1 - j} \sqrt x
\ll_{J} \sqrt x,
\]
since the exponent $z - 1 - j$ is negative and $\log t \ge \log 2$;
multiplied by $|\lambda_j| \le O(1)$ and summed over $j < J$, this
costs $O_J(\sqrt x)$, admissible. Collecting all the pieces,
\[
\sum_{n \le x} d_z(n)\e(n\alpha)
= \sum_{j = 0}^{J-1} \lambda_j(z, q)
\int_{2}^{x} (\log t)^{z - 1 - j}\e(\beta t)\, dt
+ O_{z, B, J}\big( x (\log x)^{z - 1 - J + B} \big),
\]
which is \eqref{eq:major} (indeed with the stronger error exponent
$z - 1 - J + B$ in place of $z - 1 - J + 2B$). The ineffectivity is
inherited from Proposition~\ref{prop:SW}, and from nowhere else.
\qed


\subsection{Uniform estimates: proof of Theorem \ref{thm:sup}}\label{sec:proof-sup}

In this section we fix integers $1 \le a < b$ with $(a, b) = 1$ and
take $z = a/b$ or $z = -a/b$, so that the minor-arc results of Section~\ref{sec:main-proof}
(Theorem~\ref{thm:main} and Corollary~\ref{cor:minor}) are
available;
we write $C = C(a, b)$ for the constant of
Theorem~\ref{thm:main}, and allow all implied constants to
depend on $a$ and $b$. The reader interested in the case
$z = \pm\tfrac12$ may take $(a, b) = (1, 2)$ throughout. We write
\[
S(x, \alpha) := \sum_{n \le x} d_z(n)\e(n\alpha).
\]
Part (i) of Theorem~\ref{thm:sup} is proved first, by positivity; the
Dirichlet dissection in the last two subsections is required only for
part (ii), although the arguments there are valid for both signs.

\subsubsection*{Proof of Theorem~\ref{thm:sup}(i)}

By Lemma~\ref{lem:binom}(iii) we have $d_{a/b}(p^k) > 0$ for every
prime power, hence $d_{a/b}(n) > 0$ for every $n$ by
multiplicativity. Therefore, for every $\alpha \in \R$ and every
$x \ge 1$,
\[
\Big| \sum_{n \le x} d_{a/b}(n)\e(n\alpha) \Big|
\le \sum_{n \le x} d_{a/b}(n)
= S(x, 0),
\]
with equality at $\alpha = 0$; taking the supremum over $\alpha$
gives the stated identity, the supremum being attained at
$\alpha = 0$.

For the asymptotic, the series $F(s) = \zeta(s)^{a/b}$ has property
$\mathcal P( a/b;\ \tfrac14,\ 1,\ 1 )$ trivially ($G \equiv 1$), so
Proposition~\ref{prop:LSD} with $w = a/b$ and $J = 1$ gives
\[
\sum_{n \le x} d_{a/b}(n)
= \frac{x (\log x)^{a/b - 1}}{\Gamma(a/b)}
+ O\big( x (\log x)^{a/b - 2} \big)
= \frac{x (\log x)^{a/b - 1}}{\Gamma(a/b)}
\Big( 1 + O\Big( \frac{1}{\log x} \Big) \Big),
\]
using $\Gamma(a/b) > 0$ (the Gamma function is positive on
$(0, \infty)$). 
Every constant here is effective: the only analytic
input is the Selberg-Delange method for $\zeta(s)^{a/b}$ itself,
with \(G \equiv 1\), which rests on the classical effective zero-free
region of \(\zeta\). \qed

\subsubsection*{Proof of Theorem~\ref{thm:sup}(ii): the lower bound}

Here we take $z = -a/b$; the arguments of this and the next
subsection are in fact valid for both signs, but part (i) has
already been proved.

By Parseval's identity (orthogonality of the exponentials
$\e(n\alpha)$ on $[0, 1]$),
\[
\sup_{\alpha \in \R} | S(x, \alpha) |^{2}
\ge \int_{0}^{1} | S(x, \alpha) |^{2}\, d\alpha
= \sum_{n \le x} d_z(n)^{2} \ge d_z(1)^{2} = 1,
\]
so $\sup_\alpha |S(x, \alpha)| \ge 1$ for every $x \ge 1$; this
disposes of any fixed compact range of $x$ (on which
$x (\log x)^{z-1} \ll_{x_2} 1$). For large $x$ we use $\alpha = 0$:
the series $F(s) = \zeta(s)^{z}$ has property
$\mathcal P(z; \tfrac14, 1, 1)$ trivially ($G \equiv 1$), so
Proposition~\ref{prop:LSD} with $J = 1$ gives
\[
S(x, 0) = \sum_{n \le x} d_z(n)
= \frac{x (\log x)^{z-1}}{\Gamma(z)}
+ O\big( x (\log x)^{z - 2} \big).
\]
Since $z \in (-1, 1) \setminus \{0\}$ is not a pole of $\Gamma$, and
$\Gamma$ has no zeros, $\Gamma(z)$ is finite and nonzero (for
$z = \pm\tfrac12$: $\Gamma(\tfrac12) = \sqrt\pi$ and
$\Gamma(-\tfrac12) = -2\sqrt\pi$); hence there is
$x_2 = x_2(z)$ such that
$| S(x, 0) | \ge \tfrac{1}{2} |\Gamma(z)|^{-1} x (\log x)^{z-1}$ for
$x \ge x_2$. Combining the two ranges,
$\sup_\alpha | S(x, \alpha) | \gg x (\log x)^{z-1}$ for all
$x \ge 3$, with an effective constant. \qed

\subsubsection*{Proof of Theorem~\ref{thm:sup}(ii): the upper bound}

Set
\[
A := 2, \qquad B := 2A + 2C = 4 + 2C, \qquad
J := \lceil 2B + 2 \rceil, \qquad Q := \frac{x}{(\log x)^{B}} .
\]
We may assume $x \ge x_1(z)$, a threshold to be accumulated below;
for $3 \le x < x_1$ the trivial bound $|S(x,\alpha)| \le x
\ll_{x_1} x(\log x)^{z-1}$ suffices. In particular assume $x$ large
enough that $Q \ge (\log x)^{B}$.

Let $\alpha \in \R$ be arbitrary. By Lemma~\ref{lem:dirichlet} there
is a reduced fraction $r/q$ with $q \le Q$ and
$|\alpha - r/q| \le 1/(qQ) \le 1/q^{2}$. We distinguish two cases,
which together cover every $\alpha$.

\emph{Case 1: $(\log x)^{B} < q \le Q$.} The hypotheses of
Corollary~\ref{cor:minor} hold with $A = 2$ (note
$B = 2A + 2C$ and $Q = x(\log x)^{-B}$ exactly match its range of
$q$), so, for $x$ beyond the threshold of that proposition,
\[
| S(x, \alpha) | \ll \frac{x}{(\log x)^{2}}
\le x (\log x)^{z - 1},
\]
the last step because $z - 1 > -2$ (as $z > -1$) and $\log x \ge 1$.

\emph{Case 2: $q \le (\log x)^{B}$.} Then
$\beta := \alpha - r/q$ satisfies
$|\beta| \le 1/(qQ) \le 1/Q = (\log x)^{B}/x$, so
Theorem~\ref{thm:major} applies with the parameters $B$ and $J$
fixed above:
\[
S(x, \alpha)
= \sum_{j = 0}^{J-1} \lambda_j(z, q)
\int_{2}^{x} (\log t)^{z - 1 - j}\e(\beta t)\, dt
+ O\Big( \frac{x}{(\log x)^{J + 1 - z - 2B}} \Big).
\]
The error term: since $J \ge 2B + 2$, its exponent satisfies
$J + 1 - z - 2B \ge 3 - z \ge 1 - z$, so the error is
$O( x (\log x)^{z - 1} )$ (with two powers of $\log x$ to spare).
The main terms: for each $j$, the exponent $w = z - 1 - j$ is
negative, so by the triangle inequality and
Lemma~\ref{lem:log-integral},
\[
\bigg| \int_{2}^{x} (\log t)^{z-1-j} \e(\beta t)\, dt \bigg|
\le \int_{2}^{x} (\log t)^{z-1-j}\, dt
\ll_{J} x (\log x)^{z - 1 - j} + \sqrt x
\ll_{J} x (\log x)^{z - 1},
\]
where in the last step we used $(\log x)^{-j} \le 1$ and
$\sqrt x \le x (\log x)^{z-1}$ for $x \ge x_1$ (as
$(\log x)^{1 - z} \le (\log x)^{2} \le \sqrt x$ eventually). Since
$|\lambda_j(z, q)| \ll_{J, \varepsilon} q^{-1 + \varepsilon} \ll_{J} 1$
by \eqref{eq:lambda-props}, summing over the $J = O(1)$ values of
$j$ gives $| S(x, \alpha) | \ll x (\log x)^{z - 1}$ in this case as
well.

Taking the supremum over $\alpha$ completes the proof of
Theorem~\ref{thm:sup}. \qed

\begin{remark}\label{rem:where-ineff}
Part (i) of Theorem~\ref{thm:sup} is entirely effective, as is the
lower bound in part (ii). In the upper bound of part (ii), the
minor-arc input (Corollary~\ref{cor:minor}) is effective: apart from the type~II estimate quoted from \cite{K},
whose proof there is likewise effective, the proof of
Theorem~\ref{thm:main} in Sections
\ref{sec:identity}--\ref{sec:main-proof} is elementary; the ineffective constant
enters only through Theorem~\ref{thm:major} in Case~2, i.e.,
ultimately through Siegel's theorem in
Proposition~\ref{prop:SW}, exactly as in Davenport's theorem.
\end{remark}

\begin{remark}
\label{rem:not-at-zero}
In part (ii) the supremum need not sit at $\alpha = 0$. In contrast with part (i), where the supremum is attained at
$\alpha = 0$ exactly, the proof of part (ii) shows only that
$\sup_\alpha |S(x, \alpha)|$ is comparable to
$|S(x, 0)| \sim x(\log x)^{z-1}/|\Gamma(z)|$; it need not be
asymptotic to it: the major-arc main term at $r/q$ has leading
coefficient $\g_z(q)/\Gamma(z)$, and $|\g_z(q)|$ can exceed
$\g_z(1) = 1$. For instance, by \eqref{eq:gamma-p},
$\g_{-1/2}(2) = 1 - 2^{3/2} = -1.828\dots$, so the sum is
asymptotically \emph{larger} at $\alpha$ near $1/2$ than at $\alpha$
near $0$ when $z = -\tfrac12$. Determining the exact asymptotic
constant of the supremum in part (ii) would require maximizing
$|\g_z(q)| \cdot |{\textstyle\int_2^x} (\log t)^{z-1} \e(\beta t) dt|
/ (x (\log x)^{z-1})$ over $q$ and $\beta$; we do not pursue this.
\end{remark}


\subsection{Moments: proof of Theorem \ref{thm:scalar}}\label{subsec:moments-proof}

Throughout this subsection, $z = \pm a/b$ is fixed, $s > 2$ is a
fixed real number, $L = \log X$, and
$\theta = s(z - 1) < 0$ is the logarithmic exponent of
Theorem~\ref{thm:scalar}; $C = C(a, b)$ denotes the constant of
Theorem~\ref{thm:main}. Fix also, once and for all, $\varepsilon$
with $0 < \varepsilon < (s-2)/(2s)$, as required by
Lemma~\ref{lem:singular-series}. Let $B > 0$ be a parameter, fixed at
\eqref{eq:BJ-choice} below, and set
\begin{equation}\label{eq:P-def}
P = (\log X)^{B}.
\end{equation}
For $1 \le q \le P$ and $(r, q) = 1$ put
\[
\fM(q, r) = \Big\{ \alpha \in [0, 1] :
\Big| \alpha - \frac{r}{q} \Big| \le \frac{P}{X} \Big\},
\qquad
\fM = \bigcup_{q \le P} \bigcup_{\substack{r = 1 \\ (r, q) = 1}}^{q}
\fM(q, r),
\qquad
\fm = [0, 1] \setminus \fM,
\]
with the usual convention that $\fM(1, 1)$ is the interval
$[0, P/X] \cup [1 - P/X, 1]$ regarded modulo $1$; by the periodicity of $S_z$, this identifies
the two end intervals with the single arc $|\beta| \le P/X$ around
the integer point. Two distinct
fractions $r/q \ne r'/q'$ with $q, q' \le P$ differ by at least
$1/(qq') \ge P^{-2}$, while each arc has length $2P/X$; since
$2P^{3} < X$ for $X$ large, the arcs $\fM(q, r)$ are pairwise
disjoint.

\begin{proposition}
\label{prop:minor-arcs}
For every $B > 0$,
\begin{equation}\label{eq:minor-arcs}
\int_{\fm} |S_{z}(\alpha)|^{s}\, d\alpha
\ll_{s, z, B}
X^{s-1} (\log X)^{(s-2)(C - B/2)} .
\end{equation}
\end{proposition}

\begin{proof}
Let $\alpha \in \fm$. By Dirichlet's approximation theorem
(Lemma~\ref{lem:dirichlet}) with parameter $Q = X/P$ there exist
$q \le Q$ and $(r, q) = 1$ with
\begin{equation}\label{eq:dirichlet-moments}
\Big| \alpha - \frac{r}{q} \Big| \le \frac{1}{q Q}
= \frac{P}{q X} \le \frac{1}{q^{2}}.
\end{equation}
If $q \le P$, then \eqref{eq:dirichlet-moments} gives
$|\alpha - r/q| \le P/X$, so $\alpha \in \fM(q, r) \subseteq \fM$
(after reducing $r/q$ modulo $1$ into $[0,1]$), contradicting
$\alpha \in \fm$. Hence $P < q \le X/P$, and
Theorem~\ref{thm:main} applies, its hypothesis
$|\alpha - r/q| \le 1/q^2$ holding by
\eqref{eq:dirichlet-moments}:
\[
|S_{z}(\alpha)|
\ll
\Big( \frac{X}{\sqrt{q}} + X^{4/5} + \sqrt{Xq} \Big)
(\log X)^{C}
\ll
X (\log X)^{C - B/2},
\]
where we used $q > P$ in the first term, $q \le X/P$ in the third
(so that $\sqrt{Xq} \le X/\sqrt P$), and
$X^{4/5} \le X (\log X)^{-B/2}$ for $X$ large. Consequently
\[
\int_{\fm} |S_z(\alpha)|^{s}\, d\alpha
\le
\Big( \sup_{\alpha \in \fm} |S_z(\alpha)| \Big)^{s-2}
\int_0^1 |S_z(\alpha)|^{2}\, d\alpha ,
\]
and by Parseval and Lemma~\ref{lem:dz}(b),
\[
\int_0^1 |S_z(\alpha)|^{2}\, d\alpha
= \sum_{n \le X} d_z(n)^{2} \le X .
\]
Combining the three displays proves \eqref{eq:minor-arcs}.
\end{proof}

\begin{remark}\label{rem:B-choice}
Since $s > 2$ is fixed, the exponent in \eqref{eq:minor-arcs} can be
made at most $\theta - 1$ by choosing
$B \ge 2 C + 2 \big( s (1 - z) + 1 \big)/(s - 2)$; the minor arcs
will therefore be negligible compared with the main term of
Theorem~\ref{thm:scalar} once $B$ is fixed sufficiently large in
terms of $s$ and $z$. This is the only point in the argument at
which the hypothesis $s > 2$ is used on the minor arcs; the same
hypothesis reappears independently in the convergence of the
singular series (Lemma~\ref{lem:singular-series}).
\end{remark}

We now fix the parameters, each depending only on those preceding
it: the data $s, z$ are given, $C = C(a,b)$ is as above, and then
\begin{equation}\label{eq:BJ-choice}
B = 2 C + \frac{2\big( s (1 - z) + 1 \big)}{s - 2},
\qquad
J_1 = \big\lceil 4B + 4s + 8 \big\rceil,
\end{equation}
and all implied constants may depend on $s, z, B, J_1$ and
$\varepsilon$. Note that $B(s-2)/2 = C(s-2) + s(1-z) + 1 \ge 1$, so
in particular $B \ge 2/(s-2)$, the standing hypothesis of
Lemma~\ref{lem:singular-series}.

\subsubsection*{Replacing the sum by its major-arc model}

For $j \ge 0$ write
\begin{equation}\label{eq:Wj}
W_j(\beta) = \int_2^X (\log t)^{z - 1 - j}\e(\beta t)\, dt
= X (\log X)^{z-1-j}\, \Phi_{z-1-j}(\beta X),
\end{equation}
the second equality being the rescaling recorded after
\eqref{eq:Phi-def}. Fix $q \le P$ and a reduced residue $r$ modulo
$q$, and write, for $|\beta| \le P/X$,
\begin{equation}\label{eq:Tq-def}
T_q(\beta) =
\sum_{j=0}^{J_1 - 1} \lambda_j(z, q)\, W_j(\beta).
\end{equation}
By Theorem~\ref{thm:major} (applied with the depth $J = J_1$ and
the same $B$), $T_q$ does not depend on $r$ and satisfies,
uniformly on the arc,
\begin{equation}\label{eq:Sz-Tq}
S_{z}\Big( \frac{r}{q} + \beta \Big)
= T_q(\beta) + O\big( E \big),
\qquad
E := X L^{-A_1},
\quad
A_1 := J_1 - 2B - 1,
\end{equation}
since $J_1 + 1 - z - 2B \ge J_1 - 2B - 1 = A_1$. By
\eqref{eq:BJ-choice},
\begin{equation}\label{eq:A1-lower}
A_1 \ge 2B + 4s + 7
\ge 2B + (s - 2)(z - 1) - \theta + 1,
\end{equation}
the last step because $(s-2)(z - 1) \le 0$ and
$|\theta| \le s |z - 1| < 2s$.

\begin{lemma}\label{lem:model-replacement}
With the choices \eqref{eq:BJ-choice},
\begin{equation}\label{eq:model-replacement}
\int_{\fM} |S_{z}(\alpha)|^{s}\, d\alpha
=
\sum_{q \le P} \varphi(q)
\int_{|\beta| \le P/X} |T_q(\beta)|^{s}\, d\beta
\;+\;
O\big( X^{s-1} L^{\theta - 1} \big).
\end{equation}
\end{lemma}

\begin{proof}
By \eqref{eq:lambda-props} and Lemma~\ref{lem:Wj-bounds}, using
$(\log X)^{z - 1 - j} \le (\log X)^{z - 1}$ for $j \ge 0$, the
model satisfies
\begin{equation}\label{eq:Tq-bound}
|T_q(\beta)|
\ll
q^{-1+\varepsilon}\, m(\beta),
\qquad
m(\beta) := \min\Big( X L^{z-1},\ \frac{1}{|\beta|} \Big).
\end{equation}
Here $m(0) := X L^{z-1}$.
By Lemma~\ref{lem:perturbation} (applied with our fixed $s > 2$), together with
\eqref{eq:Sz-Tq} and \eqref{eq:Tq-bound},
\[
\big| |S_z(\tfrac rq + \beta)|^{s} - |T_q(\beta)|^{s} \big|
\ll_s
E \big( q^{-1+\varepsilon} m(\beta) + E \big)^{s - 1}
\ll
E\, q^{(-1+\varepsilon)(s-1)} m(\beta)^{s-1} + E^{s},
\]
using $(x+y)^{s-1} \ll_s x^{s-1} + y^{s-1}$. Since $s - 1 > 1$,
splitting at $|\beta| = (X L^{z - 1})^{-1}$ gives
\[
\int_{|\beta| \le P/X} m(\beta)^{s-1}\, d\beta
\ll (X L^{z - 1})^{s-2},
\]
so the error in replacing the integrand,
summed over the $\varphi(q)$ residues and over $q \le P$, totals
\begin{align*}
&\ll\;
P^{1+\varepsilon}\, E\, X^{s-2} L^{(z-1)(s-2)}
+ P^{3}\, E^{s}\, X^{-1} \\
&\ll\;
X^{s-1} L^{(z - 1)(s-2) - A_1 + 2B}
+ X^{s-1} L^{-s A_1 + 3B}
\ll
X^{s-1} L^{\theta - 1},
\end{align*}
by \eqref{eq:A1-lower}, using also the crude bound
$\sum_{q \le P} \varphi(q)\, q^{(-1+\varepsilon)(s-1)} \ll
P^{1+\varepsilon}$. Since the arcs are disjoint and the model does
not depend on $r$, the sum over $r$ of the model integrals produces
exactly the factor $\varphi(q)$.
\end{proof}

\subsubsection*{The model integral}

\begin{lemma}\label{lem:model-integral}
For each $q \le P$,
\begin{equation}\label{eq:model-integral}
\int_{|\beta| \le P/X} |T_q(\beta)|^{s}\, d\beta
=
X^{s-1} L^{\theta}
\Big(
|\lambda_0(z, q)|^{s}\, A_s
\;+\;
O\big( q^{(-1+\varepsilon)s}\, L^{-1} \big)
\Big).
\end{equation}
\end{lemma}

\begin{proof}
Substituting $\beta = u/X$ and using \eqref{eq:Wj},
\[
\int_{|\beta| \le P/X} |T_q|^{s}\, d\beta
=
X^{s-1} L^{\theta}
\int_{|u| \le P} | G(u) |^{s}\, du,
\qquad
G := \lambda_0 K_0 + \Delta,
\]
where $\lambda_0 = \lambda_0(z, q)$ and
\[
\Delta(u)
=
\lambda_0 \big( \Phi_{z-1}(u) - K_0(u) \big)
+ \sum_{j=1}^{J_1 - 1} \lambda_j(z, q)\, L^{-j}
\Big( K_0(u) + \big( \Phi_{z-1-j}(u) - K_0(u) \big) \Big).
\]
By Lemma~\ref{lem:kernel-approx}, applied with $A = B$ to each
rescaled kernel $\Phi_{z-1-j}$ (the exponents $z - 1 - j$ are all
negative, and the variable satisfies $|u| \le P = L^{B}$
throughout), by $|K_0(u)| \le \min(1, 1/|u|)$, and by
\eqref{eq:lambda-props}, using $L^{-j} \le L^{-1}$ for $j \ge 1$,
\begin{equation}\label{eq:Delta-bound}
|\Delta(u)|
\ll
q^{-1+\varepsilon}\, h(u),
\qquad
h(u) := \min\Big( \frac{1}{L},\ \frac{1}{|u| L^{1/2}} \Big),
\end{equation}
with the convention $h(0) := 1/L$; here we absorbed the term
$L^{-1} \min(1, 1/|u|) \le h(u)$ into $h$: for $|u| \le L^{1/2}$ one has
$L^{-1} \min(1, 1/|u|) \le L^{-1}$, while for $|u| > L^{1/2}$ one
has $L^{-1}/|u| \le 1/(|u| L^{1/2})$. Also $|G(u)|$ and
$|\lambda_0 K_0(u)|$ are both
$\ll q^{-1+\varepsilon} \min(1, 1/|u|)$, since
$h(u) \le \min(1, 1/|u|)$. By Lemma~\ref{lem:perturbation},
\[
\big| |G|^{s} - |\lambda_0 K_0|^{s} \big|
\ll
\big( |\lambda_0 K_0|^{s - 1} + |\Delta|^{s - 1} \big) |\Delta|
\ll
q^{(-1+\varepsilon) s} \min\Big(1, \frac{1}{|u|}\Big)^{s - 1} h(u),
\]
so that
\[
\int_{|u| \le P}
\big| |G|^{s} - |\lambda_0 K_0|^{s} \big|\, du
\ll
q^{(-1+\varepsilon)s}
\int_{\R} \min\Big(1, \frac{1}{|u|}\Big)^{s-1}\, h(u)\, du
\ll
\frac{q^{(-1+\varepsilon)s}}{L},
\]
where the last integral was bounded by splitting at $|u| = 1$ and
$|u| = L^{1/2}$:
\[
\int_{|u| \le 1} \frac{du}{L}
+ \int_{1 \le |u| \le L^{1/2}} \frac{du}{|u|^{s-1} L}
+ \int_{|u| \ge L^{1/2}} \frac{du}{|u|^{s} L^{1/2}}
\ll
\frac{1}{L} + \frac{1}{L} + L^{-s/2}
\ll \frac{1}{L},
\]
using $s > 2$ in the middle integral (convergent at infinity) and
$L^{(1-s)/2}/L^{1/2} = L^{-s/2} \le L^{-1}$ in the third. Finally,
by Lemma~\ref{lem:kernel-moments},
\[
\int_{|u| \le P} |\lambda_0 K_0(u)|^{s}\, du
=
|\lambda_0|^{s}
\Big( A_s + O\big( P^{1-s} \big) \Big)
=
|\lambda_0|^{s} A_s
+ O\big( q^{(-1+\varepsilon)s} L^{-B(s-1)} \big),
\]
and $B(s-1) \ge 1$. Collecting the three displays proves
\eqref{eq:model-integral}.
\end{proof}

\subsubsection*{Assembly}

\begin{proof}[Proof of Theorem~\ref{thm:scalar}]
Write
$\gamma_0 = A_s\, \mathfrak G_{s}(z) / |\Gamma(z)|^{s}$, which is
the leading constant of the theorem. Combining
Lemmas~\ref{lem:model-replacement}, \ref{lem:model-integral}
and~\ref{lem:singular-series} (whose standing hypothesis
$B \ge 2/(s-2)$ holds by \eqref{eq:BJ-choice}),
\[
\int_{\fM} |S_{z}|^{s}\, d\alpha
=
X^{s-1} L^{\theta}
\Big(
A_s\, \frac{\mathfrak G_{s}(z)}{|\Gamma(z)|^{s}}
+ O\big( L^{-1} \big)
\Big)
=
\gamma_0\, X^{s-1} L^{\theta}
\Big( 1 + O\big( L^{-1} \big) \Big),
\]
where the $O(L^{-1})$ inside the first parenthesis collects: the
error of Lemma~\ref{lem:model-replacement} (relative size
$L^{-1}$); the correction of Lemma~\ref{lem:model-integral} summed
against $\sum_q \varphi(q)\, q^{(-1+\varepsilon)s} \ll 1$ (relative
size $L^{-1}$); and the singular-series tail (relative size
$L^{-1}$); and where we used
$\gamma_0 \ge A_s |\Gamma(z)|^{-s} > 0$ to convert the additive
error into a relative one. By Proposition~\ref{prop:minor-arcs}
and the choice of $B$ in \eqref{eq:BJ-choice}, which gives
$(s-2)(C - B/2) = -\big( s(1-z) + 1 \big) = \theta - 1$,
\[
\int_{\fm} |S_{z}|^{s}\, d\alpha
\ll
X^{s-1} L^{\theta - 1},
\]
which is also of relative size $L^{-1}$. Adding the two ranges
proves the asymptotic of Theorem~\ref{thm:scalar}. The constants
are ineffective because those of Theorem~\ref{thm:major} are;
every other step is effective.
\end{proof}

\section{Conclusion}\label{sec:conclusion}

We close by recording why the proof is arranged in this order, and
what it does not do.

The negative rational case is the basic case because
\[
\big( \zeta^{-a/b} \big)^{b}\, \zeta^{a} = 1
\]
is an integer-power relation. That relation produces the finite
identity of Proposition~\ref{prop:negative-identity} directly. The
endpoint $\mu = d_{-1}$ is the same mechanism with $a = b = 1$, and
gives the particularly transparent coefficients $5, -10, 10, -5, 1$
of \eqref{eq:mu-identity}. The positive rational case is then best
obtained from
\[
\zeta^{a/b} = \zeta \cdot \zeta^{-(b-a)/b},
\]
because this keeps every unrestricted long factor as a positive
power of $\zeta$. Thus every unrestricted long variable carries coefficient $1$:
essential in the type~I branch, and harmless in the type~II branch,
where the grouped coefficients are controlled in mean square.

The exponent $4/5$ comes from the same numerical balance as
Vinogradov's estimate for $\Lambda$. In the type~II case one side
is grouped into the interval $x^{2/5} < M \le x^{3/5}$, and the two
middle terms of the bilinear estimate are
\[
x M^{-1/2} \le x^{4/5},
\qquad
x^{1/2} M^{1/2} \le x^{4/5} .
\]
If no such grouping is available, the window-or-long-variable
lemma produces a long $\zeta$-variable, and the remaining variables
have product at most $x^{7/10}$ up to a constant depending only on
the number of variables. The type~I branch is then handled by the
two-range argument: the mean-square type~I estimate outside the middle range
$x^{3/10} < q < x^{7/10}$, and the pointwise $x^{o(1)}$ estimate only
inside it. This is precisely where the false logarithmic-loss
divisor-weighted type~I estimate of
Remark~\ref{rem:no-false-type-I} is avoided; the size of the admissible type~II window is what determines the
exponent, as discussed in the introduction.

Regarding effectivity, the proof gives the following precise picture:
every constant in Theorems~\ref{thm:main} and~\ref{thm:mu} and
Corollary~\ref{cor:minor} is effective. In the applications, the only
ineffectivity enters through Siegel's theorem in the proof of the
Siegel-Walfisz estimate for \(d_z\chi\), Proposition~\ref{prop:SW},
proved in Appendix~\ref{app:SW}; see Remark~\ref{rem:where-ineff}.

Three directions are left open. First, the minor-arc results and the identities behind them concern
rational exponents -- the major-arc analysis of
Section~\ref{sec:applications} already applies to every real
$0 < |z| < 1$ -- and the restriction is structural: the
identity rests on $b$-th roots in the algebra of arithmetic
functions and on the polynomial factorization of $X^b - Y^b$, so
real $z \in (-1, 1)$ is not reached by taking limits. An extension to real $z$ would, within the present approach,
require uniformity of Theorem~\ref{thm:main} in
$a$ and $b$ -- the constant $C(a, b)$ and the coefficient sizes
$\lambda_{b, m}$ are the obstructions -- together with rational
approximation of the exponent, or a genuinely different
construction. Second, the moment asymptotics of
Theorem~\ref{thm:scalar} are the harmonic-analysis input for
additive representation problems of partition type with the
weights $d_z$; carrying these out requires the corresponding
singular-series analysis, which we do not undertake here. Third, the hypothesis $s > 2$ in Theorem~\ref{thm:scalar} is used, beyond the threshold $s > 1$, at three points: the
minor-arc H\"older-Parseval estimate, the perturbative estimate for
the major-arc model integral, and the convergence of the singular
series. The range $1 < s \le 2$, where these estimates no longer
yield the same logarithmic saving, appears to require a different method; for
$0 < s \le 1$ the kernel and perturbation estimates change character
as well (the integral defining $A_s$ diverges at $s = 1$).

\section*{Acknowledgments}

The author wishes to thank Ayla Gafni for jointly working out the
details of Exercise~9 of \cite[Section~17.2.1]{MV}. The author is
grateful to Anji Dong and G\'erald Tenenbaum for reading parts of the manuscript
and providing valuable feedback. The author thanks Dimitris
Koukoulopoulos for a helpful suggestion about future research directions.
Moreover, the author also wishes to thank Kunjakanan Nath for fruitful discussions,
valuable feedback, and for bringing the announced result
\cite{MPR} to the author's attention.
Anthropic's Claude Opus 4.8 was used for proof development (with
emphasis on \S\S2.3, 2.4, 4.4, 5.3 and Appendix \ref{app:SW}),
exposition, and revision.


\appendix

\section{Proof of Proposition~\ref{prop:SW}}
\label{app:SW}

This appendix proves Proposition~\ref{prop:SW} by the classical
contour method, from the following standard inputs: the zero-free
region for Dirichlet $L$-functions \cite[\S 14]{Da}, Siegel's theorem
\cite[\S 21]{Da}, standard bounds for $L(s,\chi)^{\pm1}$ in a
zero-free region \cite[Theorem~11.4]{MVI}, the truncated Perron formula
\cite[Theorem~5.2 and Corollary~5.3]{MVI}, and Mertens' theorem.

\begin{proof}
Put $T = \exp\big( \sqrt{\log x} \big)$. We may assume $x$
sufficiently large in terms of $z$, $A$, $B$, smaller values being
absorbed into the implied constant. Let $\chi^*$ be the primitive
character of conductor $q^* \mid q$ inducing $\chi$; since $\chi$ is
non-principal, so is $\chi^*$.

For $\real s > 1$, the Euler product \eqref{eq:twist-series} factors
as
\begin{equation}\label{eq:app-factor}
L(s, \chi)^{z}
= L(s, \chi^*)^{z}
\prod_{\substack{p \mid q \\ p \nmid q^*}}
\Big( 1 - \frac{\chi^*(p)}{p^{s}} \Big)^{z},
\end{equation}
the finite product removing the Euler factors present for $\chi^*$
but absent for $\chi$; all powers are defined by the branches of the
absolutely convergent Euler products in $\real s > 1$.

By the classical zero-free region together with Siegel's theorem,
there is $c = c(B) > 0$, ineffective, such that
$L(s, \chi^*) \ne 0$ on the rectangle
\[
\mathcal R
= \Big\{ \sigma + it :\
\sigma \ge 1 - \frac{2c}{\sqrt{\log x}},\
|t| \le T + 1 \Big\} .
\]
Indeed, $q^* \le (\log x)^{B}$ and $|t| \le T + 1$ give
$\log\big( q^* (|t| + 3) \big) \ll_B \sqrt{\log x}$, so the zero-free
region \cite[\S 14]{Da} covers $\mathcal R$ except possibly for an
exceptional real zero $\beta_1$ of a real character. Siegel's theorem
\cite[\S 21]{Da} gives
$1 - \beta_1 \gg_{\eta} q^{-\eta} \ge (\log x)^{-\eta B}$ for every
fixed $\eta > 0$, and any choice $\eta < 1/(2B)$ places $\beta_1$
outside $\mathcal R$ for all large $x$, after decreasing $c$ if
necessary; the ineffectivity of the implied constant originates here.

Since $\mathcal R$ is simply connected and $L(s, \chi^*)$ has neither
zero nor pole there ($\chi^*$ being non-principal), the branch of
$\log L(s, \chi^*)$ from $\real s > 1$ continues to $\mathcal R$, and
we set $L(s, \chi^*)^{z} = \exp\{ z \log L(s, \chi^*) \}$ there. For
$\sigma > 0$ and $p \ge 2$ we have $|\chi^*(p)\, p^{-s}| < 1$, so
each finite factor in \eqref{eq:app-factor} continues holomorphically
as well, and $L(s, \chi)^{z}$ is holomorphic on $\mathcal R$.

On $\mathcal R$ we bound $L(s, \chi^*)^{\pm 1}$ via
\cite[Theorem~11.4]{MVI}. After shrinking $c$ once more if
necessary, $\mathcal R$ lies in the region
$\sigma \ge 1 - c_0/(2 \log(q^* (|t| + 3)))$ of that theorem, $c_0$
being the constant of the zero-free region. If $L(s, \chi^*)$ has no
exceptional zero, or if $|s - \beta_1| \ge 1/\log q^*$, then
estimates (11.6)--(11.7) there give
$|\log L(s, \chi^*)| \le \log\log( q^* (|t| + 3) ) + O(1)$, whence
$L(s, \chi^*)^{\pm 1} \ll (\log x)^{1/2}$ on $\mathcal R$ for all
large $x$. If instead $|s - \beta_1| < 1/\log q^*$, then estimate
(11.10) there gives
$|s - \beta_1| \ll |L(s, \chi^*)| \ll |s - \beta_1| (\log q^*)^{2}$,
while Siegel's theorem as above yields, on $\mathcal R$,
$\sigma - \beta_1 \ge c_1(\eta) (\log x)^{-\eta B}
- 2c (\log x)^{-1/2} \gg (\log x)^{-\eta B}$
for all large $x$, since $\eta B < 1/2$; hence
$L(s, \chi^*)^{\pm 1} \ll_{B, \eta} (\log x)^{1/2}$ in this regime
as well, with an ineffective threshold. In either case
$L(s, \chi^*)^{\pm 1} \ll_B (\log x)^{C_B}$; since $z$ is real,
$| L(s, \chi^*)^{z} | = | L(s, \chi^*) |^{z}
\ll_{z, B} (\log x)^{C_{z,B}}$. For the finite product, write $\delta = 2c/\sqrt{\log x}$, so that
$\sigma \ge 1 - \delta$ on $\mathcal R$, and recall that every prime
$p \mid q$ satisfies $p \le (\log x)^{B}$. For $z > 0$, each factor
has modulus at most $2^{z}$, so the product is
$\ll_z 2^{z\, \omega(q)} \ll_{z, B} (\log x)^{C_{z, B}}$. For
$z < 0$, we use
$| 1 - \chi^*(p)\, p^{-s} | \ge 1 - p^{-\sigma}$ and estimate the
inverse product logarithmically: for $x$ large enough that
$\sigma > 3/4$,
\[
\log \prod_{p \mid q} \big( 1 - p^{-\sigma} \big)^{-1}
\ll \sum_{p \mid q} p^{-\sigma} + O(1),
\]
the $O(1)$ accounting for the higher-order terms via
$\sum_p p^{-3/2} < \infty$. Since
$p^{\delta} = 1 + O( \delta \log p )$ uniformly for
$p \le (\log x)^{B}$, and $\sum_{p \mid q} \log p \le \log q$,
\[
\sum_{p \mid q} p^{-\sigma}
\le \sum_{p \mid q} p^{-1 + \delta}
\le \sum_{p \mid q} \frac{1}{p}
+ O\Big( \delta \sum_{p \mid q} \frac{\log p}{p} \Big)
\ll_B \log\log\log x + 1,
\]
using
$\sum_{p \mid q} 1/p
\le \sum_{p \le 2 \log 3q} 1/p + \omega(q)/(2 \log 3q)
\ll \log\log 3q \ll_B \log\log\log x + 1$
by Mertens' estimate together with $\omega(q) \le (\log q)/\log 2$,
and
$\delta \sum_{p \mid q} (\log p)/p \le \delta \log 3q
\ll_B (\log\log x)/\sqrt{\log x} \ll 1$. Hence
$\prod_{p \mid q} \big( 1 - p^{-\sigma} \big)^{-|z|}
\ll_{z, B} (\log\log x)^{C_{z, B}}$, and in all cases the finite
Euler factors in \eqref{eq:app-factor} contribute at most a fixed
power of $\log x$. Hence
\begin{equation}\label{eq:app-Fbound}
L(s, \chi)^{z} \ll_{z, B} (\log x)^{C_{z,B}}
\qquad (s \in \mathcal R).
\end{equation}

Let $\kappa = 1 + 1/\log x$. Since $|d_z(n) \chi(n)| \le 1$
(Lemma~\ref{lem:dz}(b)), the truncated Perron formula
\cite[Corollary~5.3]{MVI} gives
\begin{equation}\label{eq:app-perron}
\sum_{n \le x} d_z(n) \chi(n)
= \frac{1}{2\pi i}
\int_{\kappa - iT}^{\kappa + iT} L(s, \chi)^{z}\,
\frac{x^{s}}{s}\, ds
+ O\Big( \frac{x (\log x)^{2}}{T} \Big),
\end{equation}
and the error term is $\ll_A x (\log x)^{-A}$ by the choice of $T$.
Deform the contour to the boundary of the rectangle with vertical
sides $\real s = \kappa$ and
$\real s = \sigma_0 := 1 - c/\sqrt{\log x}$ and horizontal sides
$|t| = T$; no singularity is crossed, by the holomorphy of
$L(s, \chi)^{z}$ on $\mathcal R$. By \eqref{eq:app-Fbound}, the left
vertical side contributes
\[
\ll_{z, B}\
x^{\sigma_0} (\log x)^{C}
\int_{-T}^{T} \frac{dt}{1 + |t|}
\ \ll_{z, B}\
x \exp( -c \sqrt{\log x}) (\log x)^{C + 1}
\ll_{z, A, B} x (\log x)^{-A},
\]
and the horizontal sides contribute
$\ll_{z, B} x (\log x)^{C} / T \ll_{z, A, B} x (\log x)^{-A}$.
Combining with \eqref{eq:app-perron} proves the proposition; the
implied constant is ineffective exactly through $c(B)$, that is,
through Siegel's theorem.
\end{proof}

\end{document}